\documentclass[11pt]{amsart}

\usepackage{amssymb,verbatim,hyperref,cite}
\usepackage[margin=1in]{geometry}
\usepackage[all,cmtip]{xy}

\newcommand{\cbinom}[2]{\genfrac{\{}{\}}{0pt}{}{#1}{#2}}

\newcommand{\subgrp}[1]{\langle #1 \rangle}

\newcommand{\set}[1]{\left\{ #1 \right\}}

\newcommand{\wt}[1]{\widetilde{ #1}}
\newcommand{\wh}[1]{\widehat{ #1 }}
\newcommand{\ol}[1]{\overline{#1}}

\newcommand{\ve}{\varepsilon}

\DeclareMathOperator{\chr}{char}

\DeclareMathOperator{\Dist}{Dist}
\DeclareMathOperator{\End}{End}
\DeclareMathOperator{\ev}{ev}
\DeclareMathOperator{\Ext}{Ext}

\DeclareMathOperator{\gr}{gr}
\DeclareMathOperator{\opH}{H}
\DeclareMathOperator{\Hom}{Hom}

\DeclareMathOperator{\id}{id}
\DeclareMathOperator{\im}{im}
\DeclareMathOperator{\ind}{ind}
\DeclareMathOperator{\Lie}{Lie}
\DeclareMathOperator{\Mor}{Mor}
\DeclareMathOperator{\Nil}{Nil}

\DeclareMathOperator{\res}{res}

\renewcommand{\mod}{\, \textup{mod}\, }

\newcommand{\Cbul}{C^\bullet}

\newcommand{\Hbul}{\opH^\bullet}
\newcommand{\Hev}{\opH^{\textup{ev}}}
\newcommand{\Hodd}{\opH^{\textup{odd}}}

\newcommand{\A}{\mathbb{A}}

\newcommand{\F}{\mathbb{F}}

\newcommand{\N}{\mathbb{N}}

\newcommand{\Z}{\mathbb{Z}}

\newcommand{\Fp}{\F_p}

\newcommand{\g}{\mathfrak{g}}
\newcommand{\gl}{\mathfrak{gl}}

\newcommand{\Bbar}{\ol{B}}
\newcommand{\Nbar}{\ol{N}}
\newcommand{\Ybar}{\ol{Y}}
\newcommand{\Xbar}{\ol{X}}
\newcommand{\betabar}{\ol{\beta}}

\newcommand{\fone}{f_{\ol{1}}}
\newcommand{\gone}{\g_{\ol{1}}}
\newcommand{\glone}{\gl(m|n)_{\ol{1}}}
\newcommand{\Lone}{L_{\ol{1}}}
\newcommand{\fzero}{f_{\ol{0}}}
\newcommand{\gzero}{\g_{\ol{0}}}
\newcommand{\glzero}{\gl(m|n)_{\ol{0}}}
\newcommand{\Lzero}{L_{\ol{0}}}
\newcommand{\Vone}{V_{\ol{1}}}
\newcommand{\Vzero}{V_{\ol{0}}}

\newcommand{\alg}{\mathfrak{alg}}

\newcommand{\svec}{\textup{svec}}
\newcommand{\smod}{\textup{smod}}
\newcommand{\csalg}{\mathfrak{csalg}}
\newcommand{\scomod}{\textup{scomod}}
\newcommand{\salg}{\mathfrak{salg}}
\newcommand{\fsvec}{\mathfrak{svec}}
\newcommand{\fsmod}{\mathfrak{smod}}
\newcommand{\fscomod}{\mathfrak{scomod}}

\newcommand{\lab}{L_{ab}}
\newcommand{\vlab}{V(\lab)}
\newcommand{\Gev}{G_{\textrm{ev}}}
\newcommand{\Gred}{G_{\textrm{red}}}
\newcommand{\Iev}{I_{\ev}}
\newcommand{\Sspk}{SSp_k}

\newcommand{\Ug}{U(\g)}
\newcommand{\Vg}{V(\g)}

\numberwithin{equation}{subsection}

\newtheorem{theorem}{Theorem}[subsection]
\newtheorem{proposition}[theorem]{Proposition}
\newtheorem{conjecture}[theorem]{Conjecture}
\newtheorem{corollary}[theorem]{Corollary}
\newtheorem{lemma}[theorem]{Lemma}

\theoremstyle{definition}

\newtheorem{remark}[theorem]{Remark}

\title[Cohomological finite generation for restricted Lie super\-algebras]{Cohomological finite generation for restricted Lie super\-algebras and finite supergroup schemes}

\author{Christopher M.\ Drupieski}
\address{Department of Mathematical Sciences, DePaul University, Chicago, IL 60614, USA}
\email{cdrupies@depaul.edu}

\date{\today}


\begin{document}

\begin{abstract}
We prove that the cohomology ring of a finite-dimensional restricted Lie superalgebra over a field of characteristic $p > 2$ is a finitely-generated algebra. Our proof makes essential use of the explicit projective resolution of the trivial module constructed by J.\ Peter May for any graded restricted Lie algebra. We then prove that the cohomological finite generation problem for finite supergroup schemes over fields of odd characteristic reduces to the existence of certain conjectured universal extension classes for the general linear supergroup $GL(m|n)$ that are similar to the universal extension classes for $GL_n$ exhibited by Friedlander and Suslin.
\end{abstract}

\maketitle

\section{Introduction}

Motivated by the desire to apply geometric methods in the study of cohomology and modular representation theory, much attention has been paid in the past decades to the qualitative properties of the cohomology rings associated to various classes of Hopf algebras. Specifically, given a Hopf algebra $A$ over the field $k$, if the cohomology ring $\Hbul(A,k)$ is a finitely-generated $k$-algebra, then one can establish, via a now well-tread path, a theory of support varieties that associates a geometric invariant to each $A$-module $M$. One of the most general results in this direction is due to Friedlander and Suslin \cite{Friedlander:1997}. Building on the earlier work of Venkov \cite{Venkov:1959} and Evens \cite{Evens:1961} for finite groups, as well as the work of Friedlander and Parshall \cite{Friedlander:1986,Friedlander:1986b} for restricted Lie algebras, Friedlander and Suslin proved cohomological finite generation for any finite-dimensional cocommuative Hopf algebra over $k$, by using the fact that the category of finite-dimensional cocommuative Hopf algebras over $k$ is naturally equivalent to the category of finite $k$-group schemes. Their argument involved embedding a given finite group scheme in some general linear group, and then exhibiting certain universal extension classes for $GL_n$ that provided the necessary generators for the cohomology ring. More generally, Etingof and Ostrik \cite{Etingof:2004} have conjectured that the Ext-algebra of an arbitrary finite tensor category is a finitely-generated algebra. This conjecture encompasses the cohomology rings of arbitrary finite-dimensional Hopf algebras and of more general finite-dimensional algebras.

In this paper we prove that the cohomology ring of an arbitrary finite-dimensional restricted Lie super\-algebra is a finitely-generated algebra. Similar results have also been claimed by Liu \cite{Liu:2012} and by Bagci \cite{Bagci:2012a}, though their proofs are incomplete; see Remarks \ref{remark:Liu} and \ref{remark:Bagci}. We then begin the process of verifying Etingof and Ostrik's conjecture for the class of finite tensor categories that arise as module categories for finite-dimensional cocommuative Hopf superalgebras. Since the category of finite-dimensional cocommuative Hopf superalgebras is naturally equivalent to the category of finite supergroup schemes, this is equivalent to showing that the cohomology ring $\Hbul(G,k)$ of an arbitrary finite supergroup scheme $G$ over $k$ is a finitely-generated algebra. In analogy to the classical situation, we show that the finite generation problem for finite supergroup schemes reduces to the existence of certain conjectured universal extension classes for the general linear supergroup $GL(m|n)$. Combined with the extension classes for $GL_n$ exhibited by Friedlander and Suslin, the conjectured classes provide the necessary generators for $\Hbul(G,k)$.

We do not address the subject of support varieties for finite supergroup schemes in any depth at this time, though using our cohomological finite generation result for restricted Lie superalgebras, one can begin to recover the results of Liu \cite[\S4]{Liu:2012} and of Bagci \cite[\S4]{Bagci:2012a} (taking care to note that the two authors' definitions of support varieties for restricted Lie superalgebras are slightly different; see Remark \ref{remark:braidedgraded}). In future work we hope to address the existence of the conjectured universal extension classes for $GL(m|n)$, and to explore connections to the representation theory and cohomology of Schur superalgebras. We also hope to apply the methods and techniques developed in this paper to better understand the explicit structure of the cohomology rings for Lie superalgebras, restricted Lie superalgebras, and finite supergroup schemes over fields of positive characteristic, with the ultimate goal of understanding cohomological support varieties for these objects.

Now let us summarize the contents of the paper in slightly more detail. We begin in Section \ref{section:preliminaries} with an account of the preliminary material needed for the rest of the paper. Next, Section \ref{section:cohomologyrls} forms the paper's technical heart, in which we establish cohomological finite generation for finite-dimensional restricted Lie superalgebras. Broadly speaking, the finite generation argument for a restricted Lie superalgebra $L$ parallels the argument of Friedlander and Parshall \cite{Friedlander:1986b} for an ordinary restricted Lie algebra. Specifically, we construct in Section \ref{subsection:Mayspecseq} a spectral sequence, natural in $L$, that converges to the cohomology ring $\Hbul(V(L),k)$. Here $V(L)$ is the restricted enveloping superalgebra of $L$. The $E_1$-page of the spectral sequence turns out to be noetherian over a finitely-generated polynomial algebra. We then show in Section \ref{subsection:comparison} that the polynomial algebra consists of permanent cycles, and use this to deduce that $\Hbul(V(L),k)$ is finitely-generated (Theorem \ref{theorem:V(L)cohomfingen}).

The major obstacle to implementing the strategy of the previous paragraph is the ``super'' phenomenon that, in contrast to the classical situation of \cite{Friedlander:1986b}, the aforementioned polynomial algebra does not lie on the bottom row of a first quadrant spectral sequence (even after reindexing the spectral sequence), and hence does not automatically consist of permanent cycles. To overcome this obstacle, we consider in Section \ref{subsection:comparison} a new spectral sequence, which though not natural in $L$, is much more explicit, arising as it does from an explicit projective resolution of the trivial module constructed by J.\ Peter May \cite{May:1966}. The details of (a modified version of) the projective resolution are summarized in Section \ref{subsection:Koszulresolutionrls}. The new and old spectral sequences are isomorphic from the $E_1$-page onward, and using the new spectral sequence we are able to show that the polynomial algebra is generated by permanent cycles. Along the way to proving cohomological finite generation for finite-dimensional restricted Lie superalgebras, we establish in Section \ref{subsection:ordinarycohomology} the corresponding result for ordinary finite-dimensional Lie superalgebras over fields of odd characteristic (Theorem \ref{theorem:fgforU(L)}). This last mentioned result is not obvious in the super case as it is for ordinary Lie algebras, since the cohomology ring of a finite-dimensional Lie superalgebra need not vanish in large degrees.

In Section \ref{section:supergroupschemes} we recall the basic notions and definitions for supergroup schemes that are needed to study the cohomology rings of finite and infinitesimal supergroup schemes. Our treatment parallels that of Jantzen \cite{Jantzen:2003} for ordinary algebraic group schemes. Some of the topics touched here have been treated before in the papers of Brundan and Kleshchev \cite{Brundan:2003}, Shu and Wang \cite{Shu:2008}, and Zubkov \cite{Zubkov:2009}. In addition to establishing notation, we prove in Section \ref{subsection:Frobeniuskernels} that restricted enveloping superalgebras correspond, under the equivalence between the category of finite-dimensional cocommutative Hopf superalgebras and the category of finite supergroup schemes, to height one infinitesimal supergroup schemes. In particular, given a supergroup scheme $G$ with restricted Lie superalgebra $\g = \Lie(G)$, the representation theory of $\Vg$ is equivalent to that of $G_1$, the first Frobenius kernel of $G$.

Finally, in Section \ref{section:cohomologysupergroups} we begin the treatment of cohomology rings for arbitrary finite supergroup schemes. Given a finite supergroup scheme $G$ with restricted Lie superalgebra $\g$, in Section \ref{subsection:MayspecseqG1} we carefully show that the May spectral sequence of Section \ref{subsection:Mayspecseq} that converges to $\Hbul(\Vg,k) = \Hbul(G_1,k)$ is a spectral sequence of rational $G$-supermodules. Next, in Section \ref{subsection:reductioninfinitesimal} we show that a finite supergroup scheme $G$ over a (perfect) field of odd characteristic is the semidirect product of an infinitesimal supergroup scheme $G^0$ and an ordinary etale finite group scheme $\pi_0(G)$. Applying a Lyndon--Hochschild--Serre spectral sequence as in \cite{Friedlander:1997}, this lets us reduce the cohomological finite generation problem for finite supergroup schemes to the problem for infinitesimal supergroup schemes. Finally, in Section \ref{subsection:cohomologyinfinitesimal} we show that the finite generation problem for infinitesimal supergroup schemes reduces to the existence of certain conjectured extension classes for the general linear supergroup $GL(m|n)$, described in Conjecture \ref{conjecture:univextclasses}. These extension classes do not by themselves provide the generators for $\Hbul(G,k)$, but must be combined with the extension classes for $GL_{m+n}$ previously exhibited by Friedlander and Suslin. This corresponds to the fact that the quotient supergroup $GL(m|n)/GL(m|n)_1$ is isomorphic to the underlying purely even subgroup $GL_n \times GL_m$ of $GL(m|n)$ (cf.\ Lemma \ref{lemma:G/Gr}), which is then a closed subgroup of $GL_{m+n}$.

\section{Preliminaries} \label{section:preliminaries}

\subsection{Graded objects} \label{subsection:conventions}

Let $k$ be a field of characteristic $p \neq 2$. All objects in this paper are defined over $k$, all unadorned tensor products denote tensor products over $k$, and, except as indicated, all algebras are associative with unit. Set $\Z_2 = \Z/2\Z = \{\ol{0},\ol{1}\}$. Recall that a super\-space $V = \Vzero \oplus \Vone$ is a $\Z_2$-graded vector space. Given a super\-space $V$ and a homogeneous element $v \in V$, write $\ol{v} \in \Z_2$ for the $\Z_2$-degree of $v$. We say that $v$ is even (resp.\ odd) if $\ol{v} = \ol{0}$ (resp.\ $\ol{v} = \ol{1}$). Similarly, given $n \in \Z$, write $\ol{n} \in \Z_2$ for the parity of $n$. We assume in this paper that the reader is familiar with the sign conventions for working with $\Z_2$-graded objects, and is familiar with the definitions for such objects as super\-algebras, super\-coalgebras, super\-bialgebras, and Hopf super\-algebras. The reader unfamiliar with these notions can consult the thesis of Westra \cite{Westra:2009} and the references therein.

In general, we use the adjective \emph{graded} to indicate that an object admits an additional $\Z$-grading that is compatible with its underlying structure. Thus, a \emph{graded space} is a $\Z$-graded vector space, a \emph{graded super\-space} is a $(\Z \times \Z_2)$-graded vector space, a \emph{graded super\-algebra} is a $(\Z \times \Z_2)$-graded algebra, and so on. Given a graded super\-space $V$ and an element $v \in V$ of bidegree $(m,\ol{n}) \in \Z \times \Z_2$, we call $m$ the \emph{external degree} of $v$, written $\deg(v) = m$, and we call $\ol{n}$ the \emph{internal degree} of $v$. In general, if the adjective \emph{graded} is omitted from the description of an object, then we consider the object as graded but concentrated in external degree $0$. In particular, we consider the field $k$ as a graded super\-algebra concentrated in bidegree $(0,\ol{0})$.

Let $V$ and $W$ be graded super\-spaces, and set $\Gamma = \Z \times \Z_2$. Given $\gamma \in \Gamma$, the $\gamma$-graded component of $V$ is denoted $V_\gamma$. The tensor product $V \otimes W$ is $\Gamma$-graded with $(V \otimes W)_\gamma = \sum_{\alpha + \beta = \gamma} V_\alpha \otimes W_\beta$. The graded super\-spaces $V \otimes W$ and $W \otimes V$ are isomorphic via the graded twist map $T$, which is defined on homogeneous simple tensors by
\begin{equation} \label{eq:gradedtwist}
T(v \otimes w) = (-1)^{\deg(v)\deg(w)} (-1)^{\ol{v} \cdot \ol{w}} w \otimes v.
\end{equation}
For the remainder of the paper, whenever we state a formula such as \eqref{eq:gradedtwist} in which either internal or external degrees have been specified, we mean that the formula holds as given for homogeneous elements, and that it extends linearly to nonhomogeneous elements. If $V$ and $W$ are concentrated in external degree $0$, then $T$ restricts to the usual twist map on tensor products of super\-spaces.

A linear map $f: V \rightarrow W$ is homogeneous of degree $\gamma$ if $f(V_\alpha) \subset W_{\alpha - \gamma}$ for all $\alpha \in \Gamma$. The space of all degree-$\gamma$ homogeneous linear maps from $V$ to $W$ is denoted $\Hom_k(V,W)_\gamma$. Suppose $V$ is finite-dimensional. Then $\Hom_k(V,W) = \bigoplus_{\gamma \in \Gamma} \Hom_k(V,W)_\gamma$. In particular, $V^* = \Hom_k(V,k)$ is a graded superspace with $(V^*)_\gamma = (V_\gamma)^*$. Also, $V \cong (V^*)^*$ as graded super\-spaces via the map $\varphi: V \rightarrow (V^*)^*$ satisfying $\varphi(v)(g) = (-1)^{\deg(v) \deg(g)}(-1)^{\ol{v} \cdot \ol{g}} g(v)$. Continuing to assume that $V$ is finite-dimensional, $W \otimes V^* \cong \Hom_k(V,W)$ as graded super\-spaces via the map $\eta: W \otimes V^* \rightarrow \Hom_k(V,W)$ defined by $\eta(w \otimes g)(v) = w \cdot g(v) = g(v).w$. Given graded superspaces $V,W,V',W'$ and linear maps $f: V \rightarrow V'$ and $g: W \rightarrow W'$, define $f \otimes g: V \otimes W \rightarrow V' \otimes W'$ to be the linear map satisfying
\begin{equation} \label{eq:signsforproductofmaps}
(f \otimes g)(v \otimes w) = (-1)^{\deg(g)\deg(v)}(-1)^{\ol{g} \cdot \ol{v}} f(v) \otimes g(w).
\end{equation}

A graded super\-algebra $A$ with product $\mu: A \otimes A \rightarrow A$ is \emph{graded-commutative} if $\mu \circ T = \mu$, that is, if $ab = (-1)^{\deg(a)\deg(b)} (-1)^{\ol{a} \cdot \ol{b}} ba$ for all homogeneous elements $a,b \in A$. Similarly, $A$ and is \emph{graded-anticommutative} if $\mu \circ T = -\mu$. A graded supercoalgebra $C$ with coproduct $\Delta: C \rightarrow C \otimes C$ is \emph{graded-cocommutative} if $T \circ \Delta = \Delta$. If $A$ and $B$ are graded super\-algebras with product maps $\mu_A$ and $\mu_B$, then the tensor product $A \otimes B$ is a graded super\-algebra with product map $(\mu_A \otimes \mu_B) \circ (1 \otimes T \otimes 1)$. Similarly, if $C$ and $D$ are graded super\-coalgebras with coproducts $\Delta_C$ and $\Delta_D$, then $C \otimes D$ is a graded super\-coalgebra with coproduct $(1 \otimes T \otimes 1) \circ (\Delta_C \otimes \Delta_D)$. If $A$ is a super\-algebra, considered as a graded super\-algebra concentrated in external degree $0$, then $A$ is graded-commutative (resp.\ graded-anticommutative) if and only if it is commutative (resp.\ anticommutative) in the usual sense for super\-algebras.\footnote{In this paper, the term \emph{commutative}, as applied to super\-algebras, will always be used in the sense given here. The usual notion of commutativity for abstract rings will be referred to as \emph{ordinary} commutativity.} Similarly, if $C$ is a graded super\-coalgebra concentrated in external degree $0$, then $C$ is graded-cocommutative if and only if it is cocommutative in the usual sense for super\-coalgebras.

Let $A$ be a graded super\-algebra with product $\mu: A \otimes A \rightarrow A$. Then $A$ is a \emph{differential graded super\-algebra} if there exists a homogeneous linear map $d: A \rightarrow A$, of bidegree $(1,\ol{0})$ or $(-1,\ol{0})$, satisfying $d \circ d = 0$ and $\mu \circ (d \otimes 1 + 1 \otimes d) = d \circ \mu$. In other words, $d(ab) = d(a) \cdot b + (-1)^{\deg(a)} a \cdot d(b)$ for all $a,b \in A$. Similarly, a graded super\-coalgebra $C$ with coproduct $\Delta$ is a \emph{differential graded super\-coalgebra} if there exists a homogeneous linear map $d : C \rightarrow C$, of bidegree $(1,\ol{0})$ or $(-1,\ol{0})$, satisfying $d \circ d = 0$ and $(d \otimes 1 + 1 \otimes d) \circ \Delta = \Delta \circ d$. If $A$ and $B$ are differential graded super\-algebras with differentials $d_A$ and $d_B$, respectively, then $A \otimes B$ is a differential graded super\-algebra with differential $d_{A \otimes B} = d_A \otimes 1 + 1 \otimes d_B$. Similarly, the tensor product of two differential graded super\-coalgebras is again a differential graded super\-coalgebra.

\subsection{Graded modules} \label{subsection:gradedmodules}

Let $A$ be a graded super\-algebra. A graded super\-space $M$ is a \emph{graded left $A$-super\-module} if $M$ is a left $A$-module and if the module structure map $\sigma: A \otimes M \rightarrow M$ is homo\-geneous of bidegree $(0,\ol{0})$. Graded right $A$-super\-modules are defined similarly. Let $M$ and $N$ be graded left $A$-super\-modules. A linear map $g: M \rightarrow N$ is a \emph{graded left $A$-super\-module homo\-morphism} if for all $m \in M$ and $a \in A$, one has $g(a.m) = (-1)^{\deg(a)\deg(g)} (-1)^{\ol{a} \cdot \ol{g}} a.g(m)$. Similarly, if $M$ and $N$ are graded right $A$-super\-modules, then $g: M \rightarrow N$ is a graded right $A$-super\-module homomorphism if $g(m.a) = g(m).a$ for all $m \in M$ and $a \in A$. The space of all graded $A$-super\-module homomorphisms $M \rightarrow N$ is denoted $\Hom_A(M,N)$.

Let $A$ and $B$ be graded super\-algebras, let $M$ be a graded left $A$-super\-module with module structure map $\sigma: A \otimes M \rightarrow M$, and let $N$ be a graded left $B$-super\-module with structure map $\tau: B \otimes N \rightarrow N$. Then $M \otimes N$ is a graded left $A \otimes B$-super\-module with module structure map $(\sigma \otimes \tau) \circ (1 \otimes T \otimes 1)$. In particular, if $A$ is a graded super\-bialgebra with coproduct $\Delta: A \rightarrow A \otimes A$, and if $M$ and $N$ are graded left $A$-super\-modules, then $M \otimes N$ is naturally a graded left $A$-supermodule via $\Delta$; we call this the diagonal action of $A$ on $M \otimes N$. The conventions for right modules are similar. Suppose $M$ is finite-dimensional, and that $A$ is a graded Hopf super\-algebra with antipode $S$. Then $M^*$ is a graded left $A$-super\-module, with the action of $a \in A$ on $g \in M^*$ defined by $(a.g)(m) = (-1)^{\deg(a) \deg(g)} (-1)^{\ol{a} \cdot \ol{g}} g(S(a).m)$, and $\Hom_k(M,N)$ is a graded left $A$-super\-module via the isomorphism $\Hom_k(M,N) \cong N \otimes M^*$ and the diagonal action of $A$ on $N \otimes M^*$.

Let $A$ be a differential graded super\-algebra with differential $d_A$, and let $P$ be a graded left $A$-super\-module with structure map $\sigma: A \otimes P \rightarrow P$. Then $P$ is a \emph{differential graded $A$-super\-module} if there exists a homogeneous map $d_P: P \rightarrow P$, of bidegree $(1,\ol{0})$ or $(-1,\ol{0})$, such that $d_P \circ d_P = 0$ and $d_P \circ \sigma = \sigma \circ (d_A \otimes 1 + 1 \otimes d_P)$. In particular, suppose $A$ is concentrated in external degree $0$, so that $d_A = 0$, and suppose $d_P$ is of bidegree $(1,\ol{0})$. Then $P$ is a chain complex of left $A$-super\-modules. Let $M$ be a left $A$-super\-module, considered as a graded $A$-supermodule concentrated in external degree $0$. Then we consider $\bigoplus_{n \in \N} \Hom_A(P_n,M)$ as a cochain complex with differential $\partial: \Hom_A(P_n,M) \rightarrow \Hom_A(P_{n+1},M)$ defined by
\begin{equation} \label{eq:cochaindifferential}
\partial(f) = (-1)^{\deg(f)} f \circ d = (-1)^n f \circ d.
\end{equation}

\subsection{Symmetric, exterior, and divided polynomial algebras} \label{subsection:superexterioralgebra}

Let $V = \Vzero \oplus \Vone$ be a super\-space. Considering $V$ as a graded super\-space concentrated in external degree $1$, the tensor algebra $T(V)$ is naturally a graded super\-algebra. The \emph{supersymmetric algebra} $S_s(V)$ is defined by
\[
S_s(V) = T(V)/\subgrp{x \otimes y - (-1)^{\ol{x} \cdot \ol{y}}y \otimes x: x,y \in V},
\]
and the \emph{superexterior algebra} $\Lambda_s(V)$ is the quotient of $T(V)$ defined by
\[
\Lambda_s(V) = T(V)/\subgrp{x \otimes y + (-1)^{\ol{x} \cdot \ol{y}}y \otimes x: x,y \in V}.
\]
Then $S_s(V)$ is the free commutative super\-algebra generated by $V$, and $\Lambda_s(V)$ is the free graded-commutative graded super\-algebra generated by $V$. The super\-exterior algebra $\Lambda_s(V)$ is also the free anti-commutative super\-algebra generated by $V$. The linear maps $V \rightarrow V \otimes V: v \mapsto v \otimes 1 + 1 \otimes v$ and $V \rightarrow V: v \mapsto -v$ extend to a super\-algebra homomorphism $T(V) \rightarrow T(V) \otimes T(V)$ and a super\-algebra anti-homomorphism $T(V) \rightarrow T(V)$, respectively, as well as to a homomorphism of graded super\-algebras $T(V) \rightarrow T(V) \otimes T(V)$, and an anti-homomorphism of graded super\-algebras $T(V) \rightarrow T(V)$. In the former case, the maps endow $T(V)$ and $S_s(V)$ with the structures of Hopf super\-algebras, and in the latter case, they endow $T(V)$ and $\Lambda_s(V)$ with the structures of graded Hopf super\-algebras. (In the latter case, the external grading plays a role in the algebra structure on the tensor product $T(V) \otimes T(V)$.)

Write $S(V)$, $\Lambda(V)$, and $\Gamma(V)$ for the ordinary symmetric, exterior, and divided polynomial algebras on $V$, respectively. Then $S(\Vzero)$ and $\Lambda(\Vone)$ are naturally commutative super\-algebras, and $\Lambda(\Vzero)$, $S(\Vone)$, and $\Gamma(\Vone)$ are naturally graded-commutative graded super\-algebras. As an algebra, $S_s(V)$ is isomorphic to the tensor product of commutative super\-algebras $S(\Vzero) \otimes E(\Vone)$, and $\Lambda_s(V)$ is isomorphic to the tensor product of graded-commutative graded super\-algebras $\Lambda(\Vzero) \otimes S(\Vone)$. Fix bases $\set{x_1,\ldots,x_s}$ and $\set{y_1,\ldots,y_t}$ for $\Vzero$ and $\Vone$, respectively. The exterior algebra $\Lambda(\Vzero)$ is generated as a graded super\-algebra by the symbols $\subgrp{x_i}$ of bidegrees $(1,\ol{0})$. We denote a typical monomial in $\Lambda(\Vzero)$ by $\subgrp{x_{i_1} \cdots x_{i_n}}$. The polynomial algebra $S(\Vone)$ is generated by the symbols $y_j$ of bidegrees $(1,\ol{1})$, with a typical monomial denoted $y_1^{a_1} y_2^{a_2} \cdots y_t^{a_t}$. The divided polynomial algebra $\Gamma(\Vone)$ is generated by the symbols $\gamma_r(y)$ for $r \in \N$ and $y \in \Vone$, with $\gamma_r(y)$ of bidegree $(r,\ol{r})$. The generators for $\Gamma(\Vone)$ satisfy the relations $\gamma_0(y) = 1$, $\gamma_r(0) = 0$ if $r \geq 1$, $\gamma_r(y)\gamma_s(y) = \binom{r+s}{r} \gamma_{r+s}(y)$, and $\gamma_r(y+y') = \sum_{i=0}^r \gamma_i(y)\gamma_{r-i}(y')$. By convention, $\gamma_r(y) = 0$ if $r < 0$.

The external degree-$n$ components of $\Lambda(\Vzero)$, $S(\Vone)$, and $\Gamma(\Vzero)$ are denoted $\Lambda^n(\Vzero)$, $S^n(\Vone)$, and $\Gamma_n(\Vone)$, respectively. Then $\Gamma_n(\Vone)$ admits a basis consisting of all monomials $\gamma_{a_1}(y_1) \cdots \gamma_{a_t}(y_t)$ with $a_j \in \N$ and $\sum_{i=1}^t a_i = n$. The coproducts on $\Lambda(\Vzero)$ and $S(\Vone)$ are determined by the maps $\subgrp{x_i} \mapsto \subgrp{x_i} \otimes 1 + 1 \otimes \subgrp{x_i}$ and $y_j \mapsto y_j \otimes 1 + 1 \otimes y_j$, while the coproduct on $\Gamma(\Vone)$ is determined by the map $\gamma_r(y) \mapsto \sum_{i=0}^r \gamma_i(y) \otimes \gamma_{r-i}(y)$. If $\chr(k) = 0$, then $\Gamma(\Vone) \cong S(\Vone)$ as bialgebras via the map $\gamma_r(y) \mapsto \frac{1}{r!} y^r$. The monomials $\subgrp{x_{i_1} \cdots x_{i_b}} \gamma_{a_1}(y_1) \cdots \gamma_{a_t}(y_t)$ with $1 \leq i_1 < \cdots < i_b \leq s$  and $a_j \in \N$ form a basis for the graded-commutative graded super\-algebra $\Lambda(\Vzero) \otimes \Gamma(\Vone)$, which we call the standard monomial basis for $\Lambda(\Vzero) \otimes \Gamma(\Vone)$. Similarly, we call the monomials $\subgrp{x_{i_1} \cdots x_{i_s}} y_1^{a_1} \cdots y_t^{a_t}$ the standard monomial basis vectors for $\Lambda(\Vzero) \otimes S(\Vone)$.

\subsection{Categories} \label{subsection:categories}

The category of $k$-super\-spaces $\fsvec_k$ is the abelian category with objects consisting of all $k$-super\-spaces, and with morphisms consisting of the even linear maps between super\-spaces. Given $V,W \in \fsvec_k$, the set of morphisms $V \rightarrow W$ in $\fsvec_k$ is denoted $\Hom_{\fsvec_k}(V,W)$. Then $\Hom_{\fsvec}(V,W) = \Hom_k(V,W)_{\ol{0}}$. The category consisting of the $k$-super\-spaces and arbitrary linear maps between superspaces is denoted $\svec_k$. We write $\alg_k$ for the category of $k$-algebras and $k$-algebra homomorphisms. The category $\salg_k$ of $k$-super\-algebras is the subcategory of $\alg_k$ with objects consisting of the super\-algebras over $k$, and with morphisms consisting of the even $k$-algebra homo\-morphisms between super\-algebras. Similarly, if $A \in \salg_k$ is commutative, then the category of $A$-superalgebras is denoted $\salg_A$. The subcategory of $\salg_k$ consisting of only the commutative $k$-superalgebras is denoted $\csalg_k$. When the coefficient ring is clear from the context, we may omit the subscript $k$ from the notations $\alg_k$, $\salg_k$, and $\csalg_k$.

Let $A \in \salg_k$. The category $\smod_A$ of (left or right) $A$-super\-modules and arbitrary $A$-supermod\-ule homomorphisms is not abelian, though the underlying even subcategory $\fsmod_A$, consisting of the same objects but only the even homo\-morphisms, is an abelian category. Then $\Hom_{\fsmod_A}(M,N) = \Hom_A(M,N)_{\ol{0}}$. Similarly, given a super\-coalgebra $C$ over $k$, the category of $C$-supercomodules and arbitrary $C$-supercomodule homomorphisms is denoted $\scomod_C$, and the underlying abelian even subcategory, consisting of the same objects but only the even supercomodule homomorphisms, is denoted $\fscomod_C$. Now suppose $A \in \csalg_k$. Then the categories of left $A$-super\-modules, right $A$-super\-modules, and $A$-super\-bimodules are all equivalent. Indeed, $M$ is a left $A$-super\-module with left action $a \otimes m \mapsto a.m$ if and only if $M$ is a right $A$-super\-module with right action $m \otimes a \mapsto (-1)^{\ol{a} \cdot \ol{m}} a.m$. If $V,W \in \fsvec_k$, then there exists a natural isomorphism of $A$-superbimodules
\begin{equation} \label{eq:VWbimodiso}
(V \otimes A) \otimes_A (W \otimes A) \cong (V \otimes W) \otimes A,
\end{equation}
which identifies $(v \otimes a) \otimes_A (w \otimes a')$ with $(-1)^{\ol{a} \cdot \ol{w}}(v \otimes w) \otimes aa'$.

\subsection{Cohomology} \label{subsection:cohomology}

Let $A \in \salg_k$. A super\-module $P \in \fsmod_A$ is called projective if the functor $\Hom_A(P,-): \fsmod_A \rightarrow \fsvec_k$ is exact. Similarly, $Q \in \fsmod_A$ is injective if $\Hom_A(-,Q): \fsmod_A \rightarrow \fsvec_k$ is exact. The category $\fsmod_A$ contains enough projectives and enough injectives \cite[\S\S6.2--6.3]{Westra:2009}, so we can apply the machinery of homological algebra to define cohomology groups in $\fsmod_A$. Specifically, given supermodules $M,N \in \fsmod_A$, define $\Ext_A^n(M,-)$ to be the $n$-th right derived functor of $\Hom_A(M,-): \fsmod_A \rightarrow \fsvec_k$, and $\Ext_A^n(-,N)$ to be the $n$-th right derived functor of $\Hom_A(-,N): \fsmod_A \rightarrow \fsvec_k$. In particular, if $A$ is an augmented $k$-superalgebra, then $\opH^n(A,-)$ is the $n$-th right derived functor of $\Hom_A(k,-) \cong (-)^A: \fsmod_A \rightarrow \fsvec_k$.

Let $A$ be a super\-bialgebra, and let $P_\bullet \rightarrow k$ be a projective resolution in $\fsmod_A$. Then the diagonal action makes the tensor product of complexes $P_\bullet \otimes P_\bullet$ into an exact complex of left $A$-super\-modules, and hence there exists a chain homomorphism $D \in \Hom_{\fsmod_A}(P_\bullet,P_\bullet \otimes P_\bullet)$ lifting the identity $k \rightarrow k$. Now let $M$ and $N$ be left $A$-super\-modules, and let $f \in \Hom_A(P_m,M)$ and $g \in \Hom_A(P_n,N)$ be cocycles. Then $f \odot g := (f \otimes g) \circ D: P_{m+n} \rightarrow M \otimes N$ is a cocycle in $\Hom_A(P_\bullet,M \otimes N)$, and the induced map on cohomology groups $\Hbul(A,M) \otimes \Hbul(A,N) \rightarrow \Hbul(A, M \otimes N)$ is called the cup product. By abuse of terminology, we call $f \odot g$ the cup product of the cochains $f$ and $g$. Taking $M = N = k$, the cup product gives $\Hbul(A,k)$ the structure of a graded super\-algebra. At the level of cochains, the cup product $f \odot g$ satisfies the derivation property
\begin{equation} \label{eq:differentialderivation}
\partial(f \odot g) = \partial(f) \odot g + (-1)^{\deg(f)} f \odot \partial(g).
\end{equation}
If the diagonal approximation $D: P_\bullet \rightarrow P_\bullet \otimes P_\bullet$ is coassociative, then $\partial$ makes $\Hom_A(P_\bullet,k)$ into an associative (though perhaps nonunital) differential graded super\-algebra.

\subsection{The bar and cobar complexes} \label{subsection:barcomplex}

Let $A$ be an augmented $k$-algebra with augmentation map $\ve: A \rightarrow k$. Recall that the (un-normalized) left bar resolution $B(A) = \bigoplus_{n \in \N} B_n(A)$ for $A$ is the $A$-free resolution of $k$ with $B_n(A) = A \otimes A^{\otimes n}$. Given $a_0,a_1,\ldots,a_n \in A$, write $a_0[a_1|\cdots|a_n]$ for the element $a_0 \otimes a_1 \otimes \cdots \otimes a_n \in A \otimes A^{\otimes n}$. Elements of $B_0(A) = A$ are written in the form $a [\ ]$. Then the differential on $B(A)$ is defined by
\begin{equation} \label{eq:bardifferential} \textstyle
\begin{split}
d(a_0[a_1|\cdots|a_n]) &= a_0a_1[a_2|\cdots|a_n]  + (-1)^n a_0[a_1|\cdots|a_{n-1}]\ve(a_n),\\
&\phantom{=} + \textstyle \sum_{i=1}^{n-1} (-1)^i a_0[a_1|\cdots|a_ia_{i+1}|\cdots|a_n]
\end{split}
\end{equation}
The map $s: B_n(A) \rightarrow B_{n+1}(A)$ defined by $s(a_0[a_1|\cdots|a_n]) = [a_0|a_1|\cdots|a_n]$ is a contracting homo\-topy on $B(A)$. Set $\Bbar_n(A) = 1 \otimes A^{\otimes n} \subset B_n(A)$.

If $A$ is an augmented super\-algebra, considered as a graded super\-algebra concentrated in external degree $0$, then $B(A)$ is naturally a differential graded $A$-super\-module. If $A$ is a super\-bialgebra with coproduct $\Delta$, then an explicit diagonal approximation $D: B(A) \rightarrow B(A) \otimes B(A)$ is defined by
\begin{equation} \label{eq:diagonalapproxbar} \textstyle
D([a_1|\ldots|a_n]) = \sum_{i=0}^n (-1)^{\mu} [a_1'|\cdots|a_i'] \cdot \ve(a_{i+1}' \cdots a_n') \otimes a_1'' \cdots a_i''[a_{i+1}''|\cdots|a_n''].
\end{equation}
Here the $a_i$ are assumed to be homogeneous, $\Delta(a_i) = \sum a_i' \otimes a_i''$ is the coproduct of $a_i$ written as a sum of homogeneous simple tensors (the index of summation being understood, and the summation over each term being understood in the formula for $D$), and $\mu = \sum_{i < j} \ol{a_i}'' \cdot \ol{a_j}'$. The formula for $D([a_1|\ldots|a_n])$ is obtained by considering the image of $[a_1|\cdots|a_n]$ under the map $B(A) \rightarrow B(A \otimes A)$ induced by the coproduct on $A$, and then applying the super analogue of the diagonal approximation formula in \cite[p.\ 221]{Cartan:1999}. Since $\ve$ is a super\-algebra homomorphism, $\ve(a_{i+1}' \cdots a_n') = 0$ if $\ol{a_j}' = \ol{1}$ for any $i < j \leq n$. Then it follows that \eqref{eq:diagonalapproxbar} can be rewritten as
\begin{equation}
\begin{split} 
D([a_1|\cdots|a_n]) &= \textstyle \sum_{i=0}^n [a_1'|\cdots|a_i'] \cdot \ve(a_{i+1}' \cdots a_n') \otimes a_1'' \cdots a_i''[a_{i+1}''|\cdots|a_n''] \\
&= \textstyle \sum_{i=0}^n [a_1'|\cdots|a_i'] \otimes a_1'' \cdots a_i''[a_{i+1}|\cdots|a_n].
\end{split}
\end{equation}

Let $A$ be a super\-bialgebra, and let $M$ be a left $A$-super\-module. For each $n \in \N$, set $C^n(A,M) = \Hom_A(B_n(A),M)$. Then $\opH^n(A,M)$ is equal to the $n$-th cohomology group of the cochain complex $\Cbul(A,M)$. Given left $A$-super\-modules $M$ and $N$ and cochains $f \in C^m(A,M)$ and $g \in C^n(A,N)$, the cup product $f \odot g \in C^{m+n}(A,M \otimes N)$ is given by
\[
(f \odot g)([a_1|\cdots|a_{m+n}]) = \sum (-1)^\sigma f([a_1'|\cdots|a_m']) \otimes a_1'' \cdots a_m''. g([a_{m+1}|\cdots|a_{m+n}]).
\]
Here $\sigma = m \cdot \deg(g) + \ol{g} \cdot (\ol{a_1} + \cdots + \ol{a_m})$. In particular, if $N = k$, it follows that
\begin{equation} \label{eq:cupproductN=k}
(f \odot g)([a_1|\cdots|a_{m+n}]) = (-1)^{\sigma} f([a_1|\cdots|a_m]) \otimes g([a_{m+1}|\cdots|a_{m+n}]),
\end{equation}
and hence that the cup product makes $\Cbul(A,k)$ into a (unital) differential graded super\-algebra.

Now suppose that $A$ is finite-dimensional. Then $\Bbar_n(A)$ is finite-dimensional, and there exist natural isomorphisms of graded super\-spaces
\[
C^n(A,M) \cong \Hom_k(\Bbar_n(A),M) \cong M \otimes \Hom_k(\Bbar_n(A),k) \cong M \otimes C^n(A,k)
\]
Given $m \in M$ and $f \in C^n(A,k)$, the $A$-supermodule homomorphism corresponding to $m \otimes f \in M \otimes C^n(A,k)$ is defined for $a,a_1,\ldots,a_n \in A$ by
\[
(m \otimes f)(a[a_1|\cdots|a_n]) = (-1)^{\ol{a} \cdot (\ol{m}+\ol{f})} (a.m) \cdot f([a_1|\cdots|a_n]).
\]
Given $m \in M$, write $\wt{m} \in C^0(A,M)$ for the homomorphism satisfying $\wt{m}(a) = (-1)^{\ol{m} \cdot \ol{a}}a.m$. Then $m \otimes f$ identifies with the cup product $\wt{m} \odot f$. Under this identification, the differential $\partial: M \rightarrow M \otimes C^1(A,k) \cong M \otimes A^*$ satisfies $\partial(\wt{m}) = (\sum_i \wt{m}_i \otimes f_i) - m \otimes \ve$, where $\sum_i f_i(a).m_i = (-1)^{\ol{m} \cdot \ol{a}}a.m$, and $\ve \in C^1(A,k) \cong \Hom_k(A,k)$ is the augmentation map. Using \eqref{eq:cupproductN=k}, one can check that the cup product induces an isomorphism of graded super\-algebras $\Cbul(A,k) \cong T(C^1(A,k))$, that is, $\Cbul(A,k)$ is isomorphic to the tensor algebra on $C^1(A,k)$. Thus, \eqref{eq:differentialderivation} implies that $\Cbul(A,M)$ identfies as a differential graded right $\Cbul(A,k)$-super\-module with $M \otimes \Cbul(A,k)$, and that the differential on $\Cbul(A,M)$ is determined by its restrictions to $M \cong C^0(A,M)$ and $C^1(A,k)$.\footnote{One can also check that $\wt{m} \odot f = (-1)^{\ol{m} \cdot \ol{f}} f \odot \wt{m}$, though the cup product of cochains is not graded-commutative in general. Then $\Cbul(A,M)$ also identifies as a differential graded left $\Cbul(A,k)$-super\-module with $\Cbul(A,k) \otimes M$.}

Continue to assume that $A$ is a finite-dimensional super\-bialgebra. The structure maps on $A$ induce by duality the structure of a super\-bialgebra on $A^*$. Write $\Delta_A$ and $\Delta_{A^*}$ for the coproducts on $A$ and $A^*$, respectively. Then given $f, g \in A^*$, the product $fg \in A^*$ satisfies $(fg)(a) = (f \otimes g) \circ \Delta_A(a)$ for all $a \in A$, while the coproduct $\Delta_{A^*}$ satisfies $\Delta_{A^*}(f)(a' \otimes a'') = f(a'a'')$ for all $a',a'' \in A$. Here we have identified $A^* \otimes A^*$ with $(A \otimes A)^*$ as in \eqref{eq:signsforproductofmaps}. The multiplicative identity in $A^*$ is the augmentation map $\ve: A \rightarrow k$, while the augmentation map on $A^*$ is defined by $f \mapsto f(1)$. Now $C^n(A,M) \cong M \otimes (A^*)^{\otimes n}$ as graded super\-spaces, and under this identification one has
\begin{equation} \label{eq:coproductdifferential}
\begin{split}
\partial(m \otimes f_1 \otimes \cdots \otimes f_n) &= \Delta_M(m) \otimes f_1 \otimes \cdots \otimes f_n + (-1)^{n+1} m \otimes f_1 \otimes \cdots \otimes f_n \otimes \ve \\
&\phantom{=} \textstyle + \sum_{i=1}^n (-1)^i m \otimes f_1 \otimes \cdots \otimes \Delta_{A^*}(f_i) \otimes \cdots \otimes f_n.
\end{split}
\end{equation}
Here $\Delta_M(m) = \partial(\wt{m}) + m \otimes \ve$, that is, $\Delta_M(m) \in M \otimes A^* \cong \Hom_k(A,M)$ is the function satisfying $\Delta_M(m)(a) = (-1)^{\ol{m} \cdot \ol{a}}a.m$. Under the identifications $\Cbul(A,M) \cong M \otimes (A^*)^{\otimes \bullet}$ and $\Cbul(A,k) \cong (A^*)^{\otimes \bullet}$, the right action of $\Cbul(A,k)$ on $\Cbul(A,M)$ is induced by concatenation.

\section{Cohomology for restricted Lie super\-algebras} \label{section:cohomologyrls}

In this section we present a modified version of May's construction of a free resolution of the trivial module for a graded restricted Lie algebra, and use it to show that the cohomology ring for a finite-dimensional restricted Lie super\-algebra is a finitely-generated algebra. If $\chr(k) = 0$, then the modified version of May's construction that we present in Section \ref{subsection:Koszulresolution} identifies with the standard analogue of the Koszul resolution for Lie super\-algebras; cf.\ \cite[\S 3]{Kang:2000}. Throughout this section, we attempt to remain consistent with the notation used by May \cite{May:1966}.

\subsection{The Koszul resolution for a Lie super\-algebra} \label{subsection:Koszulresolution}

Let $L = \Lzero \oplus \Lone$ be a finite-dimensional Lie super\-algebra. Once and for all, fix bases $\set{x_1,\ldots,x_s}$ and $\set{y_1,\ldots,y_t}$ for $\Lzero$ and $\Lone$, respectively. Set $\Ybar(L) = \Lambda(\Lzero) \otimes \Gamma(\Lone)$. We consider $\Ybar(L)$ as a graded-commutative graded-cocommutative graded super\-bialgebra with product and coproduct maps induced by the usual product and coproduct maps on $\Lambda(\Lzero)$ and $\Gamma(\Lone)$. The subspace $\Ybar_1(L)$ of $\Ybar(L)$ is naturally isomorphic as a super\-space to $L$; write $s: L \rightarrow \Ybar_1(L)$ for the natural isomorphism. Since $p \neq 2$, there also exists a natural isomorphism $\Gamma_2(\Lone) \cong S^2(\Lone)$, and hence an isomorphism $\Ybar_2(L) \cong \Lambda_s^2(L)$.

There exists a right action of $L$ on $\Ybar(L)$ such that $(z_1z_2).u = z_1(z_2.u) + (-1)^{\ol{u} \cdot \ol{z_2}} (z_1.u)z_2$ for all $z_1,z_2 \in \Ybar(L)$ and $u \in L$ (i.e., $u$ acts by right superderivations), and such that $\subgrp{x}.u = s([x,u])$ and $\gamma_r(y).u = \gamma_{r-1}(y)s([y,u])$ for all $x \in \Lzero$, $y \in \Lone$, and $r \in \N$. The right action of $L$ on $\Ybar(L)$ then extends to a right action of the universal enveloping super\-algebra $U(L)$. Thus, given $z_1,z_2 \in \Ybar(L)$ and $u \in U(L)$, one has $(z_1z_2).u = \sum (-1)^{\ol{u_1} \cdot \ol{z_2}} (z_1.u_1)(z_2.u_2)$. Here $\Delta(u) = \sum u_1 \otimes u_2$ is the coproduct of $u$ written as a sum of simple tensors. Now define $Y(L)$ to be the corresponding smash product algebra $U(L) \# \Ybar(L)$. Then $Y(L)$ is a graded super\-algebra with $U(L)$ concentrated in external degree $0$. As a graded super\-space and as a left $U(L)$-super\-module, $Y(L) = U(L) \otimes \Ybar(L)$. The coproducts on $U(L)$, $\Lambda(\Lzero)$, and $\Gamma(\Lone)$ induce on $Y(L)$ the structure of a graded-cocommutative graded super\-bialgebra. From now on we denote the product in $Y(L)$ by juxtaposition. Then $Y(L)$ is spanned by the monomials
\begin{equation} \label{eq:standardmonomials}
u\subgrp{x_{i_1} \cdots x_{i_b}} \gamma_{a_1}(y_1) \cdots \gamma_{a_t}(y_t),
\end{equation}
with $u \in U(L)$, $1 \leq i_1 < \cdots < i_b \leq s$, and $a_j \in \N$.

\begin{theorem} \label{theorem:Koszulresolution} \textup{\cite[Theorem 5]{May:1966}}
There exists a differential $d: Y(L) \rightarrow Y(L)$ making $Y(L)$ into a differential graded super\-bialgebra, and into a left $U(L)$-free resolution of the trivial module. Given $u \in U(L)$, $x \in \Lzero$, and $y \in \Lone$, the differential $d$ satisfies
\begin{align*}
d(u) &= 0,\\
d(\subgrp{x}) &= x, \quad \text{and} \\
d(\gamma_r(y)) &= y\gamma_{r-1}(y) - \tfrac{1}{2}\subgrp{[y,y]}\gamma_{r-2}(y).
\end{align*}
\end{theorem}

\begin{proof}
One can check that the given formulas are consistent with the relations in $Y(L)$, and hence that a derivation $d$ satisfying the given formulas does indeed exist. Next, writing $\Delta$ for the coproduct on $Y(L)$, one can check that the relations $d \circ d = 0$ and $\Delta \circ d = (d \otimes 1 + 1 \otimes d)\circ \Delta$ hold on the generators for $Y(L)$. Then by the derivation property of $d$, it follows that these relations hold on all of $Y(L)$. Finally, to show that $Y(L)$ defines a resolution of $k$, observe that there exists an increasing filtration $F_0 Y(L) \subset F_1 Y(L) \subset \cdots$ on $Y(L)$, with $F_n Y(L) = \bigoplus_{i+j = n} F_i U(L) \otimes \Ybar_j(L)$. Here $F_0 U(L) \subset F_1 U(L) \subset \cdots$ is the monomial-length filtration on $U(L)$, with $F_i U(L)$ spanned by all PBW-monomials in $U(L)$ of length at most $i$. Then $d: Y(L) \rightarrow Y(L)$ is filtration-preserving, and the associated graded complex $\gr Y(L) = \bigoplus_{n \geq 0} F_nY(L) / F_{n-1} Y(L)$ is isomorphic to $Y(\lab)$, the Koszul complex for the abelian Lie super\-algebra with the same underlying super\-space as $L$. Now by a standard filtration argument it suffices to show that $Y(\lab)$ is a resolution of $k$. If $\lab$ is one-dimensional, this can be checked in hand. In general, the claim follows via the K\"{u}nneth formula and from the observation that $Y(\lab)$ is isomorphic as a complex to the tensor product of $\dim \Lzero$ copies of $Y(k_{\ol{0}})$ with $\dim \Lone$ copies of $Y(k_{\ol{1}})$. Here $k_{\ol{0}}$ and $k_{\ol{1}}$ denote one-dimensional abelian Lie super\-algebras concentrated in $\Z_2$-degrees $\ol{0}$ and $\ol{1}$, respectively.
\end{proof}

The differential defined in Theorem \ref{theorem:Koszulresolution} is called the \emph{Koszul differential} on $Y(L)$.

\begin{remark} \label{remark:trivdif}
The augmentation map on $U(L)$ induces a natural projection $\ve: Y(L) \rightarrow \Ybar(L)$. If $L$ is abelian, it then follows that $\ve \circ d = 0$.
\end{remark}

\begin{remark}
The assignment $L \mapsto Y(L)$ is functorial in $L$. If $\varphi: L \rightarrow L'$ is a homomorphism of Lie super\-algebras, then the corresponding map $Y(\varphi): Y(L) \rightarrow Y(L')$ satisfies $Y(\varphi)(u) = \varphi(u)$, $Y(\varphi)(\subgrp{x}) = \subgrp{\varphi(x)}$, and $Y(\varphi)(\gamma_r(y)) = \gamma_r(\varphi(y))$ for $u \in U(L)$, $x \in \Lzero$, and $y \in \Lone$.
\end{remark}

\begin{remark}
The differential on $Y(L)$ is given by the following explicit formula:
\begin{align*}
d(u\subgrp{x_{i_1}\ldots x_{i_b}}&\gamma_{a_1}(y_1) \cdots \gamma_{a_t}(y_t)) \\
&=\textstyle \sum_{j=1}^b (-1)^{j+1} ux_{i_j}\subgrp{x_{i_1}\cdots \wh{x}_{i_j} \cdots x_{i_b}} \gamma_{a_1}(y_1) \cdots \gamma_{a_t}(y_t) \\
&\textstyle + \sum_{j=1}^t (-1)^b uy_j \subgrp{x_{i_1} \cdots x_{i_b}} \gamma_{a_1}(y_1) \cdots \gamma_{a_i-1}(y_i) \cdots \gamma_{a_t}(y_t) \\
&\textstyle + \sum_{1 \leq j < \ell \leq b} (-1)^{j+\ell} u\subgrp{[x_{i_j},x_{i_\ell}]x_{i_1} \cdots \wh{x}_{i_j} \cdots \wh{x}_{i_\ell} \cdots x_{i_b}} \gamma_{a_1}(y_1) \cdots \gamma_{a_t}(y_t) \\
&\textstyle +\sum_{j=1}^b \sum_{\ell=1}^t (-1)^j u\subgrp{x_{i_1} \cdots \wh{x}_{i_j} \cdots x_{i_b}} \gamma_1([x_{i_j},y_{\ell}]) \gamma_{a_1}(y_1) \cdots \gamma_{a_\ell-1}(y_\ell) \cdots \gamma_{a_t}(y_t) \\
&\textstyle -\sum_{1 \leq j < \ell \leq t} u \subgrp{[y_j,y_\ell]x_{i_1} \cdots x_{i_b}} \gamma_{a_1}(y_1) \cdots \gamma_{a_j-1}(y_j) \cdots \gamma_{a_\ell-1}(y_\ell) \cdots \gamma_{a_t}(y_t) \\
&\textstyle -\sum_{j=1}^t \frac{1}{2} u\subgrp{[y_j,y_j,]x_{i_1} \cdots x_{i_b}}\gamma_{a_1}(y_1) \cdots \gamma_{a_j-2}(y_j) \cdots \gamma_{a_t}(y_t).
\end{align*}
If $\chr(k) = 0$, then this agrees with the formula in \cite[\S3]{Kang:2000} after identifying $\gamma_a(y)$ with $\frac{1}{a!} y^a$.
\end{remark}

\subsection{The cohomology ring for a Lie super\-algebra} \label{subsection:ordinarycohomology}

Let $M$ be a left $L$-super\-module. Given $n \in \N$, set $C^n(L,M) = \Hom_{U(L)}(Y_n(L),M)$. Then as in \eqref{eq:cochaindifferential}, the Koszul differential on $Y(L)$ induces a differential $\partial$ on $C^\bullet(L,M)$, making $C^\bullet(L,M)$ into a cochain complex. Now the Lie super\-algebra cohomology group $\opH^n(L,M):= \opH^n(U(L),M)$ is the $n$-th cohomology group of the complex $\Cbul(L,M)$. Since the coproduct on $Y(L)$ defines a coassociative diagonal approximation $Y(L) \rightarrow Y(L) \otimes Y(L)$, the cup product makes $\Cbul(L,k)$ into a differential graded super\-algebra.

\begin{lemma} \label{lemma:cochainring}
$\Cbul(L,k)$ is naturally isomorphic as a graded super\-algebra to $\Lambda_s(L^*)$.
\end{lemma}

\begin{proof}
As a graded super\-space,
\[
C^n(L,k) \cong \Hom_k(\Ybar_n(L),k) \cong \bigoplus_{i+j=n} \Hom_k(\Lambda^i(\Lzero) \otimes \Gamma_j(\Lone),k).
\]
Then by dimension comparison, it suffices to show that $C^1(L,k) \cong L^*$ generates a free graded-commutative graded subsuper\-algebra of $\Cbul(L,k)$. Let $x_1^*,\ldots,x_s^*$ and $y_1^*,\ldots,y_t^*$ be the bases for ${\Lzero}^*$ and ${\Lone}^*$, respectively, that are dual to the fixed bases for $\Lzero$ and $\Lone$. We consider $x_1^*,\ldots,x_s^*$ as elements of $L^*$ via the projection $L \rightarrow \Lzero$, and similarly for $y_1^*,\ldots,y_t^*$. As elements of $C^1(L,k)$, the $x_i^*$ and $y_j^*$ are of internal degrees $\ol{0}$ and $\ol{1}$, respectively. Since the coproduct on $Y(L)$ is graded-cocommutative, it follows that the $x_i^*$ and $y_j^*$ generate a graded-commutative graded super\-algebra in $\Cbul(L,k)$. Given $a \in \N$, write ${y_j^*}^{\odot a}$ for the $a$-fold cup product $y_j^* \odot \cdots \odot y_j^*$. Then one can check for $a_1,\ldots,a_t \in \N$ and $1 \leq i_1 < \cdots < i_b \leq s$ that
\begin{equation} \label{eq:dualmonomial}
x_{i_1}^* \odot \cdots \odot x_{i_b}^* \odot {y_1^*}^{\odot a_1} \odot \cdots \odot {y_t^*}^{\odot a_t}
\end{equation}
evaluates to $\pm 1$ on the monomial $\subgrp{x_{i_1} \cdots x_{i_b}} \gamma_{a_1}(y_1) \cdots \gamma_{a_t}(y_t)$, and evaluates to $0$ on all other stan\-dard basis monomials in $\Ybar(L)$. Then the monomials of the form \eqref{eq:dualmonomial} form a linearly independent subset of $\Cbul(L,k)$, so it follows that the $x_i^*$ and $y_j^*$ generate a free subalgebra of $\Cbul(L,k)$.
\end{proof}

\begin{remark} \label{remark:braidedgraded}
Since $\Cbul(L,k)$ is a graded-commutative graded super\-algebra under the cup product, and since the product on $\Hbul(L,k)$ is induced by the product on $\Cbul(L,k)$, it follows that $\Hbul(L,k)$ is a graded-commutative graded super\-algebra. More generally, Mastnak et al.\ show in \cite[\S3]{Mastnak:2010} that if $R$ is a bialgebra in an abelian braided monoidal category $\mathcal{C}$, then the Hochschild cohomology ring of $R$ (with trivial coefficients) is a braided graded-commutative algebra in $\mathcal{C}$. Since the category of $k$-super\-spaces is an abelian braided monoidal category with braiding $V \otimes W \rightarrow W \otimes V$ defined by $v \otimes w \mapsto (-1)^{\ol{v} \cdot \ol{w}} w \otimes v$, this implies that the cohomology ring $\Hbul(A,k)$ of a $k$-super\-bialgebra $A$ is a graded-commutative graded super\-algebra. In particular, the subalgebra
\[
\opH(A,k) := \Hev(A,k)_{\ol{0}} \oplus \Hodd(A,k)_{\ol{1}}
\]
of $\Hbul(A,k)$ is commutative in the ordinary sense, and the elements in the complementary space
\[
\Hev(A,k)_{\ol{1}} \oplus \Hodd(A,k)_{\ol{0}}
\]
are nilpotent by the assumption that $\chr(k) \neq 2$.

A theory of cohomological support varieties for $A$ would typically start with the choice of a finitely-generated commutative subalgebra of $\Hbul(A,k)$. Thus, if $\Hbul(A,k)$, and hence also $\opH(A,k)$, is a finitely-generated $k$-algebra, then $\opH(A,k)$ would seem to be a natural choice for this purpose. This convention differs from those of Liu \cite[\S4]{Liu:2012} or Bagci \cite[\S4]{Bagci:2012a}. Because he adopts a different definition for the cohomology ring $\Hbul(A,k)$ than we do, Liu's conventions are equivalent to defining cohomological support varieties using the subalgebra $\Hev(A,k)_{\ol{0}}$, while Bagci defines them using the (supercommutative) subalgebra $\Hev(A,k) = \Hev(A,k)_{\ol{0}} \oplus \Hev(A,k)_{\ol{1}}$.
\end{remark}

Let $M$ be a left $L$-super\-module. Then reasoning as in Section \ref{subsection:barcomplex}, and using the fact that $\Ybar_n(L)$ is finite-dimensional, there exists an isomorphism of differential graded right $\Cbul(L,k)$-super\-modules $\Cbul(L,M) \cong M \otimes \Cbul(L,k)$. Here $m \in M$ identifies with the homomorphism $\wt{m} \in C^0(L,M)$ satisfying $\wt{m}(u) = (-1)^{\ol{u} \cdot \ol{m}}u.m$ for $u \in U(L)$. Given $f \in C^n(L,k)$, the tensor product $m \otimes f \in M \otimes C^n(L,k)$ identifies with the cup product $\wt{m} \odot f$. The differential $\partial: M \rightarrow M \otimes C^1(L,k) \cong M \otimes L^*$ satisfies $\partial(\wt{m}) = \sum_i \wt{m}_i \otimes f_i$, where the $m_i \in M$ and $f_i \in L^*$ are such that $\sum_i f_i(z).m_i = (-1)^{\ol{z} \cdot \ol{m}}z.m$ for each $z \in L$.

Since $\Cbul(L,k) \cong \Lambda_s(L^*)$ is generated as a graded super\-algebra by $C^1(L,k) \cong \Lambda_s^1(L^*) \cong L^*$, the differential on $\Cbul(L,M)$ is completely determined by its restrictions to $M \cong C^0(L,M)$ and $C^1(L,k)$. As remarked above, $\Ybar_2(L)$ is naturally isomorphic to $\Lambda_s^2(L)$ by the assumption $\chr(k) = p \neq 2$, and one can check that the Lie bracket $[\cdot,\cdot]: L \otimes L \rightarrow L$ factors through a linear map $\Lambda_s^2(L) \rightarrow L$. Then $\partial^1: C^1(L,k) \rightarrow C^2(L,k)$ identifies with a linear map $\Lambda_s^1(L)^* \rightarrow \Lambda_s^2(L)^*$. Under this identification, one can check that if $z_1,z_2 \in L$, then $\partial^1(f)(z_1 \wedge z_2) = f([z_1,z_2])$, i.e., that $\partial^1$ identifies with the transpose of the Lie bracket.

\begin{remark} \label{remark:exteriorduality}
The coproduct on $\Lambda_s(L)$ induces an algebra structure on $\bigoplus_{n \in \N} \Lambda_s^n(L)^*$, the graded dual of $\Lambda_s(L)$. Moreover, the identification $\Lambda_s^1(L^*) \cong \Lambda_s^1(L)^*$ induces a natural homomorphism of graded super\-algebras $\varphi: \Lambda_s(L^*) \rightarrow \bigoplus_{n \in \N} \Lambda_s^n(L)^*$. If $\chr(k) = 0$, then $\varphi$ is an isomorphism, but if $\chr(k) = p > 0$, the algebras are no longer isomorphic. Indeed, if $p > 0$, then the odd elements in $\Lambda_s^1(L)^*$ are nilpotent of degree $p$, whereas the odd elements in $\Lambda_s^1(L^*)$ are not nilpotent. Still, even if $\chr(k) = p > 0$, $\varphi$ induces an isomorphism in external degrees $< p$.
\end{remark}

\begin{theorem} \label{theorem:fgforU(L)}
Let $k$ be a field of characteristic $p > 2$, let $L$ be a finite-dimensional Lie super\-algebra over $k$, and let $M$ be a left $L$-supermodule. Let $S({\Lone}^*)^p$ be the subalgebra of $S({\Lone}^*)$ generated by all $p$-th powers in $S({\Lone}^*)$. Then $\Hbul(L,M)$ is right noetherian over $\Hbul(L,k)$, and $\Hbul(L,k)$ is right noetherian over $S({\Lone}^*)^p$. In particular, $\Hbul(L,k)$ is a finitely-generated graded super\-algebra.
\end{theorem}

\begin{proof}
Since $\Cbul(L,M) \cong M \otimes \Cbul(L,k)$ as a differential graded right $\Cbul(L,k)$-super\-module, and since $\Cbul(L,k) \cong \Lambda_s(L^*) = \Lambda({\Lzero}^*) \otimes S({\Lone}^*)$ as a graded-commutative graded super\-algebra, it is clear that $\Cbul(L,M)$ and $\Cbul(L,k)$ are each finitely-generated as right modules over the commutative ring $S({\Lone}^*)^p$. Now let $f \in S({\Lone}^*)$ be of external degree $m$, hence of internal degree $\ol{m}$. Then $\partial(f)$ is of external degree $m-1$ and of internal degree $\ol{m}$, so that
\[
f \odot \partial(f) = (-1)^{m \cdot (m-1)} (-1)^{\ol{m} \cdot \ol{m}} \partial(f) \odot f = (-1)^m \partial(f) \odot f
\]
by the graded-commutativity of $\Cbul(L,k)$. From this and from the derivation property of $\partial$, it follows that $\partial(f^{\odot p}) = p \cdot \partial(f) \odot f^{\odot (p-1)} = 0$, and hence that $S({\Lone}^*)^p$ consists of cocycles in $\Cbul(L,k)$. Then $\Hbul(L,k)$ and $\Hbul(L,M)$ inherit the structure of right $S({\Lone}^*)^p$-modules, and from this observation the theorem then follows. 
\end{proof}

\begin{remark} \label{remark:cohomologyabelianL}
The cohomology ring $\Hbul(L,k)$ need not be finite-dimensional. If $L$ is abelian, it follows from Remark \ref{remark:trivdif} that the differential on $\Cbul(L,k)$ is trivial, and hence that $\Hbul(L,k)$ is naturally isomorphic to $ \Lambda_s(L^*)$. In particular, if $L = \Lone$, then $\Hbul(L,k) \cong S(L^*)$. Here $U(L) \cong \Lambda(L)$, so the isomorphism $\Hbul(L,k) \cong S(L^*)$ expresses the classical fact that the cohomology ring of an exterior algebra is a polynomial algebra generated in cohomological degree one \cite{Priddy:1970}.
\end{remark}

\begin{remark} \label{remark:MPSW}
The isomorphism $\Hbul(U(L),k) = \Hbul(L,k) \cong S(L^*)$ in the case $L = \Lone$ shows that the published statement of \cite[Theorem 4.1]{Mastnak:2010} is incorrect, or at least unclear, if in the notation used there $N_i = 2$ for some $1 \leq i \leq \theta$. Indeed, taking $\theta = \dim L$, $q_{ij} = -1$ for $1 \leq i < j \leq \theta$, and $N_i = 2$ for $1 \leq i \leq \theta$, the algebra $S$ in \cite[\S 4]{Mastnak:2010} is just the universal enveloping super\-algebra $U(L)$. Now the calculation $\Hbul(U(L),k) \cong S(L^*)$ shows that $\Hbul(S,k)$ is an integral domain, whereas \cite[Theorem 4.1]{Mastnak:2010} seems to assert that $\opH^1(S,k)$ is spanned by square-zero elements $\eta_1,\ldots,\eta_\theta$, an interpretation supported by \cite[Remark 4.2]{Mastnak:2010}. This inter\-pretation is also consistent with taking $q_{ii} = 1$ for each $1 \leq i \leq \theta$ in the definition of $S$. In a private communication, Sarah Witherspoon indicated that this interpretation of the relations in \cite[(4.4)]{Mastnak:2010} is correct provided that $N_i \neq 2$ for all $i$, but that if $N_i = 2$, the relations should be modified to indicate that $\eta_i^2$ is a scalar multiple of the cohomology class they denote by $\xi_i$. Thus, if $N_i = 2$, one should take $q_{ii} = -1$ in the definition of $S$ (in which case the relation $x_i^{N_i} = 0$ is superfluous because $\chr(k) \neq 2$).
\end{remark}

\subsection{The Koszul resolution for a restricted Lie super\-algebra} \label{subsection:Koszulresolutionrls}

From now on assume that $\chr(k) = p > 2$, and that $L$ is a finite-dimensional restricted Lie super\-algebra over $k$. Denote the restriction mapping on $\Lzero$ by $x \mapsto x^{[p]}$, and let $V(L) = U(L)/\subgrp{x^p - x^{[p]}: x \in \Lzero}$ be the restricted enveloping super\-algebra of $L$. Given $x \in \Lzero$, the (right) action of $x^p - x^{[p]} \in U(L)$ on $\Ybar(L)$ is trivial. Then the right action of $U(L)$ on $\Ybar(L)$ factors through an action of $V(L)$. Now define $W(L)$ to be the corresponding smash product algebra $V(L) \# \Ybar(L)$. Equivalently, $W(L)$ is the quotient of $Y(L)$ by the two-sided ideal $\subgrp{x^p - x^{[p]}: x \in \Lzero}$. The algebra $W(L)$ inherits from $Y(L)$ the structure of a differential graded super\-bialgebra. By abuse of notation, we denote the induced differential on $W(L)$ by $d$. As a graded super\-space and as a left $V(L)$-super\-module, $W(L) = V(L) \otimes \Ybar(L)$.

Let $\Gamma'(\Lzero)$ be the divided polynomial algebra $\Gamma(\Lzero)$ with all external degrees multiplied by~$2$. Then $\Gamma'(\Lzero)$ is generated by the homogeneous elements $\gamma_r'(x)$ for $r \in \N$ and $x \in \Lzero$, with $\gamma_r'(x)$ of bidegree $(2r,\ol{0})$. Given $n \in \N$, denote the external degree-$n$ component of $\Gamma'(\Lzero)$ by $\Gamma_n'(\Lzero)$. Then $\Gamma_n'(\Lzero) = 0$ if $n$ is odd. Now set $X(L) = W(L) \otimes \Gamma'(\Lzero)$. As a graded $V(L)$-super\-module, $X(L) = V(L) \otimes \Xbar(L)$, where $\Xbar(L) = \Ybar(L) \otimes \Gamma'(\Lzero)$. Considering $\Gamma'(\Lzero)$ as a differ\-ential graded super\-algebra with trivial differential, we consider $X(L)$ as the tensor product of complexes. Then given $w \in W(L)$ and $b \in \Gamma'(\Lzero)$, the differential $d: X(L) \rightarrow X(L)$ is defined by $d(w \otimes b) = d(w) \otimes b$.

As in \eqref{eq:standardmonomials}, $X(L)$ is spanned by monomials
\begin{equation} \label{eq:X(L)standardmonomials}
v\subgrp{x_{i_1} \cdots x_{i_b}} \gamma_{a_1}(y_1) \cdots \gamma_{a_t}(y_t) \gamma_{c_1}'(x_1) \cdots \gamma_{c_s}'(x_s),
\end{equation}
with $v \in V(L)$, $1 \leq i_1 < \cdots < i_b \leq s$, and $a_j,c_j \in \N$. In particular, let $\mathcal{B} = \set{x_1,\ldots,x_s,y_1,\ldots,y_t}$ be the fixed basis for $L$, and let $\mathcal{P} \subset V(L)$ be the set of PBW monomial basis vectors in $V(L)$ arising from the particular ordered basis $\mathcal{B}$. Then the collection of monomials of the form \eqref{eq:X(L)standardmonomials} with $v \in \mathcal{P}$ forms a basis for $X(L)$, which we call the standard monomial basis for $X(L)$ with respect to $\mathcal{B}$. We call the vectors with $v = 1$ the standard monomial basis vectors for $\Xbar(L)$.

Given $n \in \Z$, set $R^n = \bigoplus_{i \geq 0} \Hom_k(\Gamma_i'(\Lzero),Y_{i-n}(\Lzero))$, and set $R = \bigoplus_{n \in \Z} R^n$. Then the algebra structure on $Y(\Lzero)$ together with the coproduct on $\Gamma'(\Lzero)$ induces on $R$ the structure of a differential graded super\-algebra. We denote the product of $r,r' \in R$ by $r \cup r'$. The differential on $R$ is defined by $\partial(r) = d \circ r$, where $d$ is the Koszul differential on $Y(\Lzero)$. The algebra $W(L)$ is a differential graded right $Y(\Lzero)$-super\-module via the natural maps $Y(\Lzero) \twoheadrightarrow W(\Lzero) \hookrightarrow W(L)$ and the right action of $W(L)$ on itself. Write $\sigma: W(L) \otimes Y(\Lzero) \rightarrow W(L)$ for the right action of $Y(\Lzero)$ on $W(L)$. Now given $r \in R$, $w \in W(L)$, and $b \in \Gamma'(\Lzero)$, define
\begin{equation} \label{eq:Raction}
(w \otimes b) \cap r = (-1)^{\deg(r) \deg(w \otimes b)} (\sigma \otimes 1) \circ (1 \otimes r \otimes 1) \circ (1 \otimes \Delta)(w \otimes b).
\end{equation}
Here $\Delta: \Gamma'(\Lzero) \rightarrow \Gamma'(\Lzero) \otimes \Gamma'(\Lzero)$ is the coproduct on $\Gamma'(\Lzero)$. Reindexing the $\Z$-grading on $R$ by $R_n := R^{-n}$, one can check that \eqref{eq:Raction} makes $X(L)$ into a differential graded right $R$-module. In other words, $d( (w \otimes b) \cap r) = (d(w) \otimes b) \cap r + (-1)^{\deg(w \otimes b)} (w \otimes b) \cap \partial(r)$.

Now let $t \in R^1$, and define $d_t: X(L) \rightarrow X(L)$ by
\[
d_t(w \otimes b) = d(w) \otimes b + (-1)^{\deg(w \otimes b)}(w \otimes b) \cap t.
\]
Then a straightforward calculation shows that $d_t^2(w \otimes b) = (w \otimes b) \cap (\partial(t) - t \cup t)$. We call $t$ a \emph{twisting cochain} if $\partial(t) - t \cup t$ acts as zero on $X(L)$.\footnote{May requires $\partial(t) - t \cup t$ to actually be the zero element in $R$, but if one adopts this more restrictive definition, the formula given in \cite[Lemma 6]{May:1966} does not then define a twisting cochain.} Writing $\pi$ for the natural projection $Y(\Lzero) \rightarrow W(\Lzero)$, $t$ is a twisting cochain if and only if $\pi \circ (\partial(t) - t \cup t) = 0$. If $t \in R^1$ is a twisting cochain, then $d_t$ makes $X(L)$ into a chain complex of free left $V(L)$-super\-modules. In the lemma below, we write $t_m$ for the restriction of $t$ to $\Gamma_m'(\Lzero)$, and consider $t_m$ as an element of $R^1$ with $t_m(\Gamma_n'(\Lzero)) = 0$ for all $n \neq m$. Then $t \cup t = \sum_{n \geq 0} \sum_{i+j = n} t_i \cup t_j$. Of course, $t_m = 0$ if $m$ is odd, because then $\Gamma_m'(\Lzero) = 0$, and $t_0 = 0$ because $Y_{-1}(\Lzero) = 0$.

Recall that since $\set{x_1,\ldots,x_s}$ is a basis for $\Lzero$, $\{\gamma_1'(x_1),\ldots,\gamma_1'(x_s)\}$ is a basis for $\Gamma_2'(\Lzero)$. Recall also the filtration $F_0 Y(\Lzero) \subset F_1 Y(\Lzero) \subset \cdots $ on $Y(\Lzero)$ defined in the proof of Theorem \ref{theorem:Koszulresolution}.

\begin{lemma} \label{lemma:twistingcochain} \textup{\cite[Lemma 6]{May:1966}}
There exists a twisting cochain $t$ such that
\[
t_2(\gamma_1'(x_i)) = x_i^{p-1}\subgrp{x_i} - \subgrp{x_i^{[p]}},
\]
and such that for $n > 1$, $t_{2n}(\Gamma_{2n}'(\Lzero)) \subset F_{np-1} Y_{2n-1}(\Lzero)$.\footnote{The formula defining $t_2$ in \cite[Lemma 6]{May:1966} is not linear in the symbol $\wt{y}$, and so does not make sense as written. Instead, the formula should be interpreted as defining the action of $t_2$ on a basis for $\Gamma_2'(\Lzero)$.}
\end{lemma}

\begin{proof}
The maps $t_{2n}$ for $n \geq 1$ are defined by induction on $n$. First define $t_2$ as in the statement of the lemma. Then $\pi \circ d \circ t_2(\gamma_1'(x_i)) = x_i^p - x_i^{[p]} = 0$, which is zero in $W(\Lzero)$. Next set $r_2 = t_2 \cup t_2$, and consider $d \circ r_2 = d \circ (t_2 \cup t_2) = (d \circ t_2) \cup t_2 - t_2 \cup (d \circ t_2)$. Since $\Gamma_4'(\Lzero)$ is spanned by monomials of the forms $\gamma_2'(x_i)$ and $\gamma_1'(x_i)\gamma_1'(x_j)$ with $i \neq j$, and since $\Gamma'(\Lzero)$ is cocommutative, it follows that if $b \in \Gamma_4'(\Lzero)$, then $d(r_2(b)) \in Y_1(\Lzero)$ is a sum of commutators of the form $[x^p - x^{[p]},z]$ with $x \in \Lzero$ and $z \in Y_1(\Lzero)$. Then $d(r_2(b)) = 0$ because $x^p - x^{[p]}$ is central in $Y(\Lzero)$. Since the Koszul differential on $Y(\Lzero)$ is exact, this implies that $\im(r_2) \subset \im(d_3)$, and hence that it is possible to non-uniquely define a linear map $t_4: \Gamma_4'(\Lzero) \rightarrow Y_3(\Lzero)$ such that $d \circ t_4 = r_2$. Moreover, the reader can check that $\im(r_2) \subset F_{2p-1} Y_2(\Lzero)$, so it is furthermore possible, by the exactness of the associated graded complex $\gr Y(\Lzero)$, to define $t_4$ so that $\im(t_4) \subset F_{2p-1} Y_3(\Lzero)$.

Now let $n > 2$, and assume by way of induction that $t_{2i}$ has been defined for $1 < i < n$ so that $d \circ t_{2i} = r_i := \sum_{j=1}^{i-1} t_{2j} \cup t_{2(i-j)}$ and so that $\im(t_{2i}) \subset F_{ip-1} Y_{2i-1}(\Lzero)$. We wish to define $t_{2n}$ so that it satisfies the same conditions with $i = n$. Observe that
\begin{align*}
d \circ r_n &= \textstyle \sum_{i=1}^{n-1} (d \circ t_{2i}) \cup t_{2(n-i)} - t_{2i} \cup (d \circ t_{2(n-i)}) \\
&= \textstyle (d \circ t_2) \cup t_{2(n-1)} - t_{2(n-1)} \cup (d \circ t_2) + \left( \sum_{i=2}^{n-1} r_i \cup t_{2(n-i)}\right) - \left( \sum_{i=1}^{n-2} t_{2i} \cup r_{n-i} \right).
\end{align*}
The last two terms in this equation sum to zero, so $d \circ r_n = (d \circ t_2) \cup t_{2(n-1)} - t_{2(n-1)} \cup (d \circ t_2)$. Now an argument like that in the case $n = 2$ shows that $d \circ r_n = 0$, so we conclude that it is possible to non-uniquely define $t_{2n}: \Gamma_{2n}'(\Lzero) \rightarrow Y_{2n-1}(\Lzero)$ such that $d \circ t_{2n} = r_n$. Moreover, the induction hypothesis also implies that $\im(r_n) \subset F_{np-1} Y_{2n-2}(\Lzero)$, so we may also define $t_{2n}$ so that $\im(t_{2n}) \subset F_{np-1} Y_{2n-1}(\Lzero)$. Now taking $t = \sum_{i \geq 1} t_{2i}$, one has $\pi \circ (\partial(t) - t \cup t) = 0$.
\end{proof}

\begin{remark} \label{remark:tnzero}
We may assume that $t_{2n}(b) = 0$ whenever $r_n(b) = 0$. In particular, if $L$ is abelian with trivial restriction mapping, then $r_2 = t_2 \cup t_2 = 0$, so it is possible to take $t_4 = 0$. Now arguing by induction on $n$, it follows that it is possible to take $t_{2n}  = 0$ for all $n \geq 2$. From now on we adopt this convention whenever $L$ is abelian with trivial restriction mapping.
\end{remark}

\begin{theorem} \textup{\cite[Theorem 8]{May:1966}} \label{theorem:X(L)resolution}
Let $t$ be a twisting cochain as in Lemma \ref{lemma:twistingcochain}. Then $d_t$ makes $X(L)$ into a $V(L)$-free resolution of $k$.
\end{theorem}

\begin{proof}
The strategy is similar to that for the proof of Theorem \ref{theorem:Koszulresolution}. First define a filtration on $\Xbar(L)$ by $F_j\Xbar(L)  = \bigoplus_{m + np \leq j} \Ybar_m(L) \otimes \Gamma_{2n}'(\Lzero)$, and then define a filtration on $X(L)$ by
\begin{equation} \label{eq:X(L)filtration} \textstyle
F_n X(L) = \bigoplus_{i+j = n} F_i V(L) \otimes F_j \Xbar(L).
\end{equation}
Here $F_0 V(L) \subset F_1 V(L) \subset \cdots$ is the monomial-length filtration on $V(L)$ induced by the corres\-ponding filtration on $U(L)$. One can check that $d_t$ is filtration-preserving, and that the associated graded complex $\gr X(L) = \bigoplus_{n \geq 0} F_n X(L)/F_{n-1} X(L)$ is isomorphic as a complex to $X(\lab)$, where $\lab$ is the abelian restricted Lie super\-algebra with trivial restriction mapping and the same underlying super\-space as $L$. In particular, $\gr V(L) = \bigoplus_{n \geq 0} F_n V(L)/ F_{n-1} V(L)$ is naturally isomorphic as a Hopf super\-algebra to $\vlab$. The external grading on $\vlab$ arises from considering the subspace $\lab \subset \vlab$ as a graded super\-space concentrated in external degree $1$. Now one can show that $X(\lab)$ is isomorphic as a complex to the tensor product of $\dim \Lzero$ copies of $X(k_{\ol{0}})$ with $\dim \Lone$ copies of $Y(k_{\ol{1}})$, where $k_{\ol{0}}$ and $k_{\ol{1}}$ denote one-dimensional trivial restricted Lie super\-algebras concentrated in $\Z_2$-degrees $\ol{0}$ and $\ol{1}$, respectively. This reduces the proof of the theorem to showing that $X(k_{\ol{0}})$ is a resolution of $k$, which can be checked by hand; see \cite[\S6]{May:1966}.
\end{proof}

\begin{remark} \label{remark:notnatural}
Unlike $Y(L)$, the resolution $X(L)$ as constructed here is not functorial in $L$, since the definition of $d_t$ depends on the choice of a particular basis for $L$.
\end{remark}

From now on we fix a twisting cochain $t$ as constructed in the proof of Lemma \ref{lemma:twistingcochain}.

\subsection{The May spectral sequence} \label{subsection:Mayspecseq}

In this section set $V = V(L)$, and let $F_0 V \subset F_1 V \subset \cdots $ be the monomial-length filtration on $V$. Note that $F_i V = V$ if $i \geq (p-1)(\dim \Lzero) + \dim \Lone$. The filtration on $V$ induces an increasing filtration on the bar complex $B(V)$, with
\begin{equation} \label{eq:barfiltration}
F_i B_n(V) = \sum_{j_0 + j_1 + \cdots + j_n \leq i} F_{j_0} V \otimes F_{j_1} V \otimes \cdots \otimes F_{j_n} V.
\end{equation}
Similarly, $F$ induces an increasing filtration on $B(V) \otimes B(V)$. The differential and diagonal approx\-imation on $B(V)$ each preserve the filtration, and the associated graded complex $\gr B(V)$ is isomorphic to $B(\vlab)$, the bar complex for the abelian restricted Lie super\-algebra with trivial restriction mapping and the same underlying super\-space as $L$.

Let $M$ be a left $V$-super\-module. Define a decreasing filtration on $\Cbul(V,M)$ by
\begin{equation} \label{eq:cobarfiltration}
F^i C^n(V,M) = \Hom_V(B_n(V)/ (V . F_{i-1}\Bbar_n(V)) ,M).
\end{equation}
In other words, $F^i C^n(V,M)$ consists of the $V$-module homomorphisms $B_n(V) \rightarrow M$ vanishing on the $V$-submodule of $B_n(V)$ generated by $F_{i-1}\Bbar_n(V) := \Bbar_n(V) \cap F_{i-1} B_n(V)$. Then $F^0 C^n(V,M) = C^n(V,M)$ and $F^i C^n(V,M) = 0$ for $i  > n \cdot [(p-1)(\dim \Lzero) + \dim \Lone]$. In the associated graded complex $\gr \Cbul(V,M) = \bigoplus_{i \geq 0} F^i \Cbul(V,M) / F^{i+1} \Cbul(V,M)$, one has
\[
\gr_i C^n(V,M) = F^i C^n(V,M) / F^{i+1} C^n(V,M) \cong \Hom_k(F_i \Bbar_n(V)/F_{i-1} \Bbar_n(V),M).
\]
Regarding $M$ as a trivial $\vlab$-module, $\gr \Cbul(V,M)$ is isomorphic as a complex to $\Cbul(\vlab,M)$. The filtrations on $\Cbul(V,M)$ and $\Cbul(V,k)$ are multiplicative with respect to the cup product action of $\Cbul(V,k)$ on $\Cbul(V,M)$. Then passing to the associated graded complexes, the right action of $\Cbul(V,k)$ on $\Cbul(V,M)$ descends to the right action of $\Cbul(\vlab,k)$ on $\Cbul(\vlab,M)$.

By \cite[Theorem 2.6]{McCleary:2001}, there exists a spectral sequence $E_r(M) \Rightarrow \Hbul(V,M)$ with
\begin{equation} \label{eq:Mayspecseq}
E_0^{i,j}(M) = \gr_i C^{i+j}(V,M), \quad \text{and} \quad
E_1^{i,j}(M) \cong \opH^{i+j}(\gr_i \Cbul(V,M)).
\end{equation}
This spectral sequence is similar to that constructed in \cite[\S4]{May:1966}, though May's construction is based on a filtration arising from the powers of the augmentation ideal in $V$. Still, we refer to $E(M)$ and $E(k)$ as May spectral sequences. The cup product makes $E(k)$ into a spectral sequence of algebras, and $E(M)$ into a spectral sequence of (right) graded super\-modules over $E(k)$. Ignoring the bigradings, there exist global isomorphisms
\begin{equation}
\begin{split}
E_0(M) &\cong \Cbul(\vlab,M) \cong M \otimes \Cbul(\vlab,k), \quad \text{and} \\
E_1(M) &\cong \Hbul(\vlab,M) \cong M \otimes \Hbul(\vlab,k),
\end{split}
\end{equation}
which are compatible with the right action of $E(k)$ on $E(M)$. Keeping track of the bigrading, one has $E_0^{i,j}(M) \cong M \otimes E_0^{i,j}(k)$, and $E_1^{i,j}(M) \cong M \otimes E_1^{i,j}(k)$.

Recall that the super\-bialgebra structure on $V$ induces a corresponding structure on $V^*$.

\begin{lemma} \label{lemma:dualfiltration}
Write $I_\ve$ for the augmentation ideal in $V^*$. Then $I_\ve^n = (V/F_{n-1}V)^*$.
\end{lemma}

\begin{proof}
The restricted enveloping super\-algebra $V$ is a connected Hopf super\-algebra, and the monom\-ial-length filtration on $V$ is the same as the coradical filtration on $V$; cf.\ \cite[\S 5.5]{Montgomery:1993}. Then the claim follows as in \cite[5.2.9]{Montgomery:1993}.
\end{proof}

Recall that $\gr V = \bigoplus_{n \geq 0} F_n V / F_{n-1}V$ is naturally isomorphic as a graded super\-algebra to $\vlab$, where the external grading on $\vlab$ is obtained by considering the subspace $\lab$ as concentrated in external degree $1$. The algebra $V^*$ is filtered by the powers of its augmentation ideal, and the associated graded algebra is defined by $\gr V^* = \bigoplus_{n \geq 0} I_\ve^n / I_\ve^{n+1}$. Applying the lemma, it follows that $\gr V^*$ and $(\gr V)^* \cong \vlab^*$ are naturally isomorphic as Hopf super\-algebras. Identifying $C^n(V,M)$ with $M \otimes C^n(V,k) \cong M \otimes (V^*)^{\otimes n}$, one has
\begin{equation} \label{eq:Vcobarfiltration}
F^i C^n(V,M) = \sum_{i_1 + \cdots + i_n \geq i} M \otimes I_\ve^{i_1} \otimes \cdots \otimes I_\ve^{i_n}.
\end{equation}
Then $\gr C^n(V,M) \cong M \otimes (\gr V^*)^{\otimes n}$, with
\[
\gr_i C^n(V,M) \cong \sum_{i_1 + \cdots + i_n = i} M \otimes (\gr_{i_1} V^*) \otimes \cdots \otimes (\gr_{i_n} V^*).
\]

\begin{proposition} \label{proposition:E1page}
Suppose $k$ is a perfect field of characteristic $p > 2$. Then there exists a natural isomorphism of graded super\-algebras $\Hbul(\vlab,k) \cong \Lambda_s(L^*) \otimes S'({\Lzero}^*)$. Here $S'({\Lzero}^*)$ is the poly\-nomial algebra on ${\Lzero}^*$ with ${\Lzero}^*$ concentrated in external degree $2$. In the spectral sequence \eqref{eq:Mayspecseq}, one has for each $i,j \in \Z$ natural identifications
\begin{equation} \label{eq:E1ij}
E_1^{i,j}(k) = \bigoplus \Lambda_s^a(L^*) \otimes S^b({\Lzero}^*),
\end{equation}
where the sum is taken over all $a,b \in \N$ such that $a+pb = i$ and $a+2b = i+j$.
\end{proposition}

\begin{proof}
Proving the proposition amounts to computing the cohomology of the complex $E_0(k) \cong \Cbul(\vlab,k)$, and in the process keeping track of the extra external grading that arises when the subspace $\lab \subset \vlab$ is considered as a graded super\-space concentrated in external degree $1$. Then without loss of generality we may assume that $L$ is abelian with trivial restriction map.

For the time being forget the extra grading on $V(L)$, that is, consider $V(L)$ as a graded super\-algebra concentrated in external degree $0$. Then there exists a natural isomorphism of Hopf super\-algebras $V \cong V(\Lzero) \otimes U(\Lone) = V(\Lzero) \otimes \Lambda(\Lone)$. We have $\Hbul(U(\Lone),k) \cong S({\Lone}^*)$ by Remark \ref{remark:cohomologyabelianL}, and $\Hbul(V(\Lzero),k)$ is naturally isomorphic to $\Lambda({\Lzero}^*) \otimes S'({\Lzero}^*)$ by \cite[I.4.27]{Jantzen:2003} (specifically, by the case $r=1$ of part (b) of the proposition, which assumes that $k$ is perfect). Then
\[
\Hbul(V,k) \cong \Lambda({\Lzero}^*) \otimes S({\Lone}^*) \otimes S'({\Lzero}^*) = \Lambda_s(L^*) \otimes S'({\Lzero}^*)
\]
as a graded-commutative graded super\-algebra by \cite[Theorem 3.7]{Bergh:2008}.

To identify the summand $E_1^{i,j}(k)$ of $E_1(k) \cong \Hbul(V,k)$, we identify cochain representatives in $E_0(k) \cong \Cbul(V,k)$ for the generators of $\Hbul(V,k)$, and then determine the extra external degrees of those representatives. The chain complexes $B(V) \cong B(V(\Lzero) \otimes \Lambda(\Lone))$ and $B(V(\Lzero)) \otimes B(\Lambda(\Lone))$ are chain equivalent via the chain homomorphism $\psi: B(V(\Lzero) \otimes \Lambda(\Lone)) \rightarrow B(V(\Lzero)) \otimes B(\Lambda(\Lone))$ defined in \cite[p.\ 221]{Cartan:1999}. Then \cite[Theorem 3.7]{Bergh:2008} expresses the fact that the chain map
\begin{multline} \label{eq:psistar}
\Cbul(V(\Lzero),k) \otimes \Cbul(\Lambda(\Lone),k) \cong \Hom_{V(\Lzero) \otimes \Lambda(\Lone)}(B_\bullet(V(\Lzero)) \otimes B_\bullet(\Lambda(\Lone)),k) \\
\stackrel{\psi^*}{\longrightarrow} \Hom_V(B(V),k) = \Cbul(V,k)
\end{multline}
induced by $\psi$ produces an isomorphism in cohomology, and that this isomorphism is compatible with the cup product in cohomology. Then to determine the desired cochain representatives, it suffices to find representatives in $\Cbul(V(\Lzero),k)$ and $\Cbul(\Lambda(\Lone),k)$ for the generators of the cohomology rings $\Hbul(V(\Lzero),k)$ and $\Hbul(\Lambda(\Lone),k)$, and then to find the images of the representatives under \eqref{eq:psistar}.

We have $C^1(\Lambda(\Lone),k) \cong \Lambda^1(\Lone)^* \cong {\Lone}^*$, and then it is well-known that this space consists of cocycles whose images in $\Hbul(\Lambda(\Lone),k)$ span $\opH^1(\Lambda(\Lone),k) \cong S^1({\Lone}^*) \cong {\Lone}^*$, and hence generate $\Hbul(\Lambda(\Lone),k)$; see \cite[Example 2.2(2)]{Priddy:1970}. To identify cochain representatives in $\Cbul(V(\Lzero),k)$, we suitably interpret \cite[I.4.20--I.4.27]{Jantzen:2003}. In this context, $\Hbul(V(\Lzero),k)$ is the cohom\-ology ring for the first Frobenius kernel of the affine group scheme $(\Lzero)_a$. Then $k[(\Lzero)_a] = S({\Lzero}^*)$, and $V(\Lzero)^* \cong k[(\Lzero)_{a,1}] \cong S({\Lzero}^*)/\subgrp{f^p: f \in {\Lzero}^*}$. Moreover, under this identification, the extra external grading on $V(\Lzero)^*$ is induced by the polynomial-degree grading on $S({\Lzero}^*)$. Now the discussion of \cite[I.4.21--I.4.27]{Jantzen:2003} shows that the subspace ${\Lzero}^* \cong \gr_1 V(\Lzero)^*$ of $C^1(V(\Lzero),k) \cong V(\Lzero)^*$ consists of cocycles whose images in $\Hbul(V(\Lzero),k)$ generate the sub\-algebra $\Lambda({\Lzero}^*)$ of $\Hbul(V(\Lzero),k)$. Next, given $f \in V(\Lzero)^* \cong C^1(V(\Lzero),k)$, set
\begin{equation} \label{eq:beta} \textstyle
\beta(f) = \sum_{i=1}^{p-1} \cbinom{p}{i} f^i \otimes f^{p-i} \in V(\Lzero)^* \otimes V(\Lzero)^* \cong C^2(V(\Lzero),k).
\end{equation}
Here $\cbinom{p}{i} \in \Z$ is defined for $0 < i < p$ by $\cbinom{p}{i} = \binom{p}{i}/p$. Then $\beta$ takes cocycles to cocycles, and the induced function $\betabar : \opH^1(V(\Lzero),k) \rightarrow \opH^2(V(\Lzero),k)$ is a semilinear map that takes the generating space ${\Lzero}^*$ for $\Lambda({\Lzero}^*) \subset \Hbul(V(\Lzero),k)$ to the generating space for $S'({\Lzero}^*) \subset \Hbul(V(\Lzero),k)$.

Combining the observations of the previous paragraph, we conclude that the subspace $\gr_1 V^* \cong L^*$ of $C^1(V,k) \cong V^*$ consists of cocycles whose images in $\Hbul(V,k)$ generate the subalgebra $\Lambda_s(L^*) = \Lambda({\Lzero}^*) \otimes S({\Lone}^*)$ of $\Hbul(V,k)$, and that the subspace $\betabar((\gr_1 V^*)_{\ol{0}}) \subset \sum_{i=1}^{p-1} (\gr_i V^*) \otimes (\gr_{p-i} V^*) \subset \gr_p C^2(V,k)$ consists of cocycles whose images in $\Hbul(V,k)$ generate $S'({\Lzero}^*)$. Then it follows that $E_1^{1,0}(k) \cong \Lambda_s^1(L^*)$ and $E_1^{p,2-p}(k) \cong {\Lzero}^* \subset S'({\Lzero}^*)$, and hence that $E_1^{i,j}(k)$ is as given in \eqref{eq:E1ij}.
\end{proof}

\begin{remark} \label{remark:barcomplexbasis}
As in Section \ref{subsection:Koszulresolutionrls}, let $\mathcal{P} \subset V$ be the set of PBW monomial basis vectors in $V$ arising from the fixed ordered basis $\mathcal{B} = \set{x_1,\ldots,x_s,y_1,\ldots,y_t}$ of $L$. Now $B_n(V)$ admits a basis consisting of all $z_0[z_1|\cdots|z_n]$ with $z_j \in \mathcal{P}$. We call this set the standard monomial basis for $B_n(V)$ with respect to the given ordered basis $\mathcal{B}$ for $L$, and we call the subset of vectors with $z_0 = 1$ the standard monomial basis for $\Bbar_n(V)$. Given $z \in B_n(V)$, we say that $z_0[z_1|\cdots|z_n]$ appears in $z$ if it appears with a nonzero coefficient when $z$ is expressed as a linear combination of standard monomial basis vectors. Given $1 \leq j \leq t$ and $n \in \N$, set $[y_j]^n = [y_j|\cdots|y_j] \in \Bbar_n(V)$.

The proof of Proposition \ref{proposition:E1page} shows that if $f \in (\gr_1 V^*)_{\ol{1}} \cong {\Lone}^*$, then the cochain $f \otimes \cdots \otimes f \in C^p(V,k)$ descends in the $E_1$-page of the spectral sequence $E_r(k) \Rightarrow \Hbul(V,k)$ to a cochain representative for an element of the subalgebra $S({\Lone}^*)^p$ of $\Hbul(\vlab,k)$. In particular, given $1 \leq i \leq t$, define $f_i : B_p(V) \rightarrow k$ so that $f_i([y_i|\cdots|y_i]) = 1$, and so that $f_i$ vanishes on all other standard basis monomials for $\Bbar_p(V)$. Then the $f_i$ for $1 \leq i \leq t$ descend in $E_1(k)$ to representatives for a set of generators of $S({\Lone}^*)^p$.
\end{remark}

\begin{remark}
Replacing $V(\Lzero)$ by an arbitrary finite-dimensional cocommutative superbialgebra $A$ over $k$, the function $\beta$ defined in \eqref{eq:beta} can be generalized to a function $\beta: C^1(A,k) \rightarrow C^2(A,k)$. One can check that $\beta$ maps cocycles in $\ker(\partial^1)_{\ol{0}} = \opH^1(A,k)_{\ol{0}}$ to cocycles in $C^2(A,k)_{\ol{0}}$, and hence that $\beta$ induces a semilinear function $\betabar: \opH^1(A,k)_{\ol{0}} \rightarrow \opH^2(A,k)_{\ol{0}}$. Given $f_1,f_2 \in \ker(\partial^1)_{\ol{0}}$, one has
\begin{equation} \label{eq:betaadditive} \textstyle
\beta(f_1+f_2) = \beta(f_1)+\beta(f_2) - \partial^1(\sum_{i=1}^{p-1} \cbinom{p}{i} f_1^i f_2^{p-i}),
\end{equation}
so that $\betabar$ is additive, though in general this formula fails if $\ol{f_1} \neq \ol{f_2}$.
\end{remark}

\subsection{Comparison of spectral sequences} \label{subsection:comparison}

In this section, assume that $k$ is perfect. Let $S({\Lone}^*)^p$ be the subalgebra of $S({\Lone}^*)$ generated by all $p$-th powers in $S({\Lone}^*)$. Then $E_1(k) \cong \Hbul(V(L_{ab}),k)$ and $E_1(M) \cong M \otimes \Hbul(V(L_{ab}),k)$ are each finitely-generated graded right super\-modules over the commutative subalgebra $S({\Lone}^*)^p \otimes S'({\Lzero}^*)$ of $E_1(k)$. Our goal is to show that $S({\Lone}^*)^p \otimes S'({\Lzero}^*)$ is generated by permanent cycles in $E_1(k)$. This will imply by \cite[Lemma 2.5]{Mastnak:2010} and \cite[III.2.9, Corollary 1]{Bourbaki:1998} that $\Hbul(V(L),k)$ is a finitely-generated graded super\-algebra, and that $\Hbul(V(L),M)$ is a finitely-generated graded super\-module over $\Hbul(V(L),k)$. To accomplish this goal, we compare the spectral sequence $E_r(k) \Rightarrow \Hbul(V(L),k)$ to a new spectral sequence $D_r(k) \Rightarrow \Hbul(V(L),k)$ constructed from a filtration on the complex $X(L)$. Specifically, we exhibit a morphism of spectral sequences $E_r(k) \rightarrow D_r(k)$ that induces isomorphisms $E_r^{i,j}(k) \cong D_r^{i,j}(k)$ for all $r \geq 1$. We then show that the generators for $S({\Lone}^*)^p \otimes S'({\Lzero}^*)$ in $E_1(k)$ map onto permanent cycles in $D_1(k)$.

We continue to write $V = V(L)$. Let $N(V)$ be the normalized left bar complex for $V$. Then $N_n(V) = V \otimes I(V)^{\otimes n}$, where $I(V)$ is the cokernel of the unit map $k \rightarrow V$. Set $\ol{N}_n(V) = 1 \otimes I(V)^{\otimes n}$. The differential $d$ and contracting homotopy $s$ on $N(V)$ are defined by the same formulas as on $B(V)$. In particular, the projection map $B(V) \rightarrow N(V)$ is a chain equivalence, with quasi-inverse $\psi: N(V) \rightarrow B(V)$ induced by the vector space map $I(V) \rightarrow V : v + k \mapsto v - \ve(v)$. The filtration on $B(V)$ induces a corresponding filtration on $N(V)$ that is compatible with the differential and contracting homotopy on $N(V)$ and with the chain homomorphism $\psi: N(V) \rightarrow B(V)$.

Let $F_0 X(L) \subset F_1 X(L) \subset \cdots$ be the filtration on $X(L)$ defined in the proof of Theorem \ref{theorem:X(L)resolution}. By \cite[Proposition 13]{May:1966}, the identity map $k \rightarrow k$ lifts to a chain homomorphism $\mu': X(L) \rightarrow N(V)$ such that $\mu'(\ol{X}_n(L)) \subset \Nbar_n(V)$, and such that the restriction of $\mu'$ to $\Xbar_n(L)$ is determined recursively through the formula $\mu_n' = s \circ \mu_{n-1}' \circ d_t$. Set $\mu = \psi \circ \mu'$. Since $\mu_0$ is the identity map $V \rightarrow V$, it follows by induction on $n$ that $\mu$ is compatible with the filtrations on the two complexes, and hence that $\mu$ induces a chain map $\gr(\mu) : \gr X(L) \rightarrow \gr B(V)$. Identifying $\gr X(L) = X(\lab)$ and $\gr B(V) = B(\vlab)$, $\gr(\mu)$ is a homomorphism of $\vlab$-complexes lifting the identity $k \rightarrow k$.

Applying \cite[Proposition 13]{May:1966} to the Koszul complex $Y(L)$, one obtains a chain homomorphism $\nu': Y(L) \rightarrow N(U(L))$ satisfying $\nu'(\Ybar(L)) \subset \Nbar(U(L))$, and such that the restriction of $\nu'$ to $\Ybar_n(L)$ is determined recursively through the formula $\nu_n' = s \circ \nu_{n-1}' \circ d$. The projection homomorphism $U(L) \rightarrow V(L)$ induces a chain homomorphism $N(U(L)) \rightarrow N(V)$, and then the composite map $Y(L) \rightarrow N(U(L)) \rightarrow N(V)$ factors through a chain homomorphism $W(L) \rightarrow N(V)$, which by abuse of notation we also denote by $\nu'$. Then $\nu': W(L) \rightarrow N(V)$ also satisfies the recursive formula $\nu_n' = s \circ \nu_{n-1}' \circ d$. Set $\nu = \psi \circ \nu': W(L) \rightarrow B(V)$.

\begin{lemma} \label{lemma:mu}
For each $1 \leq i \leq s$, $1 \leq j \leq t$, and $n \in \N$, one has
\[
\mu(\subgrp{x_i}) = [x_i] \quad \text{and} \quad \mu(\gamma_n(y_j)) = [y_j]^n.
\]
Given a standard monomial basis vector $w \in \Xbar_n(L)$, $[y_j]^n$ appears in $\mu(w)$ only if $w = \gamma_n(y_j)$.
\end{lemma}

\begin{proof}
The formulas $\mu(\subgrp{x_i}) = [x_i]$ and $\mu(\gamma_n(y_j)) = [y_j]^n$ can be checked via the recursion $\mu_n' = s \circ \mu_{n-1}' \circ d_t$, the latter via induction on $n$. In particular, these two formulas establish the $n = 1$ case of the second statement of the lemma. For the general case we argue by induction on $n$.

Suppose $\deg(w) = n > 1$ and that $[y_j]^n$ appears in $\mu(w)$. Write $w = u z$ with
\begin{equation} \label{eq:uzfactorization}
u = \subgrp{x_{i_1} \cdots x_{i_b}} \gamma_{a_1}(y_1) \cdots \gamma_{a_t}(y_t) \quad \text{and} \quad z = \gamma_{c_1}'(x_1) \cdots \gamma_{c_s}'(x_s).
\end{equation}
Set $c = \sum_{i = 1}^s c_i$. First suppose that $c \geq 1$, so that $\deg(z) = 2c \geq 2$. Then $\deg(u) = n - 2c \leq n-2$, so in particular $\sum_{i=1}^t a_i \leq n-2$. Since $[y_j]^n$ appears in $\mu(w)$, it follows via the recursion formula that $y_j[y_j]^{n-1}$ must appear in $\mu \circ d_t(w)$, and hence by the induction hypothesis that the standard basis monomial $y_j \gamma_{n-1}(y_j) \in X_{n-1}(L)$ must appear in $d_t(w)$. From the definition of $d_t$, it follows that $y_j\gamma_{n-1}(y_j)$ can appear in $d_t(w)$ only if it appears in $u.t_{2c}(z)$. Here $t_{2c}(z) \in Y_{2c-1}(\Lzero)$. Considering $u$ as an element of the algebra $W(L)$, one can check that when $u$ is multiplied on the right by an element of the subalgebra $W(\Lzero)$, and when the product is written as a $V$-linear combination of the standard basis monomials $\subgrp{x_{\ell_1} \cdots x_{\ell_e}} \gamma_{q_1}(y_1) \cdots \gamma_{q_t}(y_t)$, the resulting monomials must also satisfy the condition $\sum_{i=1}^t q_i \leq n-2$. Then it follows that $[y_j]^n$ cannot appear in $\mu(w)$ if $c \geq 1$.

Now suppose that $c = 0$, so that $w = u$. In this case we do not require the induction hypothesis. Since $d_t$ restricts on the subspace $W(L) \subset X(L)$ to the Koszul differential $d: W(L) \rightarrow W(L)$, it follows via the recursion formula that $\mu$ restricts on $W(L) \subset X(L)$ to the chain homomorphism $\nu$ discussed prior to the statement of the lemma. Now arguing as in \cite[\S7]{May:1966}, one can show that $\nu$ is given in terms of the shuffle product on $\Bbar(V)$ via a formula similar to \cite[(7.2)]{May:1966}. In particular, $\nu(u)$ can be written as a linear combination of standard monomials $[z_1|\ldots|z_n]$ such that in the ordered list $z_1,\ldots,z_n$, the elements $x_{i_1},\ldots,x_{i_b}$ each appear once, and $y_i$ appears $a_i$ times. Then $[y_j]^n$ can appear in $\nu(u)$ only if $u = \gamma_n(y_j)$. 
\end{proof}

As in Section \ref{subsection:Mayspecseq}, define a decreasing filtration on the complex $\Hom_V(X_\bullet(L),k)$ by
\begin{equation} \label{eq:X(L)cochainfiltration}
F^i \Hom_{V}(X_n(L),k) = \Hom_{V}(X_n(L)/(V.F_{i-1}\Xbar_n(L)),k).
\end{equation}
Then there exists a spectral sequence $D_r(k) \Rightarrow \Hbul(V,k)$ with
\begin{align*}
D_0^{i,j}(k) &= F^i \Hom_{V}(X_{i+j}(L),k) / F^{i+1} \Hom_{V}(X_{i+j}(L),k) \\
&\cong \Hom_k(F_i \Xbar_{i+j}(L)/F_{i-1} \Xbar_{i+j}(L),k),
\end{align*}
One can check that $D_0(k)$ is isomorphic as a complex to $\Hom_{\vlab}(X(\lab),k)$. Since the differential on $X(\lab)$ induces the trivial differential on $\Hom_{\vlab}(X(\lab),k)$ (cf.\ Remarks \ref{remark:trivdif} and \ref{remark:tnzero}, and the explicit formula for $t_2$), it follows that $D_0(k) \cong D_1(k) \cong \Hbul(\vlab,k)$. Now $\mu: X(L) \rightarrow B(V)$ induces a morphism of spectral sequences $E_r(k) \rightarrow D_r(k)$. Since $\gr(\mu): X(\lab) \rightarrow B(\vlab)$ is a chain homomorphism lifting the identity $k \rightarrow k$, it follows that $\mu$ induces an isomorphism $E_1(k) \cong D_1(k)$, and hence by \cite[Theorem 3.4]{McCleary:2001} that $\mu$ induces isomorphisms $E_r^{i,j}(k) \cong D_r^{i,j}(k)$ for all $1 \leq r \leq \infty$.

\begin{lemma} \label{lemma:cocycles}
Given $1 \leq i \leq t$, let $f_i \in \Hom_V(X_p(L),k)$ be the homomorphism that evaluates to $1$ on $\gamma_p(y_i)$, and that evaluates to $0$ on all other standard basis monomials in $X_p(L)$. Then $f_i$ is a cocycle in $\Hom_V(X_\bullet(L),k)$. Similarly, given $1 \leq j \leq s$, let $g_j \in \Hom_V(X_2(L),k)$ be the homomorphism that evaluates to $1$ on $\gamma_1'(x_j)$, and that evaluates to $0$ on all other standard basis monomials in $X_2(L)$. Then $g_j$ is a cocycle in $\Hom_V(X_2(L),k)$.
\end{lemma}

\begin{proof}
Suppose that $\partial(f_i) = (-1)^{\deg(f)} f_i \circ d_t \neq 0$. Since $f_i$ and $d_t$ are $V$-super\-module homo\-morphisms, and since $V$ acts trivially on $k$, this implies that there exists a standard basis monomial
\[
w = \subgrp{x_{i_1} \cdots x_{i_b}} \gamma_{a_1}(y_1) \cdots \gamma_{a_t}(y_t) \gamma_{c_1}'(x_1) \cdots \gamma_{c_s}'(x_s) \in \Xbar_{p+1}(L)
\]
such that $f_i(d_t(w)) \neq 0$. First suppose that $c_j = 0$ for each $1 \leq j \leq s$, so that we may consider $w$ as a monomial in $Y(L)$. Then in the notation from the proof of Lemma \ref{lemma:cochainring}, $({y_i^*}^{\odot p})(d(w)) = f_i(d_t(w)) \neq 0$, a contradiction because ${y_i^*}^{\odot p}$ is a cocycle in $C^p(L,k)$. Now suppose $c := \sum_{j=1}^s c_j > 0$. Write $w = uz$ as in \eqref{eq:uzfactorization}. Since $f_i$ vanishes on $\bigoplus_{n > 0} W(L) \otimes \Gamma_{2n}'(\Lzero)$, it follows that $f_i(d_t(w)) = (-1)^{\deg(u)} f_i(u.t_{2c}(z))$. Since $\deg(w) = p+1$ and $\deg(z) = 2c \geq 2$, it follows that $\sum_{i=1}^t a_j \leq p-1$. Now arguing as in the second paragraph of the proof of Lemma \ref{lemma:mu}, it follows that $\gamma_p(y_j)$ cannot occur as a summand when $u.t_{2c}(z)$ is written as a linear combination of standard basis monomials in $X(L)$, and hence that $f_i(d_t(w)) = 0$, a contradiction. Next suppose $2c = p+1$. Then $w = z$, and $f_i(d_t(w)) = f_i(\pi \circ t_{2c}(z))$, where $\pi: Y(\Lzero) \rightarrow W(L) \subset X(L)$ is the natural map. But $t_{2c}(z) \in Y(\Lzero)$ has $\Z_2$-degree $\ol{0}$, whereas $f_i$ is nonzero only on elements of $\Z_2$-degree $\ol{1}$. Then $f_i(d_t(w)) = 0$. Thus, in all cases we reach a contradiction, so we conclude that in fact $\partial(f_i) = 0$.

The proof that $g_j$ is a cocycle is similar to that for $f_i$. Suppose $w \in \Xbar_3(L)$ is a standard basis monomial such that $g_j(d_t(w)) \neq 0$. Then either $w$ is an element of the subspace $W(L) \subset X(L)$, or $w = u \gamma_1'(x_\ell)$ for some $1 \leq \ell \leq s$ and some degree-$1$ basis monomial $u \in \Ybar(L)$ as in \eqref{eq:uzfactorization}. In the first case, $g_j(d_t(w)) = g_j(d(w)) = 0$ because $g_j$ vanishes on the subspace $W(L) \subset X(L)$. In the second case, $g_j(d_t(w)) = d(u).g_j(\gamma_1'(x_\ell)) - g_j(u.t_2(\gamma_1'(x_\ell)))$. The first term in this expression vanishes because $d(u) \in V(L)$ acts trivially on $k$, and the second term vanishes because $u.t_2(\gamma_1'(x_\ell)) \in W(L) \subset X(L)$. Then we conclude that $g_j$ is also a cocycle in $\Hom_V(X_\bullet(L),k)$.
\end{proof}

\begin{proposition} \label{proposition:permanentcycles}
Suppose $k$ is perfect. In the spectral sequence $E_r(k) \Rightarrow \Hbul(V,k)$, the subalgebra $S({\Lone}^*)^p \otimes S'({\Lzero}^*) \subset E_1(k)$ is generated by permanent cycles.
\end{proposition}

\begin{proof}
Since $E_r^{i,j}(k) \cong D_r^{i,j}(k)$ for all $r \geq 1$, it suffices to show that a set of generators for $S({\Lone}^*)^p \otimes S'({\Lzero}^*)$ in $E_1(k)$ maps onto a set of permanent cycles in $D_1(k)$. The subalgebra $S'({\Lzero}^*) \subset E_1(k)$ is generated by ${\Lzero}^* \cong E_1^{p,2-p}(k)$. Since $D_1^{p,2-p}(k) \cong D_0^{p,2-p}(k) \cong \Hom_k(\Gamma_2'(\Lzero),k)$, it follows from Lemma \ref{lemma:cocycles} that each element of $D_0^{p,2-p}(k)$ lifts to a cocycle in $\Hom_V(X(L),k)$, and hence that $D_0^{p,2-p}(k) \cong D_1^{p,2-p}(k)$ consists of permanent cycles in $D_1(k)$. Next, let $f_i \in C^p(V,k)$ be as defined in Remark \ref{remark:barcomplexbasis}. Then the $f_i$ for $1 \leq i \leq t$ descend in $E_1(k)$ to a set of generators for $S({\Lone}^*)^p$. Moreover, it follows from Lemma \ref{lemma:mu} that $f_i$ maps under the morphism $\mu^*: \Hom_V(B_p(V),k) \rightarrow \Hom_V(X_p(L),k)$ to the cocycle named $f_i$ in Lemma \ref{lemma:cocycles}. Then it follows from Lemma \ref{lemma:cocycles} that $f_i$ maps to a permanent cycle in $D_1(k)$.
\end{proof}

\begin{theorem} \label{theorem:V(L)cohomfingen}
Let $k$ be a perfect field of characteristic $p > 2$, and let $L$ be a finite-dimensional restricted Lie super\-algebra over $k$. Let $M$ be a finite-dimensional left $V(L)$-super\-module. Then the cohomology ring $\Hbul(V(L),k)$ is a finitely-generated graded super\-algebra, and $\Hbul(V(L),M)$ is a finitely-generated graded right super\-module over $\Hbul(V(L),k)$.
\end{theorem}

\begin{proof}
While the statement of \cite[Lemma 2.5]{Mastnak:2010} assumes that the bigraded commutative algebra $A^{*,*}$ is concentrated in even total degrees, the reader can inspect the proof of the lemma to check that this assumption is superfluous when working with right actions instead of with left actions, and hence that the lemma can be applied to the bigraded commutative algebra $S({\Lone}^*)^p \otimes S'({\Lzero}^*) \subset E_1(k)$, even though it is not concentrated in even total degrees. Then by the discussion at the start of Section \ref{subsection:comparison}, the theorem follows from Proposition \ref{proposition:permanentcycles}.
\end{proof}

\begin{remark} \label{remark:perfectsuperfluous}
Let $k$ be an arbitrary field of characteristic $p > 2$, let $L$ be a finite-dimensional restricted Lie superalgebra over $k$, and let $k'$ be a field extension of $k$. Then $L \otimes k'$ is a restricted Lie superalgebra over $k'$, and one can check that $\Hbul(V(L),k) \otimes k' \cong \Hbul(V(L \otimes k'),k')$ as $k'$-algebras. Moreover, $\Hbul(V(L),k) \otimes k'$ is finitely-generated as a $k'$-algebra if and only if $\Hbul(V(L),k)$ is finitely-generated as a $k$-algebra. This shows that the perfectness assumption in Theorem \ref{theorem:V(L)cohomfingen} is superfluous, since one can always first extend scalars to the perfect field $k' = \ol{k}$ if necessary.
\end{remark}

\begin{remark} \label{remark:Liu}
The functions $f_i \in C^p(V,k)$ defined in Remark \ref{remark:barcomplexbasis} appear to be closely related to the functions $f_{\alpha_i}$ defined by Liu in the case that $L = \g$ is a basic classical restricted Lie superalgebra \cite[p.\ 172]{Liu:2012}. More accurately, $f_{\alpha_i}$ appears related to the cup product $f_i \odot f_i \in C^{2p}(V,k)$. In Liu's notation, the $f_{\alpha_i}$ are defined in terms of coefficients of a monomial $x_{\alpha_i}^{N_i}$ in the universal enveloping algebra $\Ug$. Liu asserts that $\Hbul(\Vg,k)$ is finitely generated over the subalgebra generated by the classes of the $f_{\alpha_i}$. The proof of this claim involves showing that the $f_{\alpha_i}$ descend in a spectral sequence to representatives for certain nonzero cohomology classes over which the $E_1$-page of the spectral sequence is finitely generated; see \cite[p.\ 173]{Liu:2012}. However, if in Liu's notation one has $\alpha_i \in \Phi_{11}$, then $x_{\alpha_i}^{N_i} = x_{\alpha_i}^2 = 0$ in $\Ug$, and hence $f_{\alpha_i} = 0$. Thus in this case, $f_{\alpha_i}$ cannot descend to the desired nonzero cohomology class. Even if one redefines Liu's function $f_{\alpha_i}$ to conform with the cup product $f_i \odot f_i$ as defined here, arguments beyond those presented in Liu's paper seem to be required to show that the $f_{\alpha_i}$ are cocycles for $\Vg$. (Additional corrections to Liu's arguments are also required starting on page 167 of his paper, since he applies the same incorrect interpretation of \cite[Theorem 4.1]{Mastnak:2010} that we comment upon in Remark \ref{remark:MPSW}.)
\end{remark}

\subsection{Reindexing the May spectral sequence}

As in \cite[I.9.16]{Jantzen:2003}, we can reindex \eqref{eq:Mayspecseq} into a first quadrant spectral sequence so that the new $E_r^{i,j}$-term is equal to the old $E_{(p-2)r+1}^{(p-1)i+j,-(p-2)i}$.\footnote{This expression is incorrectly typeset as $E_{(p-2)r+1}^{i+j,-(p-2)i}$ in \cite[I.9.16]{Jantzen:2003}.} Then the spectral sequence can be rewritten as
\begin{equation} \label{eq:E0V(L)specseq}
E_0^{i,j}(M) = M \otimes \Lambda_s^{j-i}(L^*) \otimes S^i({\Lzero}^*) \Rightarrow \opH^{i+j}(V(L),M).
\end{equation}
Similarly, $D_r(k) \Rightarrow \Hbul(V,k)$ can be reindexed, and then $D_r^{i,j}(k) \cong E_r^{i,j}(k)$ for all $r \geq 0$. For the remainder of this section we utilize this new indexing of $E_r(M)$ and $D_r(M)$.

\begin{remark} \label{remark:Bagci}
The indexing \eqref{eq:E0V(L)specseq} appears to be the spectral sequence envisioned by Bagci \cite[Proposition 3.4]{Bagci:2012a}, though we were unable to verify the details of his con\-struction, nor does he appear to obtain as explicit a description of the $E_0$-page as we do. Using this spectral sequence, Bagci claims to show that the cohomology ring of a finite-dimensional restricted Lie super\-algebra is a finitely-generated algebra \cite[Theorem 3.5]{Bagci:2012a}. Bagci's argument is based on the fact that $E_0(k)$ is finitely-generated over the subalgebra $S({\Lone}^*) \otimes S({\Lzero}^*)$, though his argument is incomplete since he fails to show any part of this subalgebra to consist of permanent cycles.
\end{remark}

\begin{theorem} \label{theorem:V(L)specseqE1page}
The spectral sequence \eqref{eq:E0V(L)specseq} has $E_1$-page
\[
E_1^{i,j}(M) = \opH^{j-i}(L,M) \otimes S^i({\Lzero}^*).
\]
\end{theorem}

\begin{proof}
Our argument mimics that in \cite[I.9.18, I.9.20]{Jantzen:2003}. First, it follows from Proposition \ref{proposition:permanentcycles} that the subalgebra $S({\Lzero}^*)  = \bigoplus_{i = 0}^\infty E_0^{i,i}(k)$ of $E_0(k)$ consists of permanent cycles. In particular, this subalgebra lies in the kernel of the differential $d_0$ on $E_0(k)$. Then by the derivation property of the differentials on $E_0(M)$ and $E_0(k)$, and by the discussion preceding Theorem \ref{theorem:fgforU(L)}, it suffices to show that the maps $d_0^{0,0}(M): E_0^{0,0}(M) \rightarrow E_0^{0,1}(M)$ and $d_0^{0,1}(k): E_0^{0,1}(k) \rightarrow E_0^{0,2}(k)$ identify with the maps $M \rightarrow M \otimes \Lambda_s^1(L^*)$ and $\Lambda_s^1(L^*) \rightarrow \Lambda_s^2(L^*)$ arising in the computation of the cohomology groups $\Hbul(L,M)$ and $\Hbul(L,k)$. As in \cite[I.9.18]{Jantzen:2003}, one can check that $d_0^{0,0}(M): E_0^{0,0}(M) \rightarrow E_0^{0,1}(M)$ identifies with the map
\begin{multline*}
M \cong C^0(V,M) \rightarrow F^1 C^1(V,M) / F^2 C^1(V,M) \\
\cong \Hom_k(F_1 \Bbar_1(V)/F_0 \Bbar_1(V),M) \cong \Hom_k(L,M) \cong M \otimes L^*
\end{multline*}
induced by $\partial: C^0(V,M) \rightarrow C^1(V,M)$. Given $z \in L$, one has $\partial(\wt{m})([z]) = \wt{m}(z) = (-1)^{\ol{z} \cdot \ol{m}}z.m$. Thus, $d_0^{0,0}(M)$ identifies with the map appearing in the computation of $\Hbul(L,M)$. Next, one can check as in \cite[I.9.18]{Jantzen:2003} that $\partial: C^1(V,k) \rightarrow C^2(V,k)$ maps $F^1 C^1(V,k)$ into $F^2 C^2(V,k)$, and that $d_0^{0,1}(k): E_0^{0,1}(k) \rightarrow E_0^{0,2}(k)$ identifies with the map
\begin{multline} \label{eq:d0(0,1)}
E_0^{0,1}(k) \cong F^1 C^1(V,k) / F^2 C^1(V,k) \rightarrow \\ E_0^{0,2}(k) \subset F^2 C^2(V,k) / \left( F^3 C^2(V,k) + \partial(F^2C^1(V,k)) \right)
\end{multline}
induced by $\partial: C^1(V,k) \rightarrow C^2(V,k)$. Since $E_0^{i,j}(k) \cong D_0^{i,j}(k)$, it suffices to show that the differential $D_0^{0,1}(k) \rightarrow D_0^{0,2}(k)$ admits the desired description. Replacing $F^i C^n(V,k)$ by $F^i \Hom_V(X_n(L),k)$ in \eqref{eq:d0(0,1)}, one has $D_0^{0,1}(k) \cong \Hom_k(\Ybar_1(L),k) \cong L^*$ and $D_0^{0,2}(k) \cong \Hom_k(\Ybar_2(L),k) \cong \Lambda_s^2(L)^* \cong \Lambda_s^2(L^*)$; cf.\ Remark \ref{remark:exteriorduality}. Now inspecting the differential on $X(L)$, it follows that the differential $D_0^{0,1}(k) \rightarrow D_0^{0,2}(k)$ identifies with the transpose of the Lie bracket, as desired.
\end{proof}

\begin{remark} \label{remark:FPreindexing}
As in \cite[Remark I.9.20]{Jantzen:2003}, an alternative reindexing of \eqref{eq:E0V(L)specseq} is possible in which
\begin{align*}
E_0^{2i,j}(M) &= M \otimes \Lambda_s^j(L^*) \otimes S^i({\Lzero}^*), \\
E_2^{2i,j}(M) &= \opH^j(L,M) \otimes S^i({\Lzero}^*),
\end{align*}
and $E_0^{i,j}(M) = E_2^{i,j}(M) = 0$ if $i$ is odd.
\end{remark}

\section{Supergroup schemes} \label{section:supergroupschemes}

The theory of affine supergroup schemes parallels the theory of ordinary affine group schemes; see \cite{Jantzen:2003}. In this section we recall certain aspects of the theory in preparation for establishing basic results on the cohomology of finite supergroup schemes.

\subsection{Affine supergroup schemes} \label{subsection:affinesupergroups}

Recall that a \emph{$k$-superfunctor} is a functor from $\mathfrak{csalg}_k$ to the category of sets. Given $k$-superfunctors $X$ and $X'$, write $\Mor(X,X')$ for the set of morphisms, i.e., natural transformations, from $X$ to $X'$. Given $R \in \csalg_k$, the \emph{superspectrum} of $R$ is the $k$-super\-functor $\Sspk R$ satisfying $(\Sspk R)(A) = \Hom_{\salg}(R,A)$ for each $A \in \csalg_k$. A $k$-superfunctor $X$ is an \emph{affine $k$-superscheme} if $X \cong \Sspk R$ for some $R \in \csalg_k$. Given an affine $k$-superscheme $X = \Sspk R$, we set $k[X] = R$, and call $k[X]$ the coordinate super\-algebra of $X$. The coordinate super\-algebra is well-defined up to isomorphism, since by Yoneda's Lemma, $\Mor(\Sspk R,\Sspk S) \cong \Hom_{\salg}(S,R)$, so $\Sspk R \cong \Sspk S$ if and only if $R \cong S$.

A \emph{$k$-supergroup functor} is a functor from $\csalg_k$ to the category of groups. A $k$-supergroup functor $G$ is an \emph{affine $k$-supergroup scheme} if it is both a $k$-supergroup functor and an affine $k$-superscheme, i.e., if there exists $k[G] \in \csalg_k$ such that $G(A) = \Hom_{\salg}(k[G],A)$ for each $A \in \csalg_k$. As for ordinary group schemes, the group morphisms on $G$ correspond uniquely to comorphisms on $k[G]$, and endow $k[G]$ with the structure of Hopf super\-algebra. Then the category of affine $k$-supergroup schemes is anti-equivalent to the category of commutative Hopf super\-algebras over $k$. We denote the product, coproduct, antipode, and augmentation maps on $k[G]$ by $m_{k[G]}$, $\Delta_G$, $\sigma_G$, and $\ve_G$, respectively, and write $I_\ve = \ker(\ve_G)$ for the augmentation ideal of $k[G]$.

Let $A \in \csalg_k$. Replacing $k$ by $A$, it makes sense to consider $A$-super\-functors, $A$-superschemes, $A$-supergroup functors, etc. Given a $k$-supergroup scheme $G$, the tensor product $A[G_A]:= k[G] \otimes A$ is naturally a commutative Hopf super\-algebra over $A$, and hence is the coordinate super\-algebra of an affine $A$-super\-group scheme, denoted $G_A$. If $A' \in \csalg_A$, then also $A' \in \csalg_k$, and the inclusion $k[G] \hookrightarrow A[G_A]$ induces an identification $G_A(A') = G(A')$. We call $G_A$ the supergroup scheme obtained via base change to $A$ from $G$.

Let $G$ be an affine $k$-supergroup scheme. The odd subspace $k[G]_{\ol{1}}$ generates a Hopf superideal $k[G] k[G]_{\ol{1}}$ in $k[G]$, hence defines a closed sub\-super\-group scheme $\Gev$ of $G$, called the \emph{underlying purely even subgroup of $G$}, which has coordinate super\-algebra $k[\Gev] = k[G]/(k[G]k[G]_{\ol{1}})$. Since $k[\Gev]$ is concentrated in degree $\ol{0}$, we can consider $\Gev$ as an ordinary affine group scheme. Given $A \in \csalg_k$, one has identifications $\Gev(A_{\ol{0}}) = \Gev(A)$ and $\Gev(A) = G(A_{\ol{0}})$ that are induced by the inclusion $A_{\ol{0}} \hookrightarrow A$ and the projection $k[G] \twoheadrightarrow k[\Gev]$, respectively. Conversely, if $G$ is an ordinary affine group scheme, then we can consider $G$ as a supergroup scheme with $G = \Gev$.

An affine $k$-supergroup scheme $G$ is \emph{algebraic} if $k[G]$ is finitely-generated as a superalgebra. Equivalently, $G$ is algebraic if $k[G]$ is a Hopf superalgebra quotient of $S_s(V)$ for some finite-dimen\-sional superspace $V$. An affine supergroup scheme $G$ is \emph{reduced} if $\Gev$ is reduced as an ordinary group scheme, that is, if $k[\Gev]$ has no nonzero nilpotent elements. An affine $k$-supergroup scheme $G$ is \emph{finite} if $k[G]$ is finite-dimensional over $k$, and is \emph{infinitesimal} if it is finite and if the augmentation ideal $I_\ve$ of $k[G]$ is nilpotent. If $G$ is infinitesimal, then the minimal $r \in \N$ such that $x^{p^r} = 0$ for all $x \in I_\ve$ is called the \emph{height} of $G$.

\subsection{Distributions and the Lie super\-algebra of a supergroup scheme} \label{subsection:distributions}

Let $G$ be an affine $k$-supergroup scheme. Given $n \in \N$, the set $\Dist_n(G)$ of \emph{distributions of order $\leq n$ on $G$} is defined by $\Dist_n(G) = \{ \mu \in k[G]^* : \mu(I_\ve^{n+1}) = 0 \}$. Set $\Dist(G) = \bigcup_{n \geq 0} \Dist_n(G)$. As for ordinary affine group schemes, it follows that the Hopf super\-algebra structure maps on $k[G]$ induce on $\Dist(G)$ the structure of a cocommutative Hopf super\-algebra over $k$. The product of $\mu,\nu \in \Dist(G)$ is defined by $\mu \nu := (\mu \otimes \nu) \circ \Delta_G : k[G] \rightarrow k$. With this super\-algebra structure, $\Dist(G)$ is called the \emph{super\-algebra of distributions on $G$}.

Given $\mu \in \Dist_m(G)$ and $\nu \in \Dist_n(G)$, one has $[\mu,\nu] := \mu \nu - (-1)^{\ol{\mu} \cdot \ol{\nu}} \nu \mu \in \Dist_{m+n-1}(G)$. In particular, the subspace $\Dist_1^+(G) := \{ \mu \in \Dist_1(G) : \mu(1) = 0 \} \cong (I_\ve/I_\ve^2)^*$ of $\Dist_1(G)$ is closed under the super\-commutator on $\Dist(G)$. Then $\Dist_1^+(G)$ admits the structure of a Lie super\-algebra, called \emph{the Lie super\-algebra of $G$}, which we denote by $\Lie(G)$. One can check that the coproduct on $\Dist(G)$ restricts to the map $\mu \mapsto \mu \otimes 1 + 1 \otimes \mu$ on $\Dist_1^+(G) = \Lie(G)$. As described in \cite[\S2.3]{Shu:2008}, the inclusion $\Gev \subset G$ induces a Lie algebra isomorphism $\Lie(\Gev) \cong \Lie(G)_{\ol{0}}$, and an injective super\-algebra homomorphism $\Dist(\Gev) \hookrightarrow \Dist(G)$. If $\chr(k) = p > 0$, it follows that $\Lie(G)$ is a restricted Lie super\-algebra, with the restriction map on $\Lie(G)_{\ol{0}}$ arising from the $p$-power map on $\Dist(G)$, and that $\Lie(\Gev) \cong \Lie(G)_{\ol{0}}$ is an isomorphism of restricted Lie algebras.

Observe that ${\Lie(G)_{\ol{1}}}^* \cong (I_\ve/I_\ve^2)_{\ol{1}} \cong k[G]_{\ol{1}}/((I_\ve)_{\ol{0}} \cdot k[G]_{\ol{1}})$. Then by \cite[Theorem 4.5]{Masuoka:2005}, there exists an isomorphism of augmented super\-algebras (though not necessarily of Hopf superalgebras)
\begin{equation} \label{eq:Masuoka}
k[G] \cong k[\Gev] \otimes \Lambda({\Lie(G)_{\ol{1}}}^*).
\end{equation}
This implies that $G$ is algebraic (resp.\ finite, infinitesimal) as an affine supergroup scheme if and only if $\Gev$ is algebraic (resp.\ finite, infinitesimal) as an ordinary affine group scheme; cf.\ \cite[Corollary 7.1]{Zubkov:2009}. In particular, if $G$ is infinitesimal, then the height of $G$ is equal to the height of $\Gev$.

\subsection{Representations} \label{subsection:representations}

Given $M \in \fsvec_k$, define $GL(M)$ to be the $k$-supergroup functor such that, for each $A \in \csalg_k$, $GL(M)(A)$ is the group of all even $A$-linear automorphisms of $M \otimes A$. Let $G$ be an affine $k$-supergroup scheme. We say that $M$ is a \emph{rational $G$-super\-module}, or that $M$ affords a \emph{rational representation of $G$}, if there exists a morphism of $k$-supergroup functors $\rho: G \rightarrow GL(M)$. Equivalently, $M$ is a rational $G$-super\-module if there exists an even linear map $\Delta_M: M \rightarrow M \otimes k[G]$ making $M$ into a right $k[G]$-supercomodule; cf.\ \cite[I.2.8]{Jantzen:2003}. Given such a map, and given $A \in \csalg_k$, an element $g \in G(A) = \Hom_{\salg}(k[G],A)$ acts on $m \otimes 1 \in M \otimes A$ via the map $(1 \otimes g) \circ \Delta_M : M \rightarrow M \otimes A$. In this way, the category $\smod_G$ of rational $G$-super\-modules is equivalent to the category $\scomod_{k[G]}$ of right $k[G]$-supercomodules \cite[Corollary 9.1.5]{Westra:2009}. All $G$-super\-mod\-ules considered in this paper will be rational supermodules. The underlying even sub\-category of $\smod_G$, which has as morphisms only the even $G$-super\-module homomorphisms, is denoted $\fsmod_G$. Then $\Hom_{\fsmod_G}(M,N) = \Hom_G(M,N)_{\ol{0}}$.

\begin{remark} \label{remark:evensubmodule}
If $G$ is an ordinary affine group scheme, considered also as a supergroup scheme with $G = \Gev$, and if $M$ is a rational $G$-supermodule, then $M_{\ol{0}}$ and $M_{\ol{1}}$ are $G$-submodules of $M$.
\end{remark}

Let $A \in \csalg_k$. By Yoneda's Lemma, there exists a bijection $\Mor(G_A,\A^{1|1}) \cong A[G_A]$. Here $\A^{1|1}$ is the $A$-superfunctor satisfying $\A^{1|1}(A') = (A')_{\ol{0}} \oplus (A')_{\ol{1}} = A'$ for each $A' \in \csalg_A$. Given $f \in k[G]$ and $a \in A$, the natural transformation $G_A \rightarrow \A^{1|1}$ corresponding to $f \otimes a \in k[G] \otimes A = A[G_A]$ maps $h \in G_A(A') = \Hom_{\salg_A}(A[G_A],A')$ to $h(f \otimes a) \in A'$. The \emph{right regular representation} $\rho_r$, the \emph{left regular representation} $\rho_l$, and the \emph{conjugation representation} $\rho_c$ of $G$ on $k[G]$ are now defined as follows. Let $A \in \csalg_k$, and let $g \in G(A)$. Given $A' \in \csalg_A$ with $A$-super\-algebra structure map $A \rightarrow A'$, write $g_{A'}$ for the image of $g$ in $G(A')$ under the homomorphism $G(A \rightarrow A')$. Then identifying $k[G] \otimes A$ with $\Mor(G_A,\A^{1|1})$ as above, the maps $\rho_r(g),\rho_l(g),\rho_c(g) \in \End_A(k[G] \otimes A)_{\ol{0}}$ are defined by requiring, for each $h \in G(A')$,
\begin{align}
\rho_r(g)(f \otimes a)(h) &= (f \otimes a)(h \cdot g_{A'}), \\
\rho_l(g)(f \otimes a)(h) &= (f \otimes a)(g_{A'}^{-1} \cdot h), \text{ and} \\
\rho_c(g)(f \otimes a)(h) &= (f \otimes a)(g_{A'}^{-1} \cdot h \cdot g_{A'}). \label{eq:rhoc}
\end{align}
The comodule structure maps for these three representations are defined by $\Delta_{\rho_r} = \Delta_G$, $\Delta_{\rho_l} = T \circ (\sigma_G \otimes \id_{k[G]}) \circ \Delta_G$, and $\Delta_{\rho_c} = (\id_{k[G]} \otimes m_{k[G]}) \circ (T \otimes \id_{k[G]}) \circ (\sigma_G \otimes \Delta_G) \circ \Delta_G$.

Using \eqref{eq:rhoc}, one can check that $I_\ve$ is a $G$-subsupermodule for the conjugation action of $G$ on $k[G]$, and that the multiplication and comultiplication maps on $k[G]$ are homomorphisms for the conjugation action. Then it follows for all $n \in \N$ that $I_\ve^{n+1}$ is a $G$-subsupermodule of $k[G]$, and hence that $\Dist_n(G) = (k[G]/I_\ve^{n+1})^*$ is naturally a $G$-supermodule. Each quotient map $k[G]/I_\ve^{n+1} \rightarrow k[G]/I_\ve^{n+2}$ is a $G$-supermodule homomorphism, hence so is each inclusion $\Dist_n(G) \subset \Dist_{n+1}(G)$. Then we obtain a $G$-supermodule structure on $\Dist(G)$, called the \emph{adjoint representation of $G$ on $\Dist(G)$}. The restriction of the adjoint representation to $\Lie(G) = \Dist_1^+(G) = (I_\ve/I_\ve^2)^*$ is called the adjoint representation of $G$ on $\Lie(G)$. Since $\Delta_G$ is a homomorphism for the conjugation action, the product on $\Dist(G)$ is a homomorphism for the adjoint action. Then it follows for each $A \in \csalg_k$ that $G(A)$ acts on the Lie superalgebra $\Lie(G) \otimes A$ via Lie super\-algebra automorphisms.

Let $H$ be an affine $k$-supergroup scheme. The fixed point functor $(-)^H: \smod_H \rightarrow \svec_k$ is defined for $M \in \smod_H$ by $M^H = \set{m \in M: \Delta_M(m) = m \otimes 1}$. Then
\[
M^H = \{m \in M: g.(m \otimes 1) = m \otimes 1 \text{ for all $g \in G(A)$ and all $A \in \csalg_k$}\}.
\]
Let $G$ be an affine supergroup scheme, and suppose that $H$ is a sub\-super\-group scheme of $G$. Then $H$ acts diagonally on $M \otimes k[G]$ via the given action on $M$ and the left regular representation on $k[G]$. Now the induction functor $\ind_H^G: \smod_H \rightarrow \smod_G$ is defined by $\ind_H^G(M) = (M \otimes k[G])^H$. Here $M \otimes k[G]$ is considered as a $G$-supermodule via the right regular representation of $G$ on $k[G]$. The induction functor is right adjoint to the restriction functor $\res_H^G: \smod_G \rightarrow \smod_H$ (cf.\ \cite[I.3.4]{Jantzen:2003}), hence is left exact and maps injective objects in $\smod_H$ to injective objects in $\smod_G$. Then one can prove that $\smod_G$ has enough injective objects by following the arguments for ordinary $k$-group schemes \cite[I.3.9]{Jantzen:2003}. If $Q$ is injective in $\smod_G$, then it is also injective in the subcategory $\fsmod_G$, so $\fsmod_G$ has enough injectives as well.

Suppose $G$ is a finite supergroup scheme. Then $\smod_G$ is equivalent to the category of left $k[G]^*$-supermodules. Given $M \in \smod_G$, the action of $\mu \in k[G]^*$ on $m \in M$ is defined by $\mu.m = (1 \otimes \mu) \circ \Delta_M(m)$. Conversely, if $M$ is a left $k[G]^*$-supermodule and if $A \in \csalg_k$, then $M \otimes A$ is naturally a left $(k[G]^* \otimes A)$-supermodule. Since $G(A) = \Hom_{\salg}(k[G],A)$ is a subset of $\Hom_k(k[G],A) \cong A \otimes k[G]^* \cong k[G]^* \otimes A$, we obtain an action of $G(A)$ on $M \otimes A$ for each $A \in \csalg_k$. In general, if $G$ is an arbitrary affine supergroup scheme and $M \in \smod_G$, the formula $\mu.m = (1 \otimes \mu) \circ \Delta_M(m)$ defines an action of $\Dist(G)$ on $M$, making $M$ into a $\Dist(G)$-supermodule. In particular, $M$ becomes a supermodule for $\Lie(G)$.

\subsection{Frobenius kernels} \label{subsection:Frobeniuskernels}

Suppose $k$ is a perfect field of characteristic $p > 2$. Given a $k$-vector space $A$ and an integer $m \in \Z$, let $A^{(m)}$ be the $k$-vector space that is equal to $A$ as an abelian group, but in which the $k$-module structure has been twisted so that $b \in k$ acts on $A^{(m)}$ as $b^{p^{-m}}$ acts on $A$. Now let $A \in \csalg_k$, and let $r \in \N$. Then the map $\gamma_r(A): a \mapsto a^{p^r}$ defines a super\-algebra homomor\-phism $A \rightarrow A^{(-r)}$. Given a $k$-superfunctor $X$, define $X^{(r)}$ to be the $k$-superfunctor satisfying $X^{(r)}(A) = X(A^{(-r)})$ for each $A \in \csalg_k$, and define $F_X^r: X \rightarrow X^{(r)}$ by $F_X^r(A) = X(\gamma_r(A))$. Then $F_X^r$ is called the \emph{$r$-th Frobenius morphism on $X$}. If $r = 1$, then we write $F_X$ for $F_X^1$.

If $X = \Sspk R$, then $X^{(r)} = \Sspk (R^{(r)})$, and the comorphism $(F_X^r)^*: R^{(r)} \rightarrow R$ is the power map $f \mapsto f^{p^r}$. Since $a^2 = 0$ for each $a \in A_{\ol{1}}$ by the assumption $p \neq 2$, it follows that $\gamma_r(A): A \rightarrow A^{(-r)}$ has image in $(A_{\ol{0}})^{(-r)}$, and hence that $F_X^r$ has image in the subfunctor $(X_{\ev})^{(r)}$ of $X^{(r)}$. Here $X_{\ev}$ is the underlying purely even subfunctor of $X$, with $k[X_{\ev}] = k[X]/(k[X]k[X]_{\ol{1}})$ and $(X_{\ev})(A) = X(A_{\ol{0}})$ for each $A \in \csalg_k$. Then we can also consider $F_X^r$ as a morphism $F_X^r: X \rightarrow (X_{\ev})^{(r)}$. The comorphism $(F_X^r)^*: k[X_{\ev}]^{(r)} \rightarrow k[X]$ is induced by the $p^r$-power map from $k[X]^{(r)}$ to $k[X]$.

Let $G$ be an affine algebraic $k$-supergroup functor. Then $F_G^r: G \rightarrow (\Gev)^{(r)}$ is a supergroup homomorphism, and the restriction of $F_G^r$ to $\Gev$ is the ordinary $r$-th Frobenius morphism on $\Gev$. Set $G_r = \ker(F_G^r)$. Then $G_r$ is a normal sub\-super\-group scheme of $G$, called the \emph{$r$-th Frobenius kernel of $G$}. Specifically, $G_r$ is the closed sub\-super\-group scheme of $G$ defined by the ideal $\sum_{f \in I_\ve} k[G] f^{p^r}$. Let $\Iev$ be the augmentation ideal of $k[\Gev]$. Applying \eqref{eq:Masuoka} and the fact that every $p$-th power in $\Lambda({\Lie(G)_{\ol{1}}}^*)$ is zero, it follows that $G_r$ is defined by the ideal $\sum_{f \in \Iev} k[G] f^{p^r}$, and hence that
\begin{equation} \label{eq:MasuokaGr}
\begin{split}
k[G_r] &= \textstyle k[G] / (\sum_{f \in \Iev} k[G] f^{p^r}) \\
&\cong \textstyle k[\Gev]/(\sum_{f \in \Iev} k[\Gev] f^{p^r}) \otimes \Lambda({\Lie(G)_{\ol{1}}}^*) \\
&= k[(\Gev)_r] \otimes \Lambda({\Lie(G)_{\ol{1}}}^*)
\end{split}
\end{equation}
as augmented super\-algebras over $k$. Here $(\Gev)_r$ denotes the $r$-th Frobenius kernel of the ordinary affine group scheme $\Gev$. In particular, $G_r$ is an infinitesimal supergroup scheme of height $r$.

\begin{lemma} \textup{(cf.\ \cite[Lemma 1.3]{Friedlander:1997})} \label{lemma:FSLemma1.3}
Let $G$ be an infinitesimal $k$-supergroup scheme, and let $r \in \N$. Then the following are equivalent:
\begin{enumerate}
\item $G$ is of height $\leq r$.
\item There exists a closed embedding $G \hookrightarrow GL(m|n)_r$ for some $m,n \in \N$.
\item For any closed embedding $G \hookrightarrow GL(m|n)$, the image of $G$ is contained in $GL(m|n)_r$.
\end{enumerate}
\end{lemma}

\begin{proof}
By \cite[Theorem 9.3.2]{Westra:2009}, there exist $m,n \in \N$ such that $G$ is a closed sub\-super\-group scheme of $GL(m|n)$. Then (3) implies (2). Given a closed embedding $G \hookrightarrow GL(m|n)_r$, the corresponding comorphism $k[GL(m|n)_r] \rightarrow k[G]$ is a surjective super\-algebra homomorphism. Since $f^{p^r} = 0$ in $k[GL(m|n)_r]$ for each element $f$ of the augmentation ideal of $k[GL(m|n)_r]$, the corresponding condition must also hold in $k[G]$, and hence $G$ must be of height $\leq r$. Then (2) implies (1). Now suppose $G$ is of height $\leq r$, and that $G \hookrightarrow GL(m|n)$ is a closed embedding. Then the surjective super\-algebra homomorphism $k[GL(m|n)] \rightarrow k[G]$ must factor through the quotient $k[GL(m|n)_r]$ of $k[GL(m|n)]$, which implies that the image of $G$ is contained in $GL(m|n)_r$. Then (1) implies (3).
\end{proof}

The sequence of inclusions $G_1 \subset G_2 \subset \cdots \subset G$ induces a corresponding sequence of inclusions of distribution algebras $\Dist(G_1) \subset \Dist(G_2) \subset \cdots \subset \Dist(G)$. Also, one has canonical identifications $\Lie(G_r) = \Lie(G)$ for each $r \geq 1$. Set $\g = \Lie(G)$, and write $\Vg$ for the restricted enveloping super\-algebra of $\g$. Then the inclusion $\g \subset \Dist(G_1)$ induces a homomorphism of Hopf super\-algebras $\gamma: \Vg \rightarrow \Dist(G_1)$. The next lemma is the characteristic $p$ analogue of \cite[Lemma 3.1]{Zubkov:2009}.

\begin{lemma} \label{lemma:VgDistG1}
The natural homomorphism $\gamma: \Vg \rightarrow \Dist(G_1)$ is an isomorphism.
\end{lemma}

\begin{proof}
Since $G_1$ is infinitesimal, $\Dist(G_1) = k[G_1]^*$. Set $A = k[G_1]$ and set $B = \Vg^*$. One has $\dim_k k[(\Gev)_1] = p^{\dim \g_{\ol{0}}}$ by \cite[I.9.6(4)]{Jantzen:2003}, and $\dim_k \Vg = p^{\dim \g_{\ol{0}}} \cdot 2^{\dim \g_{\ol{1}}}$, so it follows from \eqref{eq:MasuokaGr} that $\dim_k A = \dim_k B$. Then to show that $\gamma$ is an isomorphism, it suffices by dimension comparison to show that the transpose map $\gamma^*: A \rightarrow B$ is surjective. Write $A_\ve$ and $B_\ve$ for the augmentation ideals of $A$ and $B$, respectively. By construction, $\gamma$ restricts to an isomorphism from $\g \subset \Vg$ to the subspace $\Dist_1^+(G_1) \cong (A_\ve/A_\ve^2)^*$ of $\Dist(G_1)$, and by Lemma \ref{lemma:dualfiltration}, $B_\ve/B_\ve^2 \cong \g^*$. Then it follows that the transpose map $\gamma^*: A \rightarrow B$ induces an isomorphism $A_\ve/A_\ve^2 \stackrel{\sim}{\rightarrow} B_\ve/B_\ve^2$, and hence that $\gamma^*$ maps a set of generators for $A$ to a set of generators for $B$. In particular, $\gamma^*$ is surjective.
\end{proof}

\begin{remark}
Let $\g$ be an arbitrary finite-dimensional restricted Lie super\-algebra over $k$. Then by the equivalence of categories between finite-dimensional cocommutative Hopf superalgebras over $k$ and finite $k$-supergroup schemes, there exists a finite $k$-supergroup scheme $G$ with $k[G] = \Vg^*$. It follows from Lemma \ref{lemma:dualfiltration} that $\Lie(G) \cong \g$. Then by Lemma \ref{lemma:VgDistG1}, $\Vg \cong \Dist(G_1) \subset k[G]^*$, so by dimension comparison one has $\Dist(G_1) = k[G]^*$. Equivalently, $k[G_1]^* = k[G]^*$, so $k[G] = k[G_1]$, and hence $G = G_1$. Then under the aforementioned equivalence of categories, restricted enveloping super\-algebras correspond to height one infinitesimal supergroup schemes, and vice versa.
\end{remark}

The following lemma is the supergroup analogue of \cite[I.9.5]{Jantzen:2003}; cf.\ also \cite[Lemma 8.2]{Zubkov:2009}.

\begin{lemma} \label{lemma:G/Gr}
Suppose $\Gev$ is a reduced algebraic $k$-group scheme. Then $F_G^r$ induces isomorphisms $G/G_r \cong (\Gev)^{(r)}$ and $G_{r'}/G_r \cong [(\Gev)^{(r)}]_{r'-r}$ for all $r' \geq r$.\footnote{The quotient $G/G_r$ here is the object denoted $G\tilde{\tilde{/}}G_r$ in \cite{Zubkov:2009}.}
\end{lemma}

\begin{proof}
Considering $F_G^r$ as a morphism $G \rightarrow (\Gev)^{(r)}$, $G/G_r$ is isomorphic by \cite[Theorem 6.1]{Zubkov:2009} to the sub\-super\-group of $(\Gev)^{(r)}$ defined by the kernel of the comorphism $(F_G^r)^*: k[\Gev]^{(r)} \rightarrow k[G]$. Since $\Gev$ is reduced, $(F_G^r)^*$ has trivial kernel, and hence $G/G_r \cong (\Gev)^{(r)}$. Next, one can check that the preimage in $G$ of $[(\Gev)^{(r)}]_{r'-r}$, in the sense of \cite[\S6]{Zubkov:2009}, is equal to $G_{r'}$. Then $G_{r'}/G_r \cong [(\Gev)^{(r)}]_{r'-r}$ by \cite[Corollary 6.2]{Zubkov:2009}.
\end{proof}

\begin{remark} \label{remark:SWfactorization}
Suppose $\Gev$ is a reduced algebraic $k$-group scheme. The above proof shows that $F_G: G \rightarrow (\Gev)^{(1)}$ is an epimorphism. Since $F_G$ restricts to the ordinary Frobenius morphism on $\Gev$, the maps $F_G|_{\Gev}: \Gev \rightarrow (\Gev)^{(1)}$ and $F_G|_{(\Gev)_r}: (\Gev)_r \rightarrow [(\Gev)^{(1)}]_{r-1}$ are also epimorphisms by \cite[I.9.5]{Jantzen:2003}. Now suppose that $\Gev$ is defined over $\Fp$. Then $(\Gev)^{(1)}$ identifies with $\Gev$, and the previous observations, combined with the first two paragraphs of \cite[\S7]{Zubkov:2009}, imply that there exist factorizations $G = \Gev G_1$ and $G_r = (\Gev)_r G_1$. The existence of such factorizations is asserted in \cite[Lemma 2.6]{Shu:2008} without the assumption that $\Gev$ is reduced, though the proof given there is incorrect, since the Frobenius morphism $F_G|_{\Gev}: \Gev \rightarrow \Gev$ need not be an epimorphism if $\Gev$ is not reduced. Indeed, if $\Gev = (\Gev)_1$, then $F_G|_{\Gev}$ is the trivial morphism.
\end{remark}

Let $M$ be a rational $\Gev$-module with comodule structure map $\Delta_M$. Then the twisted vector space $M^{(r)}$ is naturally a rational $(\Gev)^{(r)}$-module with comodule structure map ${\Delta_M}^{(r)}: M^{(r)} \rightarrow (M \otimes k[\Gev])^{(r)} \cong M^{(r)} \otimes k[\Gev]^{(r)}$, and the composition $(\id_{M^{(r)}} \otimes (F_G^r)^*) \circ {\Delta_M}^{(r)}: M^{(r)} \rightarrow M^{(r)} \otimes k[G]$ defines on $M^{(r)}$ the structure of a $G$-supermodule, called the \emph{$r$-th Frobenius twist of the rational $\Gev$-module $M$}. Given $m \in M$, write $m^{(r)}$ to denote $m$ considered as an element of $M^{(r)}$. Now if $A \in \csalg_k$ and $g \in G(A) = \Hom_{\salg}(k[G],A)$, and if $m \in M$ satisfies $\Delta_M(m) = \sum_i m_i \otimes f_i$, then the action of $g$ on $m^{(r)} \otimes 1 \in M^{(r)} \otimes A$ is given by $g.(m^{(r)} \otimes 1) = \sum_i m_i^{(r)} \otimes g(f_i^{p^r})$. Observe that the expression $g(f_i^{p^r})$ makes sense even though $f \in k[\Gev]$, because the $p^r$-power map on $k[G]$ factors through the quotient $k[\Gev] = k[G]/(k[G]k[G]_{\ol{1}})$.

\section{Cohomology for supergroup schemes} \label{section:cohomologysupergroups}

\subsection{Generalities}

Continue to assume that $G$ is an affine supergroup scheme. Since $\fsmod_G$ is an abelian category containing enough injectives, we can apply the machinery of homological algebra to define cohomology groups in $\fsmod_G$. Specifically, given $M \in \fsmod_G$, $\Ext_G^n(M,-)$ is the $n$-th right derived functor of $\Hom_G(M,-): \fsmod_G \rightarrow \fsvec_k$, and $\opH^n(G,-)$ is the $n$-th right derived functor of $(-)^G: \fsmod_G \rightarrow \fsvec_k$. Then $\opH^n(G,-) \cong \Ext_G^n(k,-)$. Reasoning as in \cite[I.4.14--I.4.16]{Jantzen:2003}, $\Hbul(G,M)$ can be computed as the cohomology of the Hochschild complex $\Cbul(G,M) = M \otimes k[G]^{\otimes \bullet}$. The differential $\partial^n: C^n(G,M) \rightarrow C^{n+1}(G,M)$ is defined by
\begin{equation}
\begin{split}
\partial^n(m \otimes f_1 \otimes \cdots \otimes f_n) &= \Delta_M(m) \otimes f_1 \otimes \cdots \otimes f_n + (-1)^{n+1} m \otimes f_1 \otimes \cdots \otimes f_n \otimes 1 \\
&\phantom{=} \textstyle + \sum_{i=1}^n (-1)^i m \otimes f_1 \otimes \cdots \otimes \Delta_G(f_i) \otimes \cdots \otimes f_n.
\end{split}
\end{equation}

\begin{remark} \label{remark:Hochschildcobar}
Let $G$ be a finite supergroup scheme. Then $C^\bullet(G,M)$ identifies under the equivalence of categories between $G$- and $k[G]^*$-supermodules with the cobar complex $C^\bullet(k[G]^*,M)$, and $\Hbul(G,M)$ identifies with the cohomology group $\Hbul(k[G]^*,M)$; cf.\ \eqref{eq:coproductdifferential}. In particular, $\Hbul(G,k)$ is a graded-commutative graded superalgebra by Remark \ref{remark:braidedgraded}.
\end{remark}

As in Section \ref{subsection:barcomplex}, the right cup product action of $C^\bullet(G,k)$ on $C^\bullet(G,M)$ is induced by right concatenation. Passing to cohomology, this yields the cup product action of $\Hbul(G,k)$ on $\Hbul(G,M)$. If $G'$ is a supergroup scheme acting by automorphisms on $G$, then $k[G]$ is a $G'$-supermodule. If $M$ is a $G \rtimes G'$-supermodule, then the action of $G'$ on $k[G]$ induces an action of $G'$ on $\Cbul(G,M)$, which passes to an action of $G'$ on $\Hbul(G,M)$. In particular, if $N$ is normal in $G$, then the conjugation action of $G$ on $N$ induces for each $M \in \fsmod_G$ a $G/N$-supermodule structure on $\Hbul(N,M)$, and there exists a Lyndon--Hochschild--Serre (LHS) spectral sequence (cf.\ \cite[\S3]{Scala:2012}):
\begin{equation} \label{eq:LHSspecseq}
E_2^{i,j}(M) = \opH^i(G/N,\opH^j(N,M)) \Rightarrow \opH^{i+j}(G,M).
\end{equation}

The LHS spectral sequence can be constructed as in \cite[Remark I.6.6]{Jantzen:2003} as the spectral sequence associated to a double complex. Specifically, consider $M$ as a complex concentrated in degree $0$, and consider $M \otimes \Cbul(G,k[G]) \otimes \Cbul(G/N,k[G/N])$ as the tensor product of complexes. Then \eqref{eq:LHSspecseq} is the spectral sequence associated to the double complex
\begin{align*}
C^{i,j}(M) &= [M \otimes C^j(G,k[G]) \otimes C^i(G/N,k[G/N])]^G \\
&= [(M \otimes k[G]^{\otimes (j+1)})^N \otimes k[G/N]^{\otimes (i+1)}]^{G/N}
\end{align*}
that is obtained by first computing cohomology along columns, and then computing cohomology along rows. Here $C^n(G,k[G]) = k[G]^{\otimes (n+1)}$ is considered as a $G$-super\-module via the left regular representation of $G$ on the first factor in $k[G]^{\otimes (n+1)}$. Similarly, $C^n(G/N,k[G/N])$ is a $G/N$-super\-module, hence a $G$-supermodule via the quotient morphism $G \rightarrow G/N$. The above construction allows one to deduce, exactly as for ordinary group schemes, that there exists a product structure on \eqref{eq:LHSspecseq} such that if $M,M' \in \fsmod_G$, then the product $E_2^{i,j}(M) \otimes E_2^{r,s}(M') \rightarrow E_2^{i+r,j+s}(M \otimes M')$ identifies with the ordinary cup product
\[
\opH^i(G/N,\opH^j(N,M)) \otimes \opH^r(G/N,\opH^s(N,M')) \rightarrow \opH^{i+r}(G/N,\opH^{j+s}(N,M \otimes M')
\]
multiplied by the scalar factor $(-1)^{j \cdot r}$. One can also deduce that the horizontal edge map $E_2^{i,0}(M) \rightarrow E_\infty^{i,0}$ identifies with the inflation map $\opH^i(G/N,M^N) \rightarrow \opH^i(G,M)$ arising from the quotient $G \rightarrow G/N$, and that the vertical edge map $\opH^i(G,M) \rightarrow E_2^{0,i}(M)$ identifies with the restriction map $\opH^i(G,M) \rightarrow \opH^i(N,M)$ induced by the inclusion $N \subset G$.

\subsection{The May spectral sequence for \texorpdfstring{$G_1$}{G1}} \label{subsection:MayspecseqG1}

In this section we discuss the analogue for $G_1$ of the spectral sequence constructed in Section \ref{subsection:Mayspecseq}. In fact, applying Remark \ref{remark:Hochschildcobar}, Lemma \ref{lemma:VgDistG1}, and \eqref{eq:Vcobarfiltration}, the spectral sequence discussed here identifies with the one constructed in Section \ref{subsection:Mayspecseq}.

\begin{proposition} \label{proposition:G1specseqE0}
Let $k$ be a perfect field of characteristic $p > 2$, let $G$ be an affine algebraic $k$-super\-group scheme, and let $M$ be a $G$-supermodule. Set $\g = \Lie(G)$. Then there exists a spectral sequence of $G$-supermodules converging to $\Hbul(G_1,M)$ with
\begin{equation} \label{eq:G1specseqE0M}
E_0^{i,j}(M) \cong M \otimes \Lambda_s^{j-i}(\g^*) \otimes S^i({\g_{\ol{0}}}^*)^{(1)}.
\end{equation}
Moreover, $E_r(k)$ is a spectral sequences of $G$-supermodule algebras, and $E_r(M)$ is a spectral sequence of right modules over $E_r(k)$.
\end{proposition}

\begin{proof}
The spectral sequence is constructed as in \cite[I.9.12]{Jantzen:2003}. Specifically, write $I_\ve$ for the augmentation ideal of $k[G_1]$. The powers of $I_\ve$ define a multiplicative filtration on $k[G_1]$, and this leads to a filtration of the Hochschild complex $\Cbul(G_1,M)$, defined by replacing $V$ by $G_1$ in \eqref{eq:Vcobarfiltration}. Then the desired spectral sequence $E_r(M) \Rightarrow \Hbul(G_1,M)$ is the one arising via \cite[Theorem 2.6]{McCleary:2001} from this filtration on $\Cbul(G_1,M)$. Since $I_\ve$ is a $G$-subsupermodule for the conjugation action of $G$ on $k[G_1]$, the filtration on $\Cbul(G_1,M)$ is a filtration by $G$-subsupermodules, which implies that $E_r(M)$ is a spectral sequence of $G$-supermodules, and that $E_r(k)$ is a spectral sequence of $G$-supermodule algebras. Now by Remark \ref{remark:Hochschildcobar}, Lemma \ref{lemma:VgDistG1}, and \eqref{eq:Vcobarfiltration}, the spectral sequence $E_r(M) \Rightarrow \Hbul(G_1,M)$  identifies with that constructed in Section \ref{subsection:Mayspecseq} in the case $L = \g$. In particular, by the discussion preceding \eqref{eq:E0V(L)specseq}, the spectral sequence can be reindexed so that $E_0^{i,j}(M)$ is as described in \eqref{eq:G1specseqE0M}. Now to finish the proof, it remains to show that \eqref{eq:G1specseqE0M} is an isomorphism of $G$-supermodules.

Since the $G$-supermodule structure is compatible with the algebra structure on $E_0(k)$ and with the $E_0(k)$-module structure of $E_0(M)$, it suffices to show that the identifications $E_0^{0,0}(M) \cong M$, $E_0^{0,1}(k) \cong \g^*$, and $E_0^{1,1}(k) \cong ({\g_{\ol{0}}}^*)^{(1)}$ are $G$-supermodule isomorphisms. The argument proceeds as in \cite[I.9.18, I.9.20]{Jantzen:2003}. First, $E_0^{0,0}(M)$ identifies as a $G$-super\-module with
\[
F^0 C^0(G_1,M)/F^1 C^0(G_1,M) = F^0 C^0(G_1,M) \cong M.
\]
Next, $E_0^{0,1}(k)$ identifies as a $G$-supermodule with
\[
F^1 C^1(G_1,k)/F^2 C^1(G_1,k) = I_\ve/I_\ve^2 \cong \g^*.
\]
Finally, $E_0^{1,1}(k)$ identifies with a $G$-subsupermodule of
\begin{equation} \label{eq:E011k}
F^p C^2(G_1,k)/(F^{p+1} C^2(G_1,k) + \partial(F^p C^1(G_1,k))).
\end{equation}
Set $Z = F^p C^2(G_1,k)$, and set $W = F^{p+1} C^2(G_1,k) + \partial(F^p C^1(G_1,k))$. Choose $f_1,\ldots,f_n \in (I_\ve)_{\ol{0}}$ such that $\{f_1 + I_\ve^2,\ldots,f_n+I_\ve^2\}$ forms a basis for $(I_\ve/I_\ve^2)_{\ol{0}}$. Then by the proof of Proposition \ref{proposition:E1page}, the cosets of $\beta(f_1),\ldots,\beta(f_n)$ in $Z/W$ form a basis for $E_0^{1,1}(k)$. The function $f_i + I_\ve^2 \mapsto \beta(f_i) + W$ extends to a linear map $\betabar: (I_\ve/I_\ve^2)_{\ol{0}} \rightarrow Z/W$. Since $\beta(\lambda f) = \lambda^p \beta(f)$ for all $\lambda \in k$, we can consider $\betabar$ as a linear map with image isomorphic to $({\gzero}^*)^{(1)}$. Then we wish to show that the image of $\betabar$ is isomorphic as a $G$-supermodule to $({\gzero}^*)^{(1)}$.

Extend the set $\set{f_1,\ldots,f_n}$ to a set $\set{f_1,\ldots,f_m}$ of homogeneous elements in $I_\ve$ such that $\{f_1+I_\ve^2,\ldots,f_m+I_\ve^2\}$ forms a basis for $I_\ve/I_\ve^2$. Let $\Delta: k[G_1] \rightarrow k[G_1] \otimes k[G]$ be the comodule structure map for the conjugation action of $G$ on $G_1$. Then there exist $f_{ij} \in k[G]$ with $\ol{f_j} + \ol{f_{ij}} = \ol{f_i}$ such that $\Delta(f_i) \in (\sum_{j=1}^m f_j \otimes f_{ij})+ I_\ve^2 \otimes k[G]$. Let $A \in \csalg_k$, and let $g \in G(A) = \Hom_{\salg}(k[G],A)$. We identify $Z/W \otimes A$ with $(Z \otimes A)/(W \otimes A)$, and similarly we identify $\g^* \otimes A \cong (I_\ve/I_\ve) \otimes A$ with $(I_\ve \otimes A)/(I_\ve^2 \otimes A)$. Then the action of $g$ on $f_i \otimes 1 + I_\ve^2\otimes A \in \g^* \otimes A$ is given by
\begin{equation} \label{eq:adjointaction}
\textstyle g.(f_i \otimes 1 + I_\ve^2\otimes A) = (\sum_{j=1}^m f_j \otimes g(f_{ij})) + I_\ve^2 \otimes A.
\end{equation}
It is convenient to consider $\Gev(A)$ as consisting of those $g \in G(A)$ that vanish on $k[G]_{\ol{1}}$. Then given $g \in \Gev(A)$ and $1 \leq i \leq n$, one has $g(f_{ij}) = 0$ for $n+1 \leq j \leq m$, since then $\ol{f_{ij}} = \ol{f_i} - \ol{f_j} = \ol{0} - \ol{1} = \ol{1}$. With these conventions, if $1 \leq i \leq n$, then \eqref{eq:adjointaction} describes the action of $\Gev(A)$ on ${\gzero}^*$. Now our goal for $1 \leq i \leq n$ is to show that
\begin{equation} \label{eq:Gactiongoal} \textstyle
g.(\beta(f_i) \otimes 1 + W \otimes A) = (\sum_{j=1}^n \beta(f_j) \otimes g(f_{ij})^p) + W \otimes A 
\end{equation}

Let $A \in \csalg_k$, let $g \in G(A)$, and let $1 \leq i \leq n$. Set $a_{ij} = g(f_{ij})$. Applying \eqref{eq:VWbimodiso} in the case $V = W = k[G_1]$, and using the fact that $G$ acts by algebra automorphisms on $k[G_1]$, i.e., that $G(A)$ acts by $A$-super\-algebra automorphisms on $k[G_1] \otimes A$ for each $A \in \csalg_k$, one obtains
\begin{equation} \label{eq:g.beta} \textstyle
g.(\betabar(f_i) \otimes 1 + W \otimes A) = ( \sum_{r=1}^{p-1} \cbinom{p}{i}(\sum_{s=1}^m f_s \otimes a_{is})^r \otimes_A (\sum_{s=1}^m f_s \otimes a_{is})^{p-r} ) + W \otimes A.
\end{equation}
There exists an obvious analogue of the Hochschild complex for the $A$-supergroup scheme $G_{1,A}$ obtained via base change to $A$ from $G_1$. Specifically, $C^n(G_{1,A},A) = A[G_{1,A}]^{\otimes n} \cong C^n(G_1,k) \otimes A$, and the differential $\partial_A$ on $\Cbul(G_{1,A},A)$ is induced by the coproduct on $A[G_{1,A}] = k[G_1] \otimes A$. The filtration on $\Cbul(G_1,k)$ also induces a corresponding filtration on $\Cbul(G_{1,A},A)$. Now there exists an evident analogue $\beta_A: C^1(G_{1,A},A) \rightarrow C^2(G_{1,A},A)$ of the function $\beta$. Using this function, \eqref{eq:g.beta} can be rewritten as $g.(\betabar(f_i) \otimes 1 + W \otimes A) = \beta_A(\sum_{s=1}^m f_s \otimes a_{is}) + W \otimes A$. Each term in the sum $\sum_{s=1}^m f_s \otimes a_{is}$ is homogeneous of degree $\ol{0}$, and satisfies $\partial_A(f_s \otimes a_{is}) \in (I_{\ve,A})^2$ because $f_s \otimes a_{is} \in I_{\ve,A}$. Here $I_{\ve,A}$ is the augmentation ideal of $A[G_{1,A}]$. Let $h_1,h_2 \in (I_{\ve,A})_{\ol{0}}$. Then in a generalization of \eqref{eq:betaadditive}, one obtains
\[ \textstyle
\beta_A(h_1+h_2) \in \beta_A(h_1) + \beta_A(h_2) - \partial_A( \sum_{i=1}^{p-1} \cbinom{p}{i} h_1^i h_2^{p-i}) + F^{p+1} C^2(G_{1,A},A).
\]
Since $W \otimes A \cong F^{p+1}C^2(G_{1,A},A) + \partial_A(F^pC^1(G_{1,A},A))$, it then follows that
\[ \textstyle
g.(\betabar(f_i) \otimes 1 + W \otimes A) = \sum_{s=1}^m \beta_A(f_s \otimes a_{is}) + W.\otimes A
\]
Now consider $\beta_A(f_s \otimes a_{is}) = \sum_{r=1}^{p-1} \cbinom{p}{r} (f_s \otimes a_{is})^r \otimes_A (f_s \otimes a_{is})^{p-r}$. If $r > 1$, and if either $\ol{f_s} = \ol{1}$ or $\ol{a_{is}} = \ol{1}$, then $(f_s \otimes a_{is})^r = \pm (f_s^r \otimes a_{is}^r) = 0$, because odd elements in a commutative $k$-superalgebra square to $0$ by the assumption $p > 2$. Then it follows for $n+1 \leq s \leq m$ that $\beta_A(f_s \otimes a_{is}) = 0$. We may also assume for $1 \leq s \leq n$ that $\ol{f_s} = \ol{a_{is}} = \ol{0}$. Then given $1 \leq s \leq n$, and applying \eqref{eq:VWbimodiso} in the case $V = W = k[G_1]$, one has
\begin{align*}
\beta_A(f_s \otimes a_{is}) &= \textstyle \sum_{r=1}^{p-1} \cbinom{p}{r} (f_s^r \otimes a_{is}^r) \otimes_A (f_s^{p-r} \otimes a_{is}^{p-r}) \\
&= \textstyle \sum_{r=1}^{p-1} \cbinom{p}{r} (f_s^r \otimes f_s^{p-r}) \otimes a_{is}^p \\
&= \beta(f_s) \otimes g(f_{is})^p
\end{align*}
Then \eqref{eq:Gactiongoal} holds, and the image of $\betabar$ is isomorphic as a $G$-supermodule to $({\gzero}^*)^{(1)}$.
\end{proof}

Given a vector space $V$, write $V(i)$ for $V$ considered as a graded space concentrated in external degree $i$. Thus, in the notation of Proposition \ref{proposition:E1page}, $S(V(2)) = S'(V)$. One has $S^j(V(i)) = 0$ unless $j \in i \Z$, and $S^{ji}(V(i)) \cong S^j(V(1))$ as $k$-spaces.

\begin{lemma} \label{lemma:Gisog1poly}
In the spectral sequence \eqref{eq:G1specseqE0M}, the subalgebra $S({\gone}^*)^p$ of $E_0(k)$ is isomorphic as a $G$-supermodule to $S({\gone}^*(p))^{(1)}$.
\end{lemma}

\begin{proof}
Fix a basis $x_1,\ldots,x_t$ for the subspace ${\gone}^* = (\g^*)_{\ol{1}}$ of $E_0^{0,1}(k)$. Then $x_1^p,\ldots,x_t^p$ generate the subalgebra $S({\gone}^*)^p$ of $\Lambda_s(\g^*)$. Now extend $x_1,\ldots,x_t$ to a homogeneous basis $x_1,\ldots,x_m$ for $\g^*$, so that $x_{t+1},\ldots,x_m$ is a basis for ${\gzero}^*$, and suppose that the structure map $\Delta_{\g^*}: \g^* \rightarrow \g^* \otimes k[G]$ for the action of $G$ on $\g^*$ is given by $\Delta_{\g^*}(x_i) = \sum_{s=1}^m x_s \otimes f_{is}$. Let $A \in \csalg_k$, let $g \in G(A)$, and set $a_{is} = g(f_{is})$. Then $g.(x_i^p \otimes 1) = (\sum_{s=1}^m x_s \otimes a_{is})^p$. Here the product is computed in the $A$-superalgebra $\Lambda_s(\g^*) \otimes A$. Now suppose that $1 \leq i \leq t$. Then $\ol{a_{is}} = \ol{0}$ if $1 \leq s \leq t$, and $\ol{a_{is}} = \ol{1}$ otherwise, so it follows that the $x_s \otimes a_{is}$ commute (in the ordinary sense) in $\Lambda_s(\g^*)$. Then $g.(x_i^p \otimes 1) = \sum_{s=1}^m x_s^p \otimes a_{is}^p = \sum_{s=1}^t x_s^p \otimes a_{is}^p$. The second equality holds because $x_s^2 = 0$ in $\Lambda_s(\g^*)$ if $\ol{x_s} = \ol{0}$. Thus, $S({\gone}^*)^p \cong S({\gone}^*(p))^{(1)}$ as $G$-super\-modules.
\end{proof}

Let $G$ be an affine supergroup scheme, and let $M$ be a $G$-supermodule. Then $M$ is naturally a $\Dist(G)$-supermodule, and the $G$- and $\Dist(G)$-supermodule structures are compatible in the sense that the $\Dist(G)$-supermodule structure map $\Dist(G) \otimes M \rightarrow M$ is a $G$-supermodule homomorphism. Here $G$ acts on $\Dist(G)$ via the adjoint representation. This compatibility descends to the action of $\g = \Lie(G)$ on $M$. Now the Koszul resolution $Y(\g)$ described in Section \ref{subsection:Koszulresolution} is a complex of $G$-supermodules, so it follows that $\Hbul(\g,M)$ is naturally a $G$-supermodule.

\begin{corollary} \label{corollary:FPG1specseqE2}
In the spectral sequence of Proposition \ref{proposition:G1specseqE0}, there exists a $G$-supermodule isomorphism $E_1^{i,j}(M) \cong \opH^{j-i}(\g,M) \otimes S^i({\gzero}^*)^{(1)}$.  The spectral sequence can be reindexed so that
\begin{equation} \label{eq:FPG1specseqE0}
E_0^{i,j}(M) = M \otimes \Lambda_s^j(\g^*) \otimes S^i({\gzero}^*(2))^{(1)} \Rightarrow \opH^{i+j}(G_1,M),
\end{equation}
and $E_2^{i,j}(M) = \opH^j(\g,M) \otimes S^i({\gzero}^*(2))^{(1)}$.
\end{corollary}

\begin{proof}
The differentials on the $E_0$-page of the spectral sequence are $G$-supermodule homomorphisms, and by the proof of Theorem \ref{theorem:V(L)specseqE1page}, they identify with the differentials in the complex $\Cbul(\g,M)$. Thus, we conclude that the isomorphism $E_1^{i,j}(M) \cong \opH^{j-i}(\g,M) \otimes S^i({\gzero}^*)^{(1)}$ of Theorem \ref{theorem:V(L)specseqE1page} is an isomorphism of $G$-supermodules. The last claim follows from Remark \ref{remark:FPreindexing}.
\end{proof}

Consider the spectral sequence of Corollary \ref{corollary:FPG1specseqE2}. By Proposition \ref{proposition:permanentcycles} and Lemma \ref{lemma:Gisog1poly}, $E_0(k)$ is finitely-generated over a subalgebra of permanent cycles that is isomorphic as a $G$-supermodule to $S({\gone}^*(p))^{(1)} \otimes S({\gzero}^*(2))^{(1)}$. The horizontal edge map of \eqref{eq:FPG1specseqE0} provides a homomorphism of graded $G$-superalgebras $S({\gzero}^*(2)) \rightarrow \Hbul(G_1,k)$. On the other hand, $S({\gone}^*(p))^{(1)}$ is concentrated in the first column of \eqref{eq:FPG1specseqE0}, and by standard properties of first quadrant spectral sequences, the permanent cycles in the first column consist precisely of the elements in the image of the vertical edge homomorphism. This implies that $S({\gone}^*(p))^{(1)}$ is the homomorphic image of a subalgebra of $\Hbul(G_1,k)$. Since $\Hbul(G_1,k)$ is graded-commutative by Remark \ref{remark:Hochschildcobar}, and since $S({\gone}^*(p))^{(1)}$ is a free graded-commutative graded superalgebra, it follows that there exists a graded subalgebra of $\Hbul(G_1,k)$ that maps isomorphically as an algebra (though not necessarily as a $G$-super\-module) onto $S({\gone}^*(p))^{(1)}$ via the vertical edge map. Equivalently, there exists a homomorphism of graded superalgebras $S({\gone}^*(p))^{(1)} \rightarrow \Hbul(G_1,k)$ whose composition with the vertical edge map has image equal to the subalgebra $S({\gone}^*)^p$ of $E_0(k)$.

Combining the map $S({\gone}^*(p))^{(1)} \rightarrow \Hbul(G_1,k)$ with the edge map $S({\gzero}^*(2)) \rightarrow \Hbul(G_1,k)$, there exists a homomorphism of graded superalgebras
\begin{equation} \label{eq:polyhomtoG1}
S({\gone}^*(p))^{(1)} \otimes S({\gzero}^*(2))^{(1)} \rightarrow \Hbul(G_1,k).
\end{equation}
Moreover, by the proof of Theorem \ref{theorem:V(L)cohomfingen}, it follows for each finite-dimensional $G_1$-supermodule $M$ that $\Hbul(G_1,M)$ is finitely-generated over the image of \eqref{eq:polyhomtoG1}. For each $r \in \N$, the $p^{r-1}$-power map defines $G$-supermodule homomorphisms
\[
({\gzero}^*)^{(r)}(2) \rightarrow S^{2p^{r-1}}({\gzero}^*(2))^{(1)} \quad \text{and} \quad ({\gone}^*)^{(r)}(p) \rightarrow S^{p^r}({\gone}^*(p))^{(1)}.
\]
Combined with \eqref{eq:polyhomtoG1}, these maps induce a homomorphism of graded superalgebras
\begin{equation} \label{eq:rpolyhomtoG1}
S({\gone}^*(p^r))^{(r)} \otimes S({\gzero}^*(2p^{r-1}))^{(r)} \rightarrow \Hbul(G_1,k)
\end{equation}
over which $\Hbul(G_1,M)$ is also finitely-generated.

\subsection{Reduction to infinitesimal supergroup schemes} \label{subsection:reductioninfinitesimal}

In this section we describe how the cohomological finite generation problem for finite supergroup schemes reduces to the special case of infinitesimal supergroup schemes. Recall that a finite group scheme $G$ is \emph{etale} if $k[G]$ is a separable $k$-algebra. By \cite[Theorem 6.7]{Waterhouse:1979}, each affine algebraic group scheme $H$ admits a maximal etale quotient group $\pi_0(H)$ (the component group of $H$).

\begin{lemma} \label{lemma:etalequotient}
Let $k$ be a perfect field of characteristic $p > 2$, and let $G$ be an affine algebraic $k$-super\-group scheme. Then there exists an etale group scheme $\pi_0(G) = \pi_0(G/G_1)$, and a normal sub\-super\-group scheme $G^0$ of $G$ such that $G/G^0 \cong \pi_0(G)$. If $G$ is finite, then $G \cong G^0 \rtimes \pi_0(G)$, and $G^0$ is infinitesimal.
\end{lemma}

\begin{proof}
Set $H = G/G_1$. By \cite[Theorem 6.1]{Zubkov:2009}, $H$ identifies with the affine $k$-supergroup scheme that has coordinate superalgebra $k[H] = \im(F_G^*) = \set{f^p: f \in k[G]} = \set{f^p : f \in k[G]_{\ol{0}}}$. Under this identification, the quotient $F_G: G \rightarrow H$ is induced by the inclusion $\im(F_G^*) \subset k[G]$. Since $k[H]$ is a purely even superalgebra, we can consider $H$ as an ordinary affine algebraic $k$-group scheme. Then by \cite[Theorem 6.7]{Waterhouse:1979}, there exists a normal subgroup scheme $H^0 \unlhd H$ (the connected component of $H$ containing the identity) and an etale group scheme $\pi_0(H)$ such that $H/H^0 \cong \pi_0(H)$. Specifically, $\pi_0(H)$ is the affine group scheme with coordinate algebra $\pi_0(k[H])$, the largest separable subalgebra of $k[H]$, and the quotient $H \rightarrow \pi_0(H)$ is induced by the inclusion $\pi_0(k[H]) \subset k[H]$. Set $\pi_0(G) = \pi_0(H)$, and let $G^0$ be the kernel of the composite morphism $G \rightarrow H \rightarrow \pi_0(H)$. Then $G^0$ is a normal subsupergroup scheme of $G$, and $G/G^0 \cong \pi_0(G)$. Considering $\pi_0(k[H])$ as a subalgebra of $k[G]$, $G^0$ is the closed subsupergroup scheme of $G$ defined by the ideal $k[G](\pi_0(k[H]) \cap I_\ve)$.

Since $k[G]$ is a commutative superalgebra, and since $\chr(k) \neq 2$, the elements of $k[G]_{\ol{1}}$ square to zero. Then given $\fzero \in k[G]_{\ol{0}}$ and $\fone \in k[G]_{\ol{1}}$, one has $(\fzero + \fone)^n = \fzero^n + n \fzero^{n-1} \fone$, and from this it follows that the nilpotent elements in $k[G]$ form a superideal, called the nilradical of $k[G]$ and denoted $\Nil(k[G])$. Set $k[\Gred] = k[G]/\Nil(k[G])$. Since $k[G]_{\ol{1}} \subset \Nil(k[G])$, $k[\Gred]$ is a purely even reduced $k$-algebra. Since $k$ is perfect, the tensor product of algebras $k[\Gred] \otimes k[\Gred]$ is also reduced \cite[V.15.5]{Bourbaki:2003}, and it follows that the composite homomorphism
\[
k[G] \stackrel{\Delta_G}{\rightarrow} k[G] \otimes k[G] \rightarrow k[\Gred] \otimes k[\Gred]
\]
factors through $k[\Gred]$. Similarly, the other Hopf superalgebra structure maps on $k[G]$ descend to maps on $k[\Gred]$, inducing on $k[\Gred]$ the structure of a commutative Hopf algebra over $k$, with corresponding affine $k$-group scheme $\Gred$. Observe that $k[\Gred]$ is a quotient of $k[\Gev]$, so that $\Gred$ is a closed subgroup scheme of $\Gev$. More precisely,
\begin{equation} \label{eq:k[Gred]}
k[\Gred] \cong k[\Gev]/\Nil(k[\Gev]).
\end{equation}

Now suppose that $G$ is finite. We claim that the composite homomorphism $\Gred \hookrightarrow G \rightarrow \pi_0(G)$ is an isomorphism. Equivalently, we claim that the composite comorphism $\pi_0(k[H]) \hookrightarrow k[G] \twoheadrightarrow k[\Gred]$ is an isomorphism. Since $k$ is perfect, a commutative $k$-algebra is separable if and only if it is reduced \cite[V.15.5]{Bourbaki:2003}. Then $\pi_0(k[H])$ is reduced, and the composite $\pi_0(k[H]) \hookrightarrow k[G] \twoheadrightarrow k[\Gred]$ is an injection. Next, recall that $k[H] = \im(F_G^*) \cong k[\Gev]^{(1)}/\ker(F_G^*)$. Since $\ker(F_G^*)$ consists of nilpotent elements in $k[\Gev]^{(1)}$, it follows from \eqref{eq:k[Gred]} and \cite[Corollary 6.8]{Waterhouse:1979} that $\dim_k \pi_0(k[H]) = \dim_k k[\Gred]$. Then the composite $\pi_0(k[H]) \hookrightarrow k[G] \twoheadrightarrow k[\Gred]$ must be an isomorphism. Thus, $\Gred$ maps isomorphically onto $\pi_0(G)$, and it follows that $G \cong G^0 \rtimes \Gred \cong G^0 \rtimes \pi_0(G)$.

Finally, it remains to show that $G$ being finite implies that $G^0$ is infinitesimal. Since $G$ is finite, $k[G_0]$ is finite-dimensional, hence (left or right) artinian. Then to prove that $G^0$ is infinitesimal, it suffices to show that the augmentation ideal of $k[G^0]$ is a nil ideal. This is equivalent to showing that if $f \in I_\ve$, then some power of $f$ is an element of the ideal $k[G](\pi_0(k[H]) \cap I_\ve)$ defining $G^0$. Write $I_H$ for the augmentation ideal of $k[H]$. Then $H^0$ has coordinate algebra $k[H]/(k[H](\pi_0(k[H]) \cap I_H))$. Since $G$ is finite, so is $H$, and hence the connected component $H^0$ of $H$ is an infinitesimal group scheme; cf.\ \cite[\S11.4]{Waterhouse:1979}. Now given $f \in I_\ve$, one has $f^p \in I_\ve \cap \im(F_G^*) \subset I_H$. Then by the fact that $H^0$ is infinitesimal, there exists $n \in \N$ such that $(f^p)^n \in k[H](\pi_0(k[H]) \cap I_H) \subset k[G](\pi_0(k[H]) \cap I_\ve)$. Thus, $G^0$ is an infinitesimal supergroup scheme.
\end{proof}

\begin{corollary}
Let $k$ be a perfect field of characteristic $p > 2$, and let $G$ be a finite $k$-supergroup scheme such that $G/G_1$ is infinitesimal. Then $G$ is infinitesimal. 
\end{corollary}

\begin{proof}
Since $G/G_1$ is infinitesimal, $\pi_0(G) = \pi_0(G/G_1)$ is trivial, and $G = G^0$ is infinitesimal.
\end{proof}

Recall that $\Hev(G,k)$ denotes the subalgebra of $\Hbul(G,k)$ of elements of even external degree.

\begin{theorem} \label{theorem:reduceinfinitesimal}
Let $k$ be a perfect field of characteristic $p > 2$, and let $G$ be a finite $k$-supergroup scheme. Write $G = G^0 \rtimes \pi_0(G)$ as in Lemma \ref{lemma:etalequotient}. Suppose $\Hbul(G^0,k)$ is a finitely-generated $k$-algebra, and for each finite-dimensional $G$-supermodule $M$ that $\Hbul(G^0,M)$ is finite over $\Hbul(G^0,k)$. Then $\Hbul(G,k)$ is a finitely-generated $k$-algebra, and $\Hbul(G,M)$ is finite over $\Hbul(G,k)$.
\end{theorem}

\begin{proof}
The strategy for the proof is the same as in \cite[p.\ 221]{Friedlander:1997}. By \cite[\S6.7]{Waterhouse:1979}, there exists a field extension $k'$ of $k$ such that $\pi_0(G)_{k'}$ is a constant group scheme in the sense of \cite[\S2.3]{Waterhouse:1979}. One has $G_{k'} = (G^0)_{k'} \rtimes \pi_0(G)_{k'}$, and as for ordinary group schemes, one has $\Hbul(G,k) \otimes k' \cong \Hbul(G_{k'},k')$ as algebras; this can be checked by applying the functor $-\otimes k'$ to the Hochschild complex for $G$. The field extension $k'$ can be chosen to be perfect (e.g., by taking $k' = \ol{k}$), so without loss of generality we may assume that $\pi_0(G)$ is a constant group scheme. Then $k[\pi_0(G)]^*$ is the group algebra over $k$ for the abstract finite group $\pi := \pi_0(G)(k)$, and the representation theory of $\pi_0(G)$ as a finite (super)group scheme is equivalent to the representation theory of $\pi$ as an abstract finite group.

Now write $G = G^0 \rtimes \pi$, and consider the pair of LHS spectral sequences
\begin{align}
E_2^{i,j}(k) &= \opH^i(\pi,\opH^j(G^0,k)) \Rightarrow \opH^{i+j}(G,k), \label{eq:LHSk} \text{ and} \\
E_2^{i,j}(M) &= \opH^i(\pi,\opH^j(G^0,M)) \Rightarrow \opH^{i+j}(G,M). \label{eq:LHSM}
\end{align}
Since $\Hbul(G^0,k)$ is a graded-commutative graded superalgebra by Remark \ref{remark:Hochschildcobar}, the sub\-algebra $\Hev(G^0,k)_{\ol{0}}$ is commutative in the ordinary sense. Moreover, the assumptions of the theorem imply that $\Hev(G^0,k)_{\ol{0}}$ is a finitely-generated $k$-algebra, and that $\Hbul(G^0,k)$ and $\Hbul(G^0,M)$ are finite over $\Hev(G^0,k)_{\ol{0}}$. Also, $\Hev(G^0,k)_{\ol{0}}$ is a $\pi$-submodule of $\Hev(G^0,k)$ by Remark \ref{remark:evensubmodule}.

Set $R = \Hev(G^0,k)_{\ol{0}}$. By \cite[V.1.9]{Bourbaki:1998}, $R$ is a finitely-generated $R^\pi$-module, and $R^\pi$ is a finitely-generated $k$-algebra. Then $\Hbul(G^0,M)$ is also finite over $R^\pi$. Choose a finite set of homogeneous generators $\xi_i \in R^\pi$, and suppose $\pi$ has order $n$. Then arguing as in the proof of \cite[Lemma 1.9]{Friedlander:1997}, it follows that there exist cohomology classes $\eta_i \in \Hev(G,k)_{\ol{0}}$ such that $\eta_i|_{G^0} = \xi_i^n$. Let $A$ be the subalgebra of $\Hev(G,k)_{\ol{0}}$ generated by the $\eta_i$, and let $B = \Hev(\pi,k)$. Then $B$ is a finitely-generated commutative $k$-algebra by \cite[Corollary 6.2]{Evens:1961}. Now $\pi$ acts trivially on $A$, and $\Hbul(G^0,M)$ is a finite $A$-module via the restriction map from $\Hbul(G,k)$ to $\Hbul(G^0,k)$. Then $E_2(M)$ is a finite module over $B \otimes A \cong \Hev(\pi,A)$ by \cite[Theorem 6.1]{Evens:1961}. Now applying \cite[Lemma 1.6]{Friedlander:1997} to \eqref{eq:LHSk} and \eqref{eq:LHSM}, we conclude that $\Hbul(G,M)$ is a finite $B \otimes A$-module, and hence also a finite $\Hev(G,k)$-module. In particular, taking $M = k$, $\Hbul(G,k)$ is a finitely-generated $k$-algebra.
\end{proof}

\subsection{Cohomology for infinitesimal supergroup schemes} \label{subsection:cohomologyinfinitesimal}

Theorem \ref{theorem:reduceinfinitesimal} shows that the cohomological finite generation problem for finite supergroup schemes reduces to the special case of infinitesimal supergroup schemes. In this section we discuss, in analogy to the situation for ordinary infinitesimal group schemes, how the problem for infinitesimal supergroup schemes reduces to the existence of certain (conjectured) extension classes for $GL(m|n)$.

Continue to assume that $k$ is a perfect field of characteristic $p > 2$, and that $G$ is an affine algebraic super\-group scheme over $k$. Let $M$ be a finite-dimensional rational $\Gev$-module, and let $r \in \N$. Since $G_1$ acts trivially on $M^{(r)}$, there exists a natural $G$-isomorphism of graded superspaces $\Hbul(G_1,M^{(r)}) \cong \Hbul(G_1,k) \otimes M^{(r)} \cong \Hom_k((M^*)^{(r)},\Hbul(G_1,k))$. Then each $e \in \opH^i(G_1,M^{(r)})$ defines a $G$-equivariant homomorphism of graded superalgebras
\[
e: S(M^*(i))^{(r)} \rightarrow \Hbul(G_1,k).
\]

\begin{conjecture} \label{conjecture:univextclasses}
Let $m,n,r \in \N$. Then there exist cohomology classes
\[
e_r^{m,n}  \in \opH^{2p^{r-1}}(GL(m|n),({\glzero}^*)^{(r)}) \quad \text{and} \quad  c_r^{m,n} \in \opH^{p^r}(GL(m|n),({\glone}^*)^{(r)})
\]
that restrict nontrivially to $GL(m|n)_1$. If $m,n \not\equiv 0 \mod p$, then the homomorphism
\[
e_r^{m,n}|_{GL(m|n)_1}: S({\glzero}^*(2p^{r-1}))^{(r)} \rightarrow \Hbul(GL(m|n)_1,k)
\]
induced by $e_r^{m,n}$ coincides with the composition
\[
S({\glzero}^*(2p^{r-1}))^{(r)} \rightarrow S({\glzero}^*(2))^{(1)} \rightarrow \Hbul(GL(m|n)_1,k),
\]
where the first arrow raises elements to the $p^{r-1}$ power, and the second arrow is the horizontal edge map of the May spectral sequence \eqref{eq:FPG1specseqE0}. Similarly, if $m,n \not\equiv 0 \mod p$, then the composition of the homomorphism
\[
c_r^{m,n}|_{GL(m|n)_1}: S({\glone}^*(p^r))^{(r)} \rightarrow \Hbul(GL(m|n)_1,k)
\]
with the vertical edge map in \eqref{eq:FPG1specseqE0} has image consisting of the subalgebra of $\Lambda_s(\gl(m|n)^*)$ gener\-ated by all $p^r$-th powers of elements in the subalgebra $S({\glone}^*)$ of $\Lambda_s(\gl(m|n)^*)$.
\end{conjecture}

Our main theorem is that the validity of Conjecture \ref{conjecture:univextclasses} implies that the cohomology ring of an infinitesimal supergroup scheme is a finitely-generated algebra. Recall from Lemma \ref{lemma:FSLemma1.3} that if $G$ is an infinitesimal supergroup of height $r$, then there exists a closed embedding $G \subset GL(m|n)_r$ for some $m,n \in \N$. Since $GL(m|n)$ is a closed subsupergroup scheme of $GL(m'|n')$ whenever $m \leq m'$ and $n \leq n'$, we can always choose $m$ and $n$ so that $m,n \not\equiv 0 \mod p$ and $m+n \not\equiv 0 \mod p$.

\begin{theorem} \label{theorem:conjimpliesfingen}
Let $k$ be a perfect field of characteristic $p > 2$, and let $G$ be an infinitesimal $k$-super\-group scheme of height $r$. Choose a closed embedding $G \subset GL(m|n)_r$ as in Lemma \ref{lemma:FSLemma1.3} with $m,n \not\equiv 0 \mod p$ and $m+n \not\equiv 0 \mod p$, and assume that Conjecture \ref{conjecture:univextclasses} holds for the given values of $m,n,r \in \N$. Then $\Hbul(G,k)$ is a finitely-generated algebra, and $\Hbul(G,M)$ is finite over $\Hbul(G,k)$ for each finite-dimensional $G$-supermodule $M$.
\end{theorem}

\begin{proof}
The proof is by induction on $r$. If $r = 1$, then it follows from the naturality of the May spectral sequence and from the discussion at the end of Section \ref{subsection:MayspecseqG1} that the restrictions $e_1^{m,n}|_G$ and $c_1^{m,n}|_G$ define a homomorphism of graded superalgebras
\[
(c_1^{m,n}|_G) \otimes (e_1^{m,n}|_G) : S({\glone}^*(p))^{(1)} \otimes S({\glzero}^*(2))^{(1)} \rightarrow \Hbul(G,k)
\]
over which $\Hbul(G,M)$ is a finite module for each finite-dimensional $G$-supermodule $M$. In particular, $\Hbul(G,k)$ is a finite module over a finitely-generated commutative algebra, hence is itself a finitely-generated algebra.

Now let $r \in \N$ be arbitrary. Let $G' \subset [(GL(m|n)_{\ev})^{(1)}]_{r-1}$ be the image of $G$ under the Frobenius morphism $F_{GL(m|n)}: GL(m|n) \rightarrow (GL(m|n)_{\ev})^{(1)}$, and let $H = G \cap GL(m|n)_1$ be the kernel of the induced surjective homomorphism $F: G \rightarrow G'$. Then there exists an extension of infinitesimal supergroup schemes $1 \rightarrow H \rightarrow G \rightarrow G' \rightarrow 1$. Let $M$ be a finite-dimensional $G$-supermodule, and consider the associated LHS spectral sequence:
\begin{equation} \label{eq:fgLHSspecseq}
E_2^{i,j}(M) = \opH^i(G',\opH^j(H,,M)) \Rightarrow \opH^{i+j}(G,M).
\end{equation}
Set $A = S({\glone}^*(p^r))^{(r)} \otimes S({\glzero}^*(2p^{r-1}))^{(r)}$. The restrictions $e_r^{m,n}|_G$ and $c_r^{m,n}|_G$ define a homomorphism of graded superalgebras $A \rightarrow \Hbul(G,k)$, which composed with the restriction homomorphism $\Hbul(G,k) \rightarrow \Hbul(H,k)^{G'} \subset \Hbul(H,k)$ makes $\Hbul(H,M)$ into a finite $A$-module.

Next, $GL(m|n)_{\ev} \cong GL_m \times GL_n$ is a subgroup scheme of $GL_{m+n}$. Since $GL_{m+n}$ is defined over $\Fp$, we may identify $G'$ with a subgroup scheme of $(GL_{m+n})_{r-1}$. Now recall the cohomology classes
\[
e_i^{(r-1-i)} \in \opH^{2p^{i-1}}(GL_{m+n},\gl_{m+n}^{(r-1)}), \quad 1 \leq i \leq r-1,
\]
exhibited by Friedlander and Suslin \cite[p.\ 215]{Friedlander:1997}. The restrictions $e_i^{(r-1-i)}|_{G'}$ define a homomorphism of graded $k$-algebras
\begin{equation} \label{eq:FSmap} \textstyle
B:= \bigotimes_{i=1}^{r-1} S((\gl_{m+n}^{(r-1)})^*(2p^{i-1})) \rightarrow \Hbul(G',k),
\end{equation}
Composed with the horizontal edge map in \eqref{eq:fgLHSspecseq}, the homomorphism \eqref{eq:FSmap} defines a homomorphism of graded super\-algebras $B \rightarrow \Hbul(G,k)$. The homomorphisms $A \rightarrow \Hbul(G,k) \rightarrow \Hbul(H,k)^{G'}$ and $B \rightarrow \Hbul(G',k) \rightarrow \Hbul(G,k)$ now place us in the setup of \cite[Lemma 1.6]{Friedlander:1997}. Considering $E_2(M)$ as a $B \otimes A$-module via the maps $A \rightarrow \Hbul(H,k)^{G'}$ and $B \rightarrow \Hbul(G',k)$, $E_2(M)$ is a finite $B \otimes A$-module by \cite[Theorem 1.5]{Friedlander:1997}. Then by \cite[Lemma 1.6]{Friedlander:1997}, the homomorphisms $A \rightarrow \Hbul(G,k)$ and $B \rightarrow \Hbul(G',k) \rightarrow \Hbul(G,k)$ make $\Hbul(G,M)$ into a finite $B \otimes A$-module. In particular, $\Hbul(G,M)$ is a finite $\Hbul(G,k)$-module, and $\Hbul(G,k)$ is a finite module over the finitely-generated commutative algebra $B \otimes A$, so we conclude that $\Hbul(G,k)$ is a finitely-generated algebra.
\end{proof}

\begin{remark}
As remarked already in the proof of Theorem \ref{theorem:reduceinfinitesimal}, if $k'$ is a field extension of $k$, then $\Hbul(G,k) \otimes k' \cong \Hbul(G_{k'},k')$ as algebras over $k'$. Then as in Remark \ref{remark:perfectsuperfluous}, the perfectness assumption in Theorem \ref{theorem:conjimpliesfingen} is superflouous.
\end{remark}

\providecommand{\bysame}{\leavevmode\hbox to3em{\hrulefill}\thinspace}


\begin{thebibliography}{MPSW}

\bibitem[Bag]{Bagci:2012a}
I.~Bagci, \href{http://dx.doi.org/10.1007/s10468-012-9395-6} {{\em Cohomology
  and support varieties for restricted {L}ie superalgebras}}, Algebr.
  Represent. Theory (2012).

\bibitem[BO]{Bergh:2008}
P.~A. Bergh and S.~Oppermann,
  \href{http://dx.doi.org/10.1016/j.jalgebra.2008.08.005} {{\em Cohomology of
  twisted tensor products}}, J. Algebra \textbf{320} (2008), no.~8, 3327--3338.

\bibitem[Bou1]{Bourbaki:1998}
N.~Bourbaki, {\em Commutative algebra. {C}hapters 1--7}, Elements of
  Mathematics (Berlin), Springer-Verlag, Berlin, 1998, Translated from the
  French, Reprint of the 1989 English translation.

\bibitem[Bou2]{Bourbaki:2003}
\bysame, {\em Algebra {II}. {C}hapters 4--7}, Elements of Mathematics (Berlin),
  Springer-Verlag, Berlin, 2003, Translated from the 1981 French edition by P.
  M. Cohn and J. Howie, Reprint of the 1990 English edition [Springer, Berlin;
  MR1080964 (91h:00003)].

\bibitem[BK]{Brundan:2003}
J.~Brundan and A.~Kleshchev,
  \href{http://dx.doi.org/10.1016/S0021-8693(02)00620-8} {{\em Modular
  representations of the supergroup {$Q(n)$}. {I}}}, J. Algebra \textbf{260}
  (2003), no.~1, 64--98, Special issue celebrating the 80th birthday of Robert
  Steinberg.

\bibitem[CE]{Cartan:1999}
H.~Cartan and S.~Eilenberg, {\em Homological algebra}, Princeton Landmarks in
  Mathematics, Princeton University Press, Princeton, NJ, 1999, With an
  appendix by David A. Buchsbaum, Reprint of the 1956 original.

\bibitem[EO]{Etingof:2004}
P.~Etingof and V.~Ostrik, {\em Finite tensor categories}, Mosc. Math. J.
  \textbf{4} (2004), no.~3, 627--654, 782--783.

\bibitem[Eve]{Evens:1961}
L.~Evens, {\em The cohomology ring of a finite group}, Trans. Amer. Math. Soc.
  \textbf{101} (1961), 224--239.

\bibitem[FP1]{Friedlander:1986}
E.~M. Friedlander and B.~J. Parshall,
  \href{http://dx.doi.org/10.1007/BF01450727} {{\em Cohomology of infinitesimal
  and discrete groups}}, Math. Ann. \textbf{273} (1986), no.~3, 353--374.

\bibitem[FP2]{Friedlander:1986b}
\bysame, \href{http://dx.doi.org/10.1007/BF01389268} {{\em Support varieties
  for restricted {L}ie algebras}}, Invent. Math. \textbf{86} (1986), no.~3,
  553--562.

\bibitem[FS]{Friedlander:1997}
E.~M. Friedlander and A.~Suslin, \href{http://dx.doi.org/10.1007/s002220050119}
  {{\em Cohomology of finite group schemes over a field}}, Invent. Math.
  \textbf{127} (1997), no.~2, 209--270.

\bibitem[Jan]{Jantzen:2003}
J.~C. Jantzen, {\em Representations of algebraic groups}, second ed.,
  Mathematical Surveys and Monographs, vol. 107, American Mathematical Society,
  Providence, RI, 2003.

\bibitem[KK]{Kang:2000}
S.-J. Kang and J.-H. Kwon, \href{http://dx.doi.org/10.1112/S0024611500012661}
  {{\em Graded {L}ie superalgebras, supertrace formula, and orbit {L}ie
  superalgebras}}, Proc. London Math. Soc. (3) \textbf{81} (2000), no.~3,
  675--724.

\bibitem[Liu]{Liu:2012}
G.~Liu, \href{http://dx.doi.org/10.1016/j.jalgebra.2012.04.010} {{\em Support
  varieties and representation types for basic classical {L}ie superalgebras}},
  J. Algebra \textbf{362} (2012), 157--177.

\bibitem[MPSW]{Mastnak:2010}
M.~Mastnak, J.~Pevtsova, P.~Schauenburg, and S.~Witherspoon,
  \href{http://dx.doi.org/10.1112/plms/pdp030} {{\em Cohomology of
  finite-dimensional pointed {H}opf algebras}}, Proc. Lond. Math. Soc. (3)
  \textbf{100} (2010), no.~2, 377--404.

\bibitem[Mas]{Masuoka:2005}
A.~Masuoka,
  \href{http://dx.doi.org.ezproxy2.lib.depaul.edu/10.1016/j.jpaa.2005.02.010}
  {{\em The fundamental correspondences in super affine groups and super formal
  groups}}, J. Pure Appl. Algebra \textbf{202} (2005), no.~1-3, 284--312.

\bibitem[May]{May:1966}
J.~P. May, {\em The cohomology of restricted {L}ie algebras and of {H}opf
  algebras}, J. Algebra \textbf{3} (1966), 123--146.

\bibitem[McC]{McCleary:2001}
J.~McCleary, {\em A user's guide to spectral sequences}, second ed., Cambridge
  Studies in Advanced Mathematics, vol.~58, Cambridge University Press,
  Cambridge, 2001.

\bibitem[Mon]{Montgomery:1993}
S.~Montgomery, {\em Hopf algebras and their actions on rings}, CBMS Regional
  Conference Series in Mathematics, vol.~82, Published for the Conference Board
  of the Mathematical Sciences, Washington, DC, 1993.

\bibitem[Pri]{Priddy:1970}
S.~B. Priddy, {\em Koszul resolutions}, Trans. Amer. Math. Soc. \textbf{152}
  (1970), 39--60.

\bibitem[SZ]{Scala:2012}
R.~L. Scala and A.~N. Zubkov,
  \href{http://dx.doi.org/10.1007/s10468-011-9269-3} {{\em Donkin--{K}oppinen
  filtration for general linear supergroups}}, Algebr. Represent. Theory
  \textbf{15} (2012), no.~5, 883--899.

\bibitem[SW]{Shu:2008}
B.~Shu and W.~Wang, \href{http://dx.doi.org/10.1112/plms/pdm040} {{\em Modular
  representations of the ortho-symplectic supergroups}}, Proc. Lond. Math. Soc.
  (3) \textbf{96} (2008), no.~1, 251--271.

\bibitem[Ven]{Venkov:1959}
B.~B. Venkov, {\em Cohomology algebras for some classifying spaces}, Dokl.
  Akad. Nauk SSSR \textbf{127} (1959), 943--944.

\bibitem[Wat]{Waterhouse:1979}
W.~C. Waterhouse, {\em Introduction to affine group schemes}, Graduate Texts in
  Mathematics, vol.~66, Springer-Verlag, New York, 1979.

\bibitem[Wes]{Westra:2009}
D.~B. Westra, {\em Superrings and supergroups}, Ph.D. thesis, Universitat Wien,
  October 2009.

\bibitem[Zub]{Zubkov:2009}
A.~N. Zubkov, \href{http://dx.doi.org/10.1007/s00031-009-9055-z} {{\em Affine
  quotients of supergroups}}, Transform. Groups \textbf{14} (2009), no.~3,
  713--745.

\end{thebibliography}
\end{document}